\documentclass{amsart}

\usepackage{amssymb,amscd,amsthm,amsxtra}
\usepackage{dsfont,latexsym}

\usepackage{mathrsfs}

\newtheorem{theorem}{Theorem}[section]
\newtheorem{corollary}[theorem]{Corollary}
\newtheorem{lemma}[theorem]{Lemma}
\newtheorem{prop}[theorem]{Proposition}

\newtheorem{rem}[theorem]{Remark}

\theoremstyle{definition}

\numberwithin{equation}{section}

\newcommand{\R}{{\mathds R}}

\renewcommand{\leq}{\leqslant}
\renewcommand{\le}{\leqslant}
\renewcommand{\geq}{\geqslant}
\renewcommand{\ge}{\geqslant}

\newcommand{\eps}{\varepsilon }
\renewcommand{\epsilon}{\varepsilon }

\newcommand{\Per}{\,{\rm{Per}}}

\newcommand{\no}{ \boldsymbol{\mathcal{E}} }

\def\XXint#1#2#3{{\setbox0=\hbox{$#1{#2#3}{\int}$ }
\vcenter{\hbox{$#2#3$ }}\kern-.6\wd0}}

\newlength{\defbaselineskip}
\newcommand{\setlinespacing}[1]
           {\setlength{\baselineskip}{#1 \defbaselineskip}}

\author[S. Dipierro]{Serena Dipierro}
\address[Serena Dipierro]{Maxwell Institute for Mathematical Sciences and School of Mathematics,
University of Edinburgh, James Clerk Maxwell Building,
Peter Guthrie Tait Road, Edinburgh EH9 3FD, United Kingdom, and
Institut f\"ur Analysis und Numerik,
Otto-von-Guericke Universit\"at Magdeburg,
Universit\"atsplatz 2, 39106 Magdeburg, Germany}
\email{serena.dipierro@ed.ac.uk}

\author[O. Savin]{Ovidiu Savin}
\address[Ovidiu Savin]{Department of Mathematics, Columbia University,
2990 Broadway,
New York NY 10027, USA}
\email{savin@math.columbia.edu}

\author[E. Valdinoci]{Enrico Valdinoci}
\address[Enrico Valdinoci]{Weierstra{\ss} Institut f\"ur Angewandte Analysis und Stochastik, Hausvogteiplatz 11A, 10117 Berlin, Germany,
and
Dipartimento di Matematica, Universit\`a degli studi di Milano,
Via Saldini 50, 20133 Milan, Italy}
\email{enrico.valdinoci@wias-berlin.de}

\begin{document}

\subjclass[2010]{35R11, 35R35, 49Q20, 49Q05.}

\keywords{Fractional perimeter, minimization problem, monotonicity formula, classification of cones.}

\thanks{{\it Acknowledgements}.
The first author has been supported by EPSRC grant EP/K024566/1
``Monotonicity formula methods for nonlinear PDEs'' and Humboldt
Foundation.
The second author has been supported by
NSF grant DMS-1200701.
The third author has been supported by ERC grant 277749 ``EPSILON Elliptic
Pde's and Symmetry of Interfaces and Layers for Odd Nonlinearities''
and PRIN grant 201274FYK7
``Aspetti variazionali e
perturbativi nei problemi differenziali nonlineari''.}

\title[A nonlocal free boundary problem]{A nonlocal free boundary problem}

\begin{abstract}
Given~$s,\sigma\in(0,1)$ and a bounded domain~$\Omega\subset\R^n$, 
we consider the following minimization problem of $s$-Dirichlet
plus $\sigma$-perimeter type
$$ [u]_{ H^s(\R^{2n}\setminus(\Omega^c)^2) }
 + \Per_\sigma\left(\{u>0\},\Omega\right), $$ 
where~$[ \cdot]_{H^s}$ is the fractional Gagliardo
seminorm and $\Per_\sigma$ is the fractional perimeter. 

Among other results,
we prove a monotonicity formula for the minimizers, glueing lemmata,
uniform energy bounds, convergence results,
a regularity theory for the planar cones
and a trivialization result for the flat case.

The classical free boundary problems are limit cases of the one
that we consider in this paper, 
as $s\nearrow1$, $\sigma\nearrow1$ or~$\sigma\searrow0$. 
\end{abstract}

\maketitle

\section{Introduction}\label{8fdsv3ws23456yfd}
In this paper we deal with 
a free boundary problem driven by some nonlocal features.
The nonlocal structures that we consider appear both in the term
that is sometimes related to ``elastic'' atomic interactions and in the so-called ``surface tension''
potential. 

These two features are allowed to have different nonlocal behaviors, namely we parameterize
them with two different fractional parameters~$s, \sigma\in(0,1)$. The classical free boundary
problems appear in the limit of our framework by taking limits
either in~$s$ (as $s\nearrow1$)
or in~$\sigma$ (as~$\sigma\nearrow1$ or~$\sigma\searrow0$), or both.
\medskip

More precisely, we will consider
here the minimization of an energy functional 
that involves a fractional gradient and a nonlocal perimeter. 
Given~$s,\sigma\in(0,1)$ and a bounded and 
Lipschitz domain~$\Omega\subset\R^n$, 
we consider 
\begin{equation}\label{functional}
\mathcal F(u,E):= \iint_{\R^{2n}\setminus (\Omega^c)^2}\frac{|u(x)-u(y)|^2}{|x-y|^{n+2s}}\, dx\,dy + \Per_\sigma\left(E,\Omega\right),  
\end{equation}
where $E$ is the positivity set for $u$ (more precisely,
$u\ge0$ a.e. in $E\cap\Omega$ and
$u\le0$ a.e. in $E^c\cap\Omega$).
As customary, the superscript~$c$ used here above denotes
the complementary set operation, i.e.~$\Omega^c:=\R^n\setminus\Omega$.
The~$\sigma$-fractional perimeter~$\Per_\sigma(E,\Omega)$ of a set~$E$ in~$\Omega$ 
was introduced in~\cite{CRS} and it is defined as 
\begin{equation}\begin{split}\label{sper}
\Per_\sigma(E,\Omega):=& \mathcal L(E\cap\Omega,E^c\cap\Omega) \\
& \quad + \mathcal L(E\cap\Omega, E^c\cap\Omega^c) 
+\mathcal L(E\cap\Omega^c, E^c\cap\Omega),  
\end{split}\end{equation}
where the interaction~$\mathcal L$ is the following
\begin{equation}\label{idid}
\mathcal L(A,B):=\iint_{A\times B} \frac{dx\,dy}{|x-y|^{n+\sigma}} \end{equation}
for any disjoint, measurable sets~$A$ and~$B$.

The nonlocal perimeter converges to the classical perimeter as~$\sigma\nearrow 1$ 
and to the Lebesgue measure of~$E$ as~$\sigma\searrow 0$ (up to multiplicative constants), see~\cite{BBM, DAVILA, CV1, ADM, MS, DFPV} for precise statements.

In~\cite{CSV} the authors consider a minimization problem that corresponds 
to~\eqref{functional} in the case~$s=1$, namely 
\begin{equation}\label{caso1}
\int_{\Omega}|\nabla u(x)|^2\,dx +\Per_\sigma\left(\{u>0\},\Omega\right). 
\end{equation} 
They use blow-up analysis to obtain regularity results for minimizers and for 
the free boundaries. 
When~$\sigma\searrow 0$, the functional in \eqref{caso1} reduces to a classical free boundary problem
related to fluid dynamics and that has been extensively studied in the literature
after the pioneer work in \cite{alt, alt2}.
On the other hand, when~$\sigma\nearrow 1$, the energy in~\eqref{caso1} reduces to the problem studied
in \cite{salsa}, where the energy functional is 
a competition between the classical Dirichlet form and the perimeter of the interface.

Thus, the energy functional in \eqref{functional} that we study here
follows in the energy framework, in which both the quadratic form
and the interface energy appearing in the functional are of nonlocal type.
The problem has also the relevant feature
of allowing different types of nonlocal behaviors
in the two components of the energy, which may
have concrete applications, since the two terms may come
from different types of long range interactions.

For other recent results on fractional free boundary problems see, 
for instance, \cite{CafSire, DeRo, DeSa, Allen}.
\medskip

The variational notion of minimizers that we consider in this
paper is the following. Fixed $E_0\subseteq\R^n$ with locally finite $\sigma$-perimeter and
$\varphi\in H^s_{\rm loc}(\R^n)$ with
$\varphi\ge0$ a.e. in $E_0$ and
$\varphi\le0$ a.e. in $E_0^c$, we say that $(u,E)$
is a minimizing pair (in the domain $\Omega$ with
external datum $\varphi$) if 
${\mathcal{F}}(u,E)$ attains the minimal possible
value among all the
functions $v$
such that 
\begin{equation}\label{Def 1}
{\mbox{$v-\varphi\in H^s(\R^n)$ with $v=\varphi$
a.e. in $\Omega^c$}} 
\end{equation}
and all the measurable sets $F\subseteq\R^n$
with $F\setminus\Omega=E_0\setminus\Omega$ and such that
\begin{equation}\label{Def 2}
{\mbox{$v\ge0$ a.e. in $F\cap\Omega$ and
$v\le0$ a.e. in $F^c\cap\Omega$.}}
\end{equation}
In spite of its technical flavor, the definition above
can be intuitively understood by saying, roughly speaking,
that the function $u$ minimizes the energy functional
among all the competitors $v$ that coincide with $u$ outside
the domain $\Omega$ (the technicality is to formally state
that $F$ is the positivity set of $v$ for which we need
to compute the $\sigma$-perimeter).

The existence of minimizing pairs will be guaranteed by
the forthcoming Lemma~\ref{MIMI}
and it follows from the direct method joined with a suitable
fractional compact embedding.\medskip

We will show that the energy of a minimizing
pair can be bounded uniformly: more precisely,
if~$(u,E)$ is a minimizing pair in a given ball,
then the energy in a smaller ball is bounded, according
to the next result:

\begin{theorem}[Uniform energy estimates]\label{ENERGY}
Let~$(u,E)$ be a mininimizing pair in~$B_2$. Then
$$ \iint_{\R^{2n}\setminus(B_1^c)^2} \frac{\big| u(x)-u(y)\big|^2
}{|x-y|^{n+2s}}\,dx\,dy +
\Per_\sigma(E,B_1) \le C\,\left(1+
\int_{\R^n}
\frac{|u(y)|^2}{1+|y|^{n+2s}}\,dy\right),$$
for some~$C>0$ only depending on~$n$ and~$s$.
\end{theorem}

The proof of Theorem~\ref{ENERGY}
relies on appropriate gluing results that are interesting in themselves
(roughly speaking, they allow us to change an admissible pair
outside a given domain, by controlling the energy produced by
the interpolation).
\medskip

For this, it is useful to consider an associated extension problem. 
That is, we set~$\R^{n+1}_+:=\{(x,z)\in\R^n\times\R {\mbox{ s.t. }}z>0\}$,  
and, given a function~$u:\R^n\rightarrow\R$, we associate a function~$\overline u$ defined in~$\R^{n+1}_+$ as 
\begin{equation}\label{ext}
\overline u(\cdot,z)=u*P_s(\cdot,z), {\mbox{ where }} 
P_s(x,z):= c_{n,s}\frac{z^{2s}}{(|x|^2+z^2)^{(n+2s)/2}}.  
\end{equation}
Here~$c_{n,s}$ is a normalizing constant depending on~$n$ and~$s$. 

Moreover, given a measurable set~$E\subset\R^n$ 
we associate a function~$U$ defined in~$\R^{n+1}_+$ as 
\begin{equation}\label{extE}
U(\cdot,z)=(\chi_E-\chi_{E^c})*P_\sigma(\cdot,z), {\mbox{ where }} 
P_\sigma(x,z):= c_{n,\sigma}\frac{z^\sigma}{(|x|^2+z^2)^{(n+\sigma)/2}},  
\end{equation}
and~$c_{n,\sigma}$ is a normalizing constant depending on~$n$ and~$\sigma$
(these constants are only needed to normalize the integral of $P_s$ and $P_\sigma$).

We will denote the extended variable as~$X:=(x,z)\in\R^{n+1}_+$, 
where~$x\in\R^n$ and~$z>0$. Moreover, $B_r:=\{|x|<r\}$ is the ball of radius $r$ 
in $\R^n$ and~$\mathcal B^+_r:=\{|X|<r\}$ is the ball of radius~$r$ in~$\R^{n+1}_+$.

The role played by the extensions defined in~\eqref{ext}
and~\eqref{extE} is to reduce the original nonlocal problem
to a local problem in an extended space (this will
be made precise in Proposition~\ref{char}). Roughly
speaking, the extended functions minimize
a weighted Dirichlet energy for a given trace, whose
weighted Neumann condition on the trace reproduces the original
nonlocal functional.
\medskip

We will study in detail
the extended problem in Section~\ref{sec4},
where we will also find equivalent minimizing conditions between the original functional in~$(u,E)$
and an extended
functional in~$(\overline{u},U)$ (see in particular Proposition~\ref{char}).
Here we just mention that the notion of minimization in the extended variables in a domain~$\Omega\subset\R^{n+1}$
requires not only that the competing functions agree near~$\partial\Omega$, but also a consistency condition on
the trace~$\{z=0\}$, where the functions reduce to characteristic functions of sets.
Namely, we say that~$(\overline{u},U)$ is a minimizing pair for the extended problem
in~$\Omega\subset\R^{n+1}$ if
\begin{equation*}\begin{split}
& \int_{\Omega_+}z^{1-2s}|\nabla\overline u|^2\, dX+c_{n,s,\sigma}
\int_{\Omega_+}z^{1-\sigma}|\nabla U|^2\, dX \\
&\qquad \le \int_{\Omega_+}z^{1-2s}|\nabla\overline v|^2\, dX+c_{n,s,\sigma}\int_{\Omega_+}z^{1-\sigma}|\nabla V|^2\, dX 
\end{split}\end{equation*}
for every  functions~$\overline v$ and~$V$
that satisfy the following conditions:
\begin{itemize}
\item[i)] $V=U$ in a neighborhood of~$\partial\Omega$,
\item[ii)] the trace of~$V$ on~$\{z=0\}$ is~$\chi_F-\chi_{F^c}$ 
for some set~$F\subset\R^n$, 
\item[iii)] $\overline v=\overline u$ 
in a neighborhood of~$\partial\Omega$, 
and~$\overline v\big|_{\{z=0\}}\ge 0$ a.e. in~$F$,
$\overline v\big|_{\{z=0\}}\le 0$ a.e. in~$F^c$. 
\end{itemize}
\medskip

In this setting, we can use glueing techniques to prove convergence of
minimizing pairs of the extended problem, as stated in the following result:

\begin{theorem}[Convergence of minimizers] \label{prop:conv_ext}
Let~$(\overline{u}_m,U_m)$ be a sequence of minimizing pairs for the extended problem
in~$\mathcal{B}_2^+$. Suppose that~$\overline u_m$
is the extension of~$u_m$ as in~\eqref{ext}, and 
\begin{equation}\label{se00}
{u}_m\rightarrow {u} \ {\mbox{ in }}L^\infty(B_2),
\quad
\overline{u}_m\rightarrow\overline{u} \ {\mbox{ in }}L^\infty(\mathcal{B}^+_2) \
{\mbox{ and }} \ U_m\rightarrow U \ {\mbox{ in }}L^2_{\sigma/2}(\mathcal{B}^+_2),
\end{equation}
as~$m\rightarrow +\infty$, for some couple~$(\overline{u},U)$, with~$\overline{u}$
continuous in~$\overline{\R^{n+1}_+}$.

Then $(\overline{u},U)$ is a minimizing pair in~$\mathcal{B}^+_{1/2}$.

Moreover
\begin{equation}\label{TESI}\begin{split}
&\lim_{m\to+\infty}\int_{\mathcal{B}^+_{1}}z^{1-2s}|\nabla\overline{u}_m|^2\,dX
= \int_{\mathcal{B}^+_{1}}z^{1-2s}|\nabla\overline{u}|^2\,dX\\
{\mbox{and }}\qquad&
\lim_{m\to+\infty}\int_{\mathcal{B}^+_{1}}z^{1-\sigma}|\nabla U_m|^2\,dX
=\int_{\mathcal{B}^+_{1}}z^{1-\sigma}|\nabla U|^2\,dX.\end{split}
\end{equation}
\end{theorem}

A particularly important case of convergence is
given by the blow-up limit.
This is also related
to the study of the minimizing pairs
that possess suitable homogeneity properties, and in particular
the ones induced by the natural scaling of the functional. For this,
we say that a minimizing pair $(u,E)$ is a minimizing cone
if $u$ is homogeneous of degree $s-\frac\sigma2$
and $E$ is a cone (i.e., for any $t>0$, $tx\in E$
if and only if $x\in E$).


In this framework, we exploit
Theorems~\ref{ENERGY} and~\ref{prop:conv_ext}, combined with
some arguments in~\cite{CRS}, and we obtain the following relation between blow-up
limits and minimizing cones:

\begin{theorem}[Blow-up cones]\label{CONI 0}
Let~$s>\sigma/2$ and~$(u,E)$ be a minimizing pair in~$B_1$, with~$0\in\partial E$.
For any~$r>0$ let
\begin{equation}\label{rescaled} u_r(x):=r^{\frac\sigma2 -s }u(rx) \ {\mbox{ and }} \
E_r:=\frac1r E.\end{equation}
Assume that~$u\in C^{s-\frac\sigma2}(\R^n)$.
Then there exist a minimizing cone~$(u_0, E_0)$
and a sequence~$r_k\to0$ such that~$u_{r_k}\to u_0$
in~$L^\infty_{\rm loc}(\R^n)$ and~$E_{r_k}\to E_0$ in~$L^1_{\rm loc}(\R^n)$.
\end{theorem}

We remark that the rescaling in~\eqref{rescaled}
is the one induced by the energy, since 
if~$(u,E)$ is a minimizing pair
for~$\mathcal F$ in~$\Omega$, then~$(u_r,E_r)$ is a minimizing pair
for~$\mathcal F$ in~$\frac1r \Omega$.
Moreover, the exponent~$\frac{\sigma}{2}-s$ in~\eqref{rescaled}
corresponds to the one obtained in~\cite{CSV} in the case~$s=1$. 
\medskip

A complete
classification of the minimal cones in dimension $2$ holds true,
according to the following result:

\begin{theorem}[Classification of minimizing cones in the plane]\label{prop:reg}
Let~$n=2$
and let $(u,E)$ be a minimizing pair in any domain.

Assume that $u$ is continuous and homogeneous of degree $s-\frac\sigma2$
and that $E$ is the union of finitely many closed conical
sectors, with both $E$ and $E^c$ nonempty.

Then $E$ is a halfplane.
\end{theorem}

The proof of Theorem \ref{prop:reg} uses a second order domain variation,
in the spirit of the technique introduced in \cite{SV1, SV2}
(since the main ideas of the proof are the same, but
some technical differences arise here due to the presence
of minimizing pairs rather than functions,
we give the full details of the proof in Appendix~\ref{temp}).
\medskip

We remark that the continuity of the minimizers
is an interesting open problem. In the one-phase case
(that is, when the datum has a sign to start with), density
estimates and continuity properties have been recently proved
in~\cite{SERENA}.

In general, the natural scaling of the problem
seems to be of degree~$s-\frac\sigma2$. Nevertheless,
even in the classical case, the regularity of
the minimizer can beat such exponent: for instance,
the minimizers in~\cite{salsa} (which correspond to
the case~$s=\sigma=1$) are better than~$C^{1/2}$
and are indeed Lipschitz (see
in particular Theorem~3.1 and 4.1 in~\cite{salsa}).
In this sense,
the natural scaling of the problem
does not exhaust the complexity of the minimizers.

We think it is an interesting problem to detect
the optimal regularity of the minimizers in our problem
and to decide whether or not the blow-up sequences
in different settings approach the trivial function.\medskip

In our framework, a special scaling feature occurs
when~$s=\sigma/2$: in this case the
Gagliardo
seminorm and the fractional perimeter
have exactly the same dimensional properties and one
may think that, under this circumstance, a minimizing pair
reduces to the characteristic function of a set,
consistently with the fact that the blow-up limits
are homogeneous of degree zero. But it turns out that this is not
the case, as next observation points out:

\begin{rem}\label{R}
Let~$s\in(0,1/2)$ and $\sigma=2s$.
Fix a set~$E_0\subseteq\R^n$ with locally finite $\sigma$-perimeter,
and let~$u_0:=\chi_{E_0}-\chi_{E_0^c}$.

Let~$(u,E)$ be a minimizing pair in~$B_1$ with respect to the datum~$(u_0,E_0)$
outside~$B_1$.

Then, it is not true that~$u=\chi_E-\chi_{E^c}$ (unless either~$E=\R^n$ or~$E=\varnothing$).
\end{rem}

We also observe that the problem we consider
may develop plateaus, i.e. fattening of the zero level set
of minimizers. For instance, we
point out that, in dimension~$1$ and for~$s=1/2$,
it is not possible that~$\{u=0\}$ is just (locally) a single point,
unless~$u$ is $(1/2)$-harmonic across the free boundary, as shown
by the following simple example:

\begin{rem}\label{Remo}
Let~$n=1$, $s=1/2$ and~$(u,E)$ be a minimizing pair in~$(-1,1)$,
with~$u\in C([-1,1])\cap H^{1/2}(\R)$.

Then either~$(-\Delta)^{1/2} u=0$ in~$(-1,1)$
or the set~$\{u=0\}\cap(-1,1)$ contains infinitely many points.
\end{rem}

We recall that the fattening of the zero level set
of the minimizers also occur in other free boundary problems,
see in particular Theorem~9.1 in~\cite{Allen}.
\medskip

In the subsequent section,
we present
some additional results that are auxiliary to the ones presented till now, but that
we believe may have independent interest.
A detailed plan about the organization of the paper will
then be presented at the end of Section~\ref{II}.

\section{Additional results}\label{II}

Here we collect some further results that complete 
the picture described in Section~\ref{8fdsv3ws23456yfd}
and that possess some independent interest.
First of all, we obtain a Weiss-type monotonicity formula for minimizing 
pairs~$(u,E)$ (see \cite{WEISS}
for the original monotonicity formula in the setting of classical free boundaries):

\begin{theorem}[Monotonicity formula]\label{TH:mon}
Let~$(u,E)$ be a minimizing pair in~$B_\rho$, and let $\overline{u}$ and $U$ 
be as in \eqref{ext} and \eqref{extE}. Then
\begin{equation}\begin{split}\label{Phi}
\Phi_u(r):=&\, r^{\sigma-n}\left(\int_{\mathcal B^+_r}z^{1-2s}
|\nabla\overline u|^2\,dX+ c_{n,s,\sigma}\int_{\mathcal B^+_r}z^{1-\sigma}|\nabla U|^2\,dX\right) \\
&\qquad -\left(s-\frac{\sigma}{2}\right) r^{\sigma-n-1}\int_{\partial\mathcal B^+_r}z^{1-2s}\, \overline u^2\,d\mathcal H^n
\end{split}\end{equation}
is increasing in~$r\in(0,\rho)$. 

Moreover, $\Phi_u$ is constant if and only if~$\overline u$ is homogeneous 
of degree~$s-\frac{\sigma}{2}$ and~$U$ is homogeneous of degree~$0$. 
\end{theorem}

We also show that the minimizing pairs enjoy
a dimensional reduction property. Namely, if a minimizing pair
is trivial in a given direction, then it can be sliced to a minimizing pair
in one dimension less. Conversely, given a minimizing pair in~$\R^n$,
one obtains a minimizing pair in~$\R^{n+1}$ by adding the trivial action of
one dimension more. The formal statement of this property sounds as follows:

\begin{theorem}[Dimensional reduction]\label{DR}
The pair $(u,E)$ is minimizing in any domain of $\R^n$ if and only if the pair
$(u^\star, E^\star)$ is minimizing in any domain of $\R^{n+1}$,
where $u^\star(x,x_{n+1}):=u(x)$ and $E^\star:=E\times\R$.
\end{theorem}

In the study
of the local free boundary problems
and minimal surfaces, homogeneous solutions and minimizing cones
are often explicit and they constitute the easiest
possible nontrivial example. In our case, the existence
of nontrivial minimizing cones is not obvious, since
the example of the halfspace trivializes, according
to the following result:

\begin{theorem}[Trivialization of halfspaces]\label{TRI}
Let $(u,E)$ be a minimizing cone, with
$u\in C(\R^n)$ and $[u]_{C^\gamma(\R^n)}<+\infty$, for some~$\gamma\in(0,1]$.

If $E$ is contained in a halfspace then $u\le0$.

Similarly, if $E^c$ is contained in a halfspace then $u\ge0$.

In particular,
if $E$ is a halfspace
then $u$ vanishes identically.\end{theorem}

We observe that, as a consequence of
Theorems~\ref{CONI 0} and~\ref{TRI},
the blow-up of minimizers at regular points
of the free boundary produces the zero function.
\medskip

The proof of Theorem \ref{TRI} relies on a suitable nonlocal
maximum principle
in unbounded domains that we explicitly state as follows:

\begin{theorem}[Nonlocal
maximum principle in a halfspace]\label{MAX PLE}
Let $D$ be an open set of $\R^n$, contained in the halfspace~$\{ x_n>0\}$.
Let $v\in L^\infty(D)\cap C^2(D)$ be
continuous on $\overline{D}$ and such that
\begin{equation}\label{eq-0} \left\{\begin{matrix}
(-\Delta)^s v \le 0 & {\mbox{ in }} D, \\
v\le 0 & {\mbox{ in }} D^c.
\end{matrix}\right.\end{equation}
Then $v\le 0$ in $D$.
\end{theorem}

The results presented in this paper
require delicate and conceptual modifications
with respect to other results already present in the literature.
The difficulties often come from the strong nonlocal
features of the problem and on the lack of explicit barriers
and test functions. Also, the glueing methods
in this case present additional complications,
since matching the functions on the trace and on the extension
may cause errors which propagate in the whole of the space.\medskip

The rest of the paper will present all the material necessary to the proofs
of the results presented here above and in Section~\ref{8fdsv3ws23456yfd}.
More precisely, in Section~\ref{prelim}
we show some preliminary properties of the
minimizing pairs.

In Section~\ref{sec4} we deal with an equivalent minimization problem 
on the extended variables and we use it to prove Theorem~\ref{TH:mon}. 
The proof of the dimensional reduction of
Theorem~\ref{DR}
is contained in Section~\ref{DRA}.

Section~\ref{GLUE} contains some glueing results that are
interesting in themselves and that are used to prove
the uniform energy estimates of Theorem~\ref{ENERGY}, which are
contained in Section~\ref{ENERGY:sec}, and the convergence result
of Theorem~\ref{prop:conv_ext}, which is contained in Section~\ref{prop:conv_ext:sec}.
The convergence to blow-up cones, as detailed in
Theorem~\ref{CONI 0}, is proved in Section~\ref{CONI 0:sec}.
Then, in Section~\ref{sec6} we prove Theorems~\ref{TRI} and~\ref{MAX PLE}.
Finally, the proofs of Remarks~\ref{R} and~\ref{Remo}
are contained in Sections~\ref{RR} and~\ref{ydd88syhh}, respectively.

\section{Preliminaries}\label{prelim}

Here we discuss some basic properties of the minimizing pairs,
such as existence and $s$-harmonicity.

\begin{lemma}\label{MIMI}
The minimizing pair exists.
\end{lemma}

\begin{proof}
Let $(u_j,E_j)$ be a minimizing sequence.
By compactness (see e.g. Theorem 7.1 in \cite{DPV12}) 
we infer that, up to subsequences, $u_j$ converges
to some $u$ and $\chi_{E_j}$ converges to some $\chi_E$
in $L^2(\Omega)$ and a.e. in $\Omega$. In fact, since
$u_j$ and $\chi_{E_j}$ are fixed outside $\Omega$,
the convergence holds a.e. in $\R^n$ and so, by Fatou Lemma,
${\mathcal{F}}(u,E)$ attains the desired minimum of the energy.
It remains to show that this pair is admissible, i.e.~$u\ge0$ 
a.e. in~$E\cap\Omega$ and~$u\le0$ a.e. in~$E^c\cap\Omega$. 
Indeed, let~$x\in E\cap\Omega$. Up to a set of
null measure we have that~$\chi_{E_j}(x)\to\chi_{E}(x)=1$.
Since the image of the characteristic function is a discrete set,
it follows that~$\chi_{E_j}(x)=1$ for large~$j$, hence~$u_j(x)\ge0$
and therefore~$u(x)\ge0$.
Similarly, one can prove that~$u\le0$ a.e. in~$E^c\cap \Omega$.\end{proof}

\begin{lemma}\label{SH}
Let $(u,E)$ be a minimizing pair. If $\Omega$ is an open subset of either $\{u>0\}$ or $\{u<0\}$, 
then $(-\Delta)^s u(x)=0$ for any $x\in\Omega$. 
In particular, if $u\in C(\R^n)$, 
then $(-\Delta)^s u(x)=0$ for any $x\in \{u>0\} \cup\{ u<0\}$.
\end{lemma}

\begin{proof} Fix $x_o\in\Omega\subset\{u>0\}$ (the case
$\Omega\subset\{u<0\}$ is similar). Then there exists $r>0$
such that $B_r(x_o)\Subset\Omega$ and therefore
$$ \mu:= \min_{B_r(x_o)} u >0.$$
Let $\eta\in C^\infty_0 (B_r(x_o))$ and $\eps\in\R$ with
$|\eps|<\mu \|\eta\|_{L^\infty(\R^n)}^{-1}$.
We define $u_\eps:=u+\eps\eta$. Notice that $u_\eps=u$ outside $B_r(x_o)$
and $u_\eps\ge \mu-|\eps|\,\|\eta\|_{L^\infty(\R^n)}>0$ in $B_r(x_o)$.

Therefore $u_\eps\ge0$ in $E$ and $u_\eps\le0$ in $E^c$, since the same
holds for $u$. This says that $(u_\eps,E)$ is an admissible competitor,
therefore
$$ 0\le {\mathcal{F}}(u_\eps,E)-{\mathcal{F}}(u,E)
=2\eps \iint_{\R^{2n}}\frac{
\big(u(x)-u(y)\big)\big(\eta(x)-\eta(y)\big)}{
|x-y|^{n+2s}}\, dx\,dy+o(\eps).$$
Dividing by $\eps$ and taking the limit we conclude that~$(-\Delta)^s u(x_o)=0$ in the weak sense,
and thus in the classical sense (see e.g. \cite{SV-w}).
\end{proof}

We prove also the following comparison principle. 

\begin{lemma}\label{CP00}
Let $(u,E)$ be a minimizing pair and let~$A\in\R$. 
If~$\varphi\ge A$ (respectively~$\varphi\le A$), 
then~$u\ge A$ (respectively~$u\le A$).  
\end{lemma}

\begin{proof}
We prove the case~$\varphi\ge A$, the case~$\varphi\le A$ is analogous.

Notice that if~$(v,E)$ is an admissible competitor against~$(u,E)$, then 
we have 
\begin{equation}\label{9COMP1}
0 \le \mathcal F(v,E)-\mathcal F(u,E) = 
\iint_{\R^{2n}}\frac{|v(x)-v(y)|^2-|u(x)-u(y)|^2}{|x-y|^{n+2s}}\,dx\,dy.\end{equation}
Suppose first that~$A=0$. 
We denote by~$\tilde u:=\max\{u,0\}$ and we notice that~$\tilde u=u\ge0$ in~$E$ 
and~$\tilde u=0$ in~$E^c$. Therefore~$(\tilde u,E)$ is an admissible competitor, 
and so~\eqref{9COMP1} holds with~$v:=\tilde u$, that is  
\begin{equation}\label{COMP10}
0 \le \iint_{\R^{2n}}\frac{|\tilde u(x)-\tilde u(y)|^2-|u(x)-u(y)|^2}{|x-y|^{n+2s}}\,dx\,dy.
\end{equation}
On the other hand, we have that~$|\tilde u(x)-\tilde u(y)|^2\le|u(x)-u(y)|^2$.
This, together with~\eqref{COMP10}, implies that
$$ \iint_{\R^{2n}}\frac{|\tilde u(x)-\tilde u(y)|^2-|u(x)-u(y)|^2}{|x-y|^{n+2s}}\,dx\,dy=0,$$
which gives that 
\begin{equation}\begin{split}\label{mathZ}
&{\mbox{there exists a set~$\mathcal Z\subset\R^{2n}$ of measure zero}}\\
&{\mbox{such that~$|\tilde u(x)-\tilde u(y)|^2=|u(x)-u(y)|^2$ for 
every~$(x,y)\in\R^{2n}\setminus\mathcal Z$.}} 
\end{split}\end{equation}

Now, we claim that 
\begin{equation}\label{mathV}
\begin{split}
&{\mbox{there exist $\bar y\in\R^n$ and ${\mathcal{V}}\subset\R^n$
such that}}\\
&{\mbox{$|{\mathcal{V}}|=0$, and}}\\ &{\mbox{$(x,\bar y)\in\R^{2n}\setminus{\mathcal{Z}}$
for any $x\in \R^n\setminus {\mathcal{V}}$.}}\end{split}
\end{equation}
Indeed, for any~$y\in\R^n$, we define
$$ b(y):=\int_{\R^n} \chi_{\mathcal{Z}}(x,y)\,dx.$$
Then, by Fubini's theorem,~$b$ is a nonnegative and measurable function, and
$$ \int_{\R^n} b(y)\,dy=\iint_{\R^{2n}} \chi_{\mathcal{Z}}(x,y)\,dx\,dy=|{\mathcal{Z}}|=0.$$
Therefore,~$b(y)=0$ for a.e.~$y\in\R^n$.
In particular, we can fix~$\bar y\in\R^n$ such that~$b(\bar y)=0$, that is
$$ \int_{\R^n} \chi_{\mathcal{Z}}(x,\bar y)\,dx=0.$$
This implies that~$\chi_{\mathcal{Z}}(x,\bar y)=0$ for a.e. $x\in\R^n$
(say, for every $x\in\R^n\setminus{\mathcal{V}}$,
for a suitable~${\mathcal{V}}\subset\R^n$ of zero measure). This concludes
the proof of~\eqref{mathV}.

Having established~\eqref{mathV}, we use it together with~\eqref{mathZ} 
to deduce that~$|\tilde u(x)-\tilde u(\bar y)|^2=|u(x)-u(\bar y)|^2$ for 
every~$x\in\R^n\setminus\mathcal V$, which means 
that~$\tilde u(x)-\tilde u(\bar y)=\pm(u(x)-u(\bar y))$ for a.e.~$x\in\R^n$.  
Setting~$c_\pm:=\tilde{u}(\bar y) \mp u(\bar y)$, 
we obtain that~$\tilde{u}(x)=\pm u(x)+c_\pm$ for a.e.~$x\in\R^n$. 
Since~$\tilde{u}=u=\varphi$ outside~$\Omega$, 
we get that $u=\tilde u$ a.e. in~$\R^n$, which implies that~$u\ge0$. 
This concludes the proof in the case~$A=0$.

Now suppose that~$A<0$. In this case we define~$\widehat u:=\max\{u,A\}$. 
It is not difficult to see that 
\begin{equation}\label{9COMP2}
|\widehat u(x)-\widehat u(y)|^2\le|u(x)-u(y)|^2. 
\end{equation}
Moreover~$\widehat u=u\ge0$ in~$E$ and~$\widehat u\le0$ in~$E^c$, 
which says that the couple~$(\widehat u,E)$ is an admissible competitor 
against~$(u,E)$. 
Therefore, from~\eqref{9COMP1} with~$v:=\widehat u$ 
and~\eqref{9COMP2} we obtain that 
$$ \iint_{\R^{2n}}\frac{|\widehat u(x)-\widehat u(y)|^2-|u(x)-u(y)|^2}{|x-y|^{n+2s}}\,dx\,dy=0.$$
Now, we proceed as in the case~$A=0$ and we deduce 
that~$u=\widehat u$ a.e. in~$\R^n$, which implies that~$u\ge A$ 
and concludes the proof in the case~$A<0$. 

Finally, we deal with the case~$A>0$. 
For this, given a function~$v:\R^n\rightarrow\R$ we use the notation
\begin{equation*}
\no(v):=\sqrt{
\iint_{\R^{2n}\setminus (\Omega^c)^2} \frac{|v(x)-v(y)|^2}{
|x-y|^{n+2s}}\,dx\,dy}.
\end{equation*}
We denote by~$u^\star$ the unique minimizer of the Dirichlet energy 
with datum~$\varphi$, that is 
$$ \no^2(u^\star)=\min_{v\in\mathcal H}\no^2(v),$$
where~$\mathcal H:=\{v\in H^s(\R^n){\mbox{ s.t. }} v=\varphi {\mbox{ a.e. in }}\Omega^c\}$.
We observe that the fact that~$\varphi\ge A$ implies that
\begin{equation}\label{ustar A}
u^\star\ge A>0
\end{equation}
(see Lemma 2.4 in \cite{DipVal}). 
This means that the positivity set of~$u^\star$ is the whole~$\R^n$. 
Therefore, we have that 
\begin{equation}
\no^2(u^\star)\le\no^2(u) \label{no} \;
{\mbox{ and }} \; \Per_\sigma(\R^n,\Omega)=0\le\Per_\sigma(E,\Omega).
\end{equation}
Now, we claim that 
\begin{equation}\label{per 0}
\Per_\sigma(E,\Omega)=0. 
\end{equation}
Indeed, suppose by contradiction that~$\Per_\sigma(E,\Omega)>0$. 
Then, 
$$ \Per_\sigma(E,\Omega)>\Per_\sigma(\R^n,\Omega),$$
and so, using this and~\eqref{no}, we have 
$$ \mathcal F(u^\star,\R^n)=\no^2(u^\star)+\Per_\sigma(\R^n,\Omega)<\no^2(u)+\Per_\sigma(E,\Omega)=\mathcal F(u,E),$$ 
which contradicts the minimality of~$(u,E)$. 
This shows~\eqref{per 0}. 

{F}rom~\eqref{no}, \eqref{per 0} and the minimality of~$(u,E)$, we obtain 
$$ \mathcal F(u^\star,\R^n)=\no^2(u^\star)\le\no^2(u)=\mathcal F(u,E)\le\mathcal F(u^\star,\R^n), $$ 
which implies that~$\no^2(u^\star)=\no^2(u)$. 
Since~$u^\star$ is the unique minimizer of the Dirichlet energy 
with datum~$\varphi$, this, in turn, gives that~$u=u^\star$ a.e. in~$\R^n$. 
Recalling~\eqref{ustar A} we conclude the proof in the case~$A>0$. 
\end{proof}

\section{An equivalent extended problem,
a monotonicity formula and proof of Theorem~\ref{TH:mon}}\label{sec4}

In this section, we discuss a problem on the extended variables
that is equivalent to our original minimization problem
(this can be seen as a generalization of the extension problem of \cite{CAF-SIL}).

For this, for any bounded Lipschitz domain~$\Omega\subset\R^{n+1}$
we set~$\Omega_0:=\Omega\cap \{z=0\}$ and~$\Omega_+:=\Omega\cap \{z>0\}$.
Hence, recalling~\eqref{ext} and~\eqref{extE}, we have the following characterization of minimizing pairs~$(u,E)$. 

\begin{prop}\label{char}
The pair~$(u,E)$ is minimizing in~$B_{r'}$ for every~$r'\in(0,r)$
if and only if
\begin{equation}\label{1983}\begin{split}
& \int_{\Omega_+}z^{1-2s}|\nabla\overline u|^2\, dX+c_{n,s,\sigma}
\int_{\Omega_+}z^{1-\sigma}|\nabla U|^2\, dX \\
&\qquad \le \int_{\Omega_+}z^{1-2s}|\nabla\overline v|^2\, dX+c_{n,s,\sigma}\int_{\Omega_+}z^{1-\sigma}|\nabla V|^2\, dX 
\end{split}\end{equation}
for every bounded, Lipschitz domain~$\Omega\subset\R^{n+1}$ with~$\Omega_0
\subset B_r$, and every functions~$\overline v$ and~$V$
that satisfy the following conditions:
\begin{itemize}
\item[i)] $V=U$ in a neighborhood of~$\partial\Omega$,
\item[ii)] the trace of~$V$ on~$\{z=0\}$ is~$\chi_F-\chi_{F^c}$ 
for some set~$F\subset\R^n$, 
\item[iii)] $\overline v=\overline u$ 
in a neighborhood of~$\partial\Omega$, 
and~$\overline v\big|_{\{z=0\}}\ge 0$ a.e. in~$F$,
$\overline v\big|_{\{z=0\}}\le 0$ a.e. in~$F^c$. 
\end{itemize}
\end{prop}

\begin{proof} We show that admissible competitors in the extensions
come from admissible competitors on the trace: more precisely,
we show that
\begin{equation}\label{1LL:1}
\begin{split}
&{\mbox{if~$\Omega_*$ is compactly contained
in~$\Omega$,}}\\&{\mbox{$\overline v(X)=\overline u(X)$
for any~$X\in\Omega_*^c$
and $v(x):=\overline v(x,0)$}}\\&{\mbox{with }}
\int_{\Omega_+} z^{1-2s} |\nabla \overline v|^2\,dX <+\infty,\\
&{\mbox{then }}
\iint_{\R^{2n}\setminus (\Omega_{0})^2}
\frac{|v(x)-v(y)|^2}{|x-y|^{n+2s}}\,dx\,dy <+\infty.\end{split}\end{equation}
To prove this, we take an intermediate domain~$\Omega_\sharp$,
which is
compactly contained
in~$\Omega$ and which compactly contains~$\Omega_*$.
So we have that~$\Omega_{\sharp0}:=\Omega_\sharp\cap\{z=0\}$
is compactly contained in~$\Omega_0$
and compactly contains~$\Omega_{*0}
:=\Omega_*\cap\{z=0\}$. Also,
there exists~$\zeta
\in C^\infty_0(\Omega)$ such that~$\zeta(X)=1$ for any~$X
\in \Omega_\sharp$. We also define~$w:=\zeta\, \overline v$.
Then,
\begin{equation}\label{LL:EJ9}
\begin{split}
& \int_{\R^{n+1}_+} z^{1-2s} |\nabla w|^2\,dX
\le C\left(
\int_{\R^{n+1}_+} z^{1-2s} |\nabla \zeta|^2 \,\overline v^2\,dX
+\int_{\R^{n+1}_+} z^{1-2s} |\nabla \overline v|^2\,\zeta^2\,dX
\right)\\
&\qquad\le
C\left(
\int_{\Omega_+} z^{1-2s} \overline v^2\,dX
+\int_{\Omega_+} z^{1-2s} |\nabla \overline v|^2\,dX
\right).
\end{split}
\end{equation}
Here~$C>0$ stands for a constant, possibly depending
on~$\Omega$, $\Omega_\sharp$ and the other fixed quantities
for this proof, and possibly varying from line to line.
Using~\eqref{LL:EJ9}
and the equivalency of the fractional norms (see e.g. page~1132
in~\cite{CRS}) we obtain that
\begin{equation*}
\begin{split}
&\iint_{\Omega_{\sharp0}\times\Omega_{\sharp0}}
\frac{|v(x)-v(y)|^2}{|x-y|^{n+2s}}\,dx\,dy
+
\iint_{\Omega_{\sharp0}\times\Omega_{0}^c}
\frac{|v(x)|^2}{|x-y|^{n+2s}}\,dx\,dy
\\ &\qquad\le
\iint_{\R^{2n}}
\frac{|w(x,0)-w(y,0)|^2}{|x-y|^{n+2s}}\,dx\,dy 
\\ &\qquad\le
C\left(
\int_{\Omega_+} z^{1-2s} \overline v^2\,dX
+\int_{\Omega_+} z^{1-2s} |\nabla \overline v|^2\,dX
\right).
\end{split}\end{equation*}
This and the embedding results in weighted spaces
(see e.g. Theorem~1.3
in~\cite{FABES-KEN-SER} with~$k:=1$) give that
\begin{equation}\label{LL:EJ12}
\begin{split}&\iint_{\Omega_{\sharp0}\times\Omega_{\sharp0}}
\frac{|v(x)-v(y)|^2}{|x-y|^{n+2s}}\,dx\,dy
+
\iint_{\Omega_{\sharp0}\times\Omega_{0}^c}
\frac{|v(x)|^2}{|x-y|^{n+2s}}\,dx\,dy
\\&\qquad
\le C\int_{\Omega_+} z^{1-2s} |\nabla \overline v|^2\,dX<+\infty.
\end{split}\end{equation}
Now we take~$\tau\in C^\infty(\R^{n})$ such that~$\tau=0$
in~$\Omega_{\sharp0}$ and~$\tau=1$ in~$\Omega_{0}^c$.
Then we have 
\begin{equation}\label{LK:lKH}
\begin{split}
&\iint_{\Omega_{\sharp0}\times\Omega_{0}^c}
\frac{|\varphi(y)|^2}{|x-y|^{n+2s}}\,dx\,dy
=
\iint_{\Omega_{\sharp0}\times\Omega_{0}^c}
\frac{|(\tau\varphi)(x)-(\tau\varphi)(y)|^2}{|x-y|^{n+2s}}\,dx\,dy
\\
&\qquad
\le C\left(
\iint_{\Omega_{\sharp0}\times\Omega_{0}^c}
\frac{|\tau(x)-\tau(y)|^2\,|\varphi(x)|^2}{|x-y|^{n+2s}}\,dx\,dy\right.\\
&\qquad\qquad\left.+\iint_{\Omega_{\sharp0}\times\Omega_{0}^c}
\frac{|\varphi(x)-\varphi(y)|^2\,|\tau(y)|^2}{|x-y|^{n+2s}}\,dx\,dy
\right)\\
&\qquad
\le C\left(
\int_{\Omega_{\sharp0}}
|\varphi(x)|^2 \,dx
+\iint_{\Omega_{\sharp0}\times\Omega_{0}^c}
\frac{|\varphi(x)-\varphi(y)|^2}{|x-y|^{n+2s}}\,dx\,dy
\right).
\end{split}\end{equation}
Notice that these latter quantities
are finite, since we may assume that the external datum~$\varphi$
has finite energy (i.e., that there is one competitor for~$u$ with
finite energy). Then, using~\eqref{LK:lKH}
and the fact that~$v(y)=\overline v(y,0)=\overline u(y,0)=u(y)=\varphi(y)$
for any~$y\in\Omega_{0}^c$,
we obtain
\begin{equation}\label{LL:EJ11}
\iint_{\Omega_{\sharp0}\times\Omega_{0}^c}
\frac{|v(y)|^2}{|x-y|^{n+2s}}\,dx\,dy
=
\iint_{\Omega_{\sharp0}\times\Omega_{0}^c}
\frac{|\varphi(y)|^2}{|x-y|^{n+2s}}\,dx\,dy
<+\infty.
\end{equation}
Similarly (taking the function~$\tau$ to be~$0$
in~$\Omega_{*0}$ and~$\tau=1$ in~$\Omega_{\sharp0}^c$),
we obtain
\begin{equation}\label{LL:EJ11:SIM}
\iint_{\Omega_{*0}\times\Omega_{\sharp0}^c}
\frac{|\varphi(y)|^2}{|x-y|^{n+2s}}\,dx\,dy
<+\infty.
\end{equation}
Using \eqref{LL:EJ11} and~\eqref{LL:EJ12}, we conclude that
\begin{equation}\label{LL:EJ13}
\iint_{\Omega_{\sharp0}\times\Omega_{\sharp0}}
\frac{|v(x)-v(y)|^2}{|x-y|^{n+2s}}\,dx\,dy
+
\iint_{\Omega_{\sharp0}\times\Omega_{0}^c}
\frac{|v(x)-v(y)|^2}{|x-y|^{n+2s}}\,dx\,dy<+\infty.\end{equation}
In addition,
\begin{eqnarray*}
&& \iint_{\Omega_{*0}\times(\Omega_0\setminus \Omega_{\sharp0})}
\frac{|v(x)-v(y)|^2}{|x-y|^{n+2s}}\,dx\,dy
\\
&=&
\iint_{\Omega_{* 0}\times(\Omega_0\setminus \Omega_{\sharp0})}
\frac{|v(x)-\varphi(y)|^2}{|x-y|^{n+2s}}\,dx\,dy
\\ &\le& C\left(
\int_{\Omega_{*0}} |v(x)|^2 \,dx
+
\iint_{\Omega_{* 0}\times(\Omega_0\setminus \Omega_{\sharp0})}
\frac{|\varphi(y)|^2}{|x-y|^{n+2s}}\,dx\,dy
\right).
\end{eqnarray*}
Hence, using \eqref{LL:EJ11:SIM}
and the Trace Inequality in the weighted
spaces (see e.g. Lemma~3.2 in~\cite{SERENA}),
we have that
\begin{equation}\label{KL:LL00yha-f}
\iint_{\Omega_{*0}\times(\Omega_0\setminus \Omega_{\sharp0})}
\frac{|v(x)-v(y)|^2}{|x-y|^{n+2s}}\,dx\,dy
<+\infty.\end{equation}
Furthermore,
$$ \iint_{(\Omega_0\setminus \Omega_{*0})^2}
\frac{|v(x)-v(y)|^2}{|x-y|^{n+2s}}\,dx\,dy
= \iint_{(\Omega_0\setminus \Omega_{*0})^2}
\frac{|\varphi(x)-\varphi(y)|^2}{|x-y|^{n+2s}}\,dx\,dy<+\infty.$$
Using this, \eqref{LL:EJ13}
and~\eqref{KL:LL00yha-f},
one establishes~\eqref{1LL:1}.

{F}rom~\eqref{1LL:1} and Lemma~7.2 of~\cite{CRS}, we know that, for any~$E\subseteq \R^n$
with~$U$ as in~\eqref{extE}, and for any~$F\subset\R^n$
that coincides with~$E$ outside a compact subset of~$B_r$, we have that
\begin{equation}\label{equiv 1}
\Per_\sigma (F,B_r)-\Per_\sigma(E,B_r)=
c_{n,\sigma} \inf_{(\Omega,V)\in I_\sigma} \int_{\Omega_+} z^{1-\sigma}
\Big( |\nabla V|^2-|\nabla U|^2\Big)\,dX.\end{equation}
The set~$I_\sigma$ above consists
of the couples of every bounded Lipschitz set~$\Omega\subset\R^{n+1}$
such that~$\Omega_0\subset B_r$ and every function~$V$
that coincides with~$U$ near~$\partial \Omega$ and such that~$V(x,0)=
(\chi_F-\chi_{F^c})(x)$. Without loss of generality,
we can prescribe that~$V=U$ outside~$\Omega$, since this
does not change the above integrals.

Similarly, for any function~$u$, with~$\overline{u}$ defined in~\eqref{ext},
and any~$v$ that coincides with~$u$
outside a compact subset of~$B_r$, we have that
\begin{equation}\label{equiv 2}\begin{split}
&\iint_{\R^{2n}\setminus(B_r^c)^2}
\frac{|v(x)-v(y)|^2 - |u(x)-u(y)|^2}{|x-y|^{n+2s}}\, dx\,dy
\\ &\qquad= c_{n,s}\inf_{(\Omega,\overline v)\in I_s}
\int_{\Omega_+} z^{1-2s}
\Big( |\nabla \overline{v}|^2-|\nabla \overline{u}|^2\Big)\,dX,
\end{split}\end{equation}
where~$I_s$ above consists
of the couples of every bounded Lipschitz set~$\Omega$
such that~$\Omega_0\subset B_r$ and every function~$\overline{v}$
that coincides with~$\overline{u}$ near~$\partial \Omega$ and
such that~$\overline v(x,0)=v(x)$. Once again,
without loss of generality,
we can prescribe
that~$\overline{v}=\overline{u}$ outside~$\Omega$.

Now we define
\begin{equation}\label{1984}\begin{aligned}
{\mathcal{G}}_\sigma(\Omega,V):=
c_{n,\sigma}\int_{\Omega_+} z^{1-\sigma}
\Big( |\nabla V|^2-|\nabla U|^2\Big)\,dX \\
 {\mbox{and }} \quad
{\mathcal{G}}_s(\Omega,\overline{v}):=
c_{n,s}\int_{\Omega_+} z^{1-2s}
\Big( |\nabla \overline{v}|^2-|\nabla \overline{u}|^2\Big)\,dX
\end{aligned}\end{equation}
and we show that
\begin{equation}\label{equiv 3}
\inf_{(\Omega,V)\in I_\sigma} 
{\mathcal{G}}_\sigma(\Omega,V)
+
\inf_{(\Omega,\overline v)\in I_s}
{\mathcal{G}}_s(\Omega,\overline{v})
=\inf_{(\Omega,\overline v,V)\in {\mathcal{I}}_{s,\sigma} }
\left( 
{\mathcal{G}}_\sigma(\Omega,V)
+
{\mathcal{G}}_s(\Omega,\overline{v})\right)
\end{equation}
where~${\mathcal{I}}_{s,\sigma}$ consists
of the triplets of every bounded Lipschitz set~$\Omega$
such that~$\Omega_0\subset B_r$, every function~$\overline{v}$
that coincides with~$\overline{u}$ 
outside a compact subset of~$\Omega$ and
such that~$\overline v(x,0)=v(x)$,
and every function~$V$
that coincides with~$U$ 
outside a compact subset of~$\Omega$
and such that~$V(x,0)=
(\chi_F-\chi_{F^c})(x)$.
To show~\eqref{equiv 3}, first take
a triplet~$(\Omega,\overline v,V)\in {\mathcal{I}}_{s,\sigma}$.
Then, by construction,~$(\Omega,V)\in I_\sigma$
and~$(\Omega,\overline v)\in I_s$, therefore
$$ \inf_{(\Omega,V)\in I_\sigma}
{\mathcal{G}}_\sigma(\Omega,V)+
\inf_{(\Omega,\overline v)\in I_s}
{\mathcal{G}}_s(\Omega,\overline{v}) \le 
{\mathcal{G}}_\sigma(\Omega,V)+
{\mathcal{G}}_s(\Omega,\overline{v})$$
and so
$$\inf_{(\Omega,V)\in I_\sigma}
{\mathcal{G}}_\sigma(\Omega,V)+
\inf_{(\Omega,\overline v)\in I_s}
{\mathcal{G}}_s(\Omega,\overline{v})
\le \inf_{(\Omega,\overline v,V)\in {\mathcal{I}}_{s,\sigma} }
\left( {\mathcal{G}}_\sigma(\Omega,V)+
{\mathcal{G}}_s(\Omega,\overline{v})\right).$$
This shows one inequality in~\eqref{equiv 3}
and we now focus on the reverse inequality.
For this, we fix~$\eta>0$ and we take~$(\Omega^{1,\eta},V^\eta)\in I_\sigma$
and~$(\Omega^{2,\eta},\overline v^\eta)\in I_s$ such that
\begin{equation}\label{equiv 4}
\eta+ \inf_{(\Omega,V)\in I_\sigma}
{\mathcal{G}}_\sigma(\Omega,V)+
\inf_{(\Omega,\overline v)\in I_s}
{\mathcal{G}}_s(\Omega,\overline{v})
\ge 
{\mathcal{G}}_\sigma(\Omega^{1,\eta},V^\eta)+
{\mathcal{G}}_s(\Omega^{2,\eta},\overline{v}^\eta).
\end{equation}
Let~$\Omega^\eta:=\Omega^{1,\eta}\cup\Omega^{2,\eta}$.
Since~$\Omega^\eta$ contains both~$\Omega^{1,\eta}$
and~$\Omega^{2,\eta}$, we have that~$\overline{v}^\eta$
coincides with~$\overline{u}$
outside a compact subset of~$\Omega^\eta$ and~$V$
coincides with~$U$
outside a compact subset of~$\Omega^\eta$.
Accordingly,~$(\Omega^\eta,\overline v^\eta,V^\eta)
\in {\mathcal{I}}_{s,\sigma}$
and so
\begin{eqnarray*}
{\mathcal{G}}_\sigma(\Omega^{1,\eta},V^\eta)+
{\mathcal{G}}_s(\Omega^{2,\eta},\overline{v}^\eta)
& =&
{\mathcal{G}}_\sigma(\Omega^\eta,V^\eta)+
{\mathcal{G}}_s(\Omega^\eta,\overline{v}^\eta)
\\ &\ge &
\inf_{(\Omega,\overline v,V)\in {\mathcal{I}}_{s,\sigma} }
\left( {\mathcal{G}}_\sigma(\Omega,V)+
{\mathcal{G}}_s(\Omega,\overline{v})
\right).
\end{eqnarray*}
By plugging this into~\eqref{equiv 4}, we obtain
$$
\eta+\inf_{(\Omega,V)\in I_\sigma}
{\mathcal{G}}_\sigma(\Omega,V)+
\inf_{(\Omega,\overline v)\in I_s}
{\mathcal{G}}_s(\Omega,\overline{v})\ge
\inf_{(\Omega,\overline v,V)\in {\mathcal{I}}_{s,\sigma} }
\left( {\mathcal{G}}_\sigma(\Omega,V)+
{\mathcal{G}}_s(\Omega,\overline{v})
\right).
$$
So we take~$\eta$ as small as we wish and we complete
the proof of the reverse inequality in~\eqref{equiv 3}.

Having established~\eqref{equiv 3}, we can sum up~\eqref{equiv 1}
and~\eqref{equiv 2} (taking~$E$ as the positivity set of $u$ and recalling \eqref{1984}) and obtain
\begin{equation}\label{1985}
{\mathcal{F}}(v,F)-{\mathcal{F}}(u,E)
=\inf_{(\Omega,\overline v,V)\in {\mathcal{I}}_{s,\sigma} }
\left(
{\mathcal{G}}_\sigma(\Omega,V)
+{\mathcal{G}}_s(\Omega,\overline{v}) \right).
\end{equation}
{F}rom this, we obtain the desired result by arguing
as follows. First, suppose that $(u,E)$ is a minimizing
pair and take $\Omega$, $\overline v$, and $V$
as in the statement of Proposition \ref{char}. We
define $v(x):=\overline v(x,0)$. Then, the triplet
$(\Omega,\overline v,V)$ belongs to ${\mathcal{I}}_{s,\sigma}$
and therefore, by \eqref{1985} we have that
\begin{equation}\label{1986}
{\mathcal{F}}(v,F)-{\mathcal{F}}(u,E)\le
{\mathcal{G}}_\sigma(\Omega,V)
+{\mathcal{G}}_s(\Omega,\overline{v}).\end{equation}
On the other hand, by item iii)
in the statement of Proposition \ref{char}, we have that
$v(x)=\overline v(x,0)=\overline u(x,0)
=u(x)$ outside a compact subset of $B_r$. Similarly,
by item i) and ii), we have that
$(\chi_{F}-\chi_{F^c})(x)=V(x,0)=U(x,0)=
(\chi_{E}-\chi_{E^c})(x)$ outside a compact subset of $B_r$.
Moreover, $v(x)=\overline v(x,0)\ge 0$ for a.e. $x\in F$
and $v(x)=\overline v(x,0)\le 0$ for a.e. $x\in F^c$,
thanks to item iii). As a consequence, $v$ and $F$
are admissible competitors with respect to
\eqref{Def 1} and \eqref{Def 2}, hence the minimality of $(u,E)$ gives that ${\mathcal{F}}(u,E)
\le {\mathcal{F}}(v,F)$. By inserting this into \eqref{1986}
we obtain
$$ 0\le
{\mathcal{G}}_\sigma(\Omega,V)
+{\mathcal{G}}_s(\Omega,\overline{v}),$$
which, recalling \eqref{1984}, gives \eqref{1983}.

Now, viceversa, suppose that 
\begin{equation}\label{1989}\begin{split}
&{\mbox{formula \eqref{1983} holds under
conditions i), ii) and iii)}}\\ &\qquad{\mbox{ of the
statement of Proposition \ref{char},}}\end{split}
\end{equation}
and let $(v,F)$
be a competing pair
according to \eqref{Def 1} and \eqref{Def 2}.
We show that, in this case,
\begin{equation}\begin{split}\label{1990}
& {\mbox{any
triplet $(\Omega,\overline v,V)\in {\mathcal{I}}_{s,\sigma}$,
satisfies}}\\
& {\mbox{conditions i), ii) and iii) of the
statement of Proposition \ref{char}.}}\end{split}\end{equation}
Indeed, since $(v,F)$ satisfies
\eqref{Def 1} and \eqref{Def 2} we have that
$\overline v(x,0)=v(x)\ge 0$ a.e. in $F$ and
$\overline v(x,0)=v(x)\le 0$ a.e. in $F^c$. This
and the definition of ${\mathcal{I}}_{s,\sigma}$
give \eqref{1990}.

By \eqref{1989} and \eqref{1990}
we obtain that \eqref{1983} holds true for any
triplet $(\Omega,\overline v,V)\in {\mathcal{I}}_{s,\sigma}$.
This means, recalling \eqref{1984}, that
$$ {\mathcal{G}}_\sigma(\Omega,V)
+{\mathcal{G}}_s(\Omega,\overline{v})\ge0$$
for any triplet $(\Omega,\overline v,V)\in {\mathcal{I}}_{s,\sigma}$.
Consequently, by \eqref{1985}, we obtain that 
$$ {\mathcal{F}}(v,F)-{\mathcal{F}}(u,E)\ge0,$$
which shows that $(u,E)$ is minimizing and thus it completes the proof
of Proposition~\ref{char}.\end{proof}

Now we address the proof of
Theorem~\ref{TH:mon}, with the aid of some simple but useful lemmata. 

\begin{lemma}\label{lem:eq}
Let~$c\in\R$ and~$u\in W^{1,1}( B_r)$ be a function satisfying 
\begin{equation}\label{CHI:LKa-89} \nabla u(x)\cdot x= c\, u(x), \quad {\mbox{ for a.e. }}x\in B_r. \end{equation}
Then~$u$ is homogeneous of degree~$c$ (more precisely, $u$
can be extended to a function defined in the whole of~$\R^n\setminus\{0\}$
that is homogeneous of degree~$c$). 
\end{lemma}

\begin{proof} Let~$r'\in(0,r)$ and~$\eta\in C^\infty_0(B_{r'})$.
For~$t\in \R$, with~$|t|< r/r'$, we define
$$ \alpha(t):=\int_{B_{r'}}u(tx)\,\eta(x)\,dx \qquad{\mbox{ and }}\qquad
\beta(t):= t^c\int_{B_{r'}} u(x)\,\eta(x)\,dx.$$
Let also~$\varphi(t):=\alpha(t)-\beta(t)$
and~$\eta_*(y,t):=\eta(y/t)$.
We observe that
$$ \nabla_y \eta_*(y,t)= \frac{1}{t} \nabla \eta\left(\frac{y}{t}\right)$$
and, changing variable~$y:=tx$,
$$ \alpha(t)=\frac{1}{t^n}\int_{\R^n} u(y)\,
\eta\left(\frac{y}{t}\right)\,dy.$$
As a consequence,
\begin{equation}\label{LK:980:304}
\begin{split}
\alpha'(t)\,&=
-\frac{n}{t^{n+1}}\int_{\R^n} u(y)\,
\eta\left(\frac{y}{t}\right)\,dy
-\frac{1}{t^{n+2}}\int_{\R^n} u(y)\,
\nabla\eta\left(\frac{y}{t}\right)\cdot y\,dy\\
&= \frac{1}{t^{n+1}}\left[
-n\int_{\R^n} u(y)\,
\eta_*(y,t)\,dy
-\int_{\R^n} u(y)\,
\nabla_y\eta_*(y,t)\cdot y\,dy
\right] \\
&=
\frac{1}{t^{n+1}}\left[
\int_{\R^n} \nabla u(y)\cdot y\,
\eta_*(y,t)\,dy -
\int_{\R^n} \,{\rm div}_y\, \Big( u(y)\,
\eta_*(y,t)\,y\Big)\,dy
\right]
\\ &=
\frac{1}{t^{n+1}}
\int_{\R^n} \nabla u(y)\cdot y\,
\eta_*(y,t)\,dy \\ &= 
\frac{c}{t^{n+1}}
\int_{\R^n} u(y)\, 
\eta\left(\frac{y}{t}\right)\,dy
\\ &= 
\frac{c}{t}
\int_{\R^n} u(tx)\,
\eta(x)\,dx
\\ &= \frac{c}{t}\alpha(t),
\end{split}\end{equation}
where~\eqref{CHI:LKa-89} was also used.
In addition,
$$ \beta'(t)=
ct^{c-1}\int_{B_{r'}} u(x)\,\eta(x)\,dx=
\frac{c}{t}\beta(t).$$
{F}rom this and~\eqref{LK:980:304}, we obtain that
$$\varphi'(t)=\frac{c}{t} \varphi(t)$$
for any~$t$ close to~$1$.
Also, we have that
$$ \alpha(1)=\int_{B_{r'}}u(x)\,\eta(x)\,dx =\beta(1),$$
and so~$\varphi(1)=0$. By uniqueness of the Cauchy problem
(starting at time~$t=1$), we get that~$\varphi$ vanishes identically,
and so
$$ \int_{B_{r'}} \Big[ u(tx)
- t^c u(x)\Big]\,\eta(x)\,dx =0.$$
Since this is valid for any test function~$\eta$,
we conclude that~$u(tx)
=t^c u(x)$, a.e. in~$x\in B_{r'}$, for any~$t\in\R$ with~$|t|< r/r'$.
Since~$r'$ can be chosen arbitrarily in~$(0,r)$,
this means that~$u(tx)
=t^c u(x)$, for a.e. fixed~$x\in B_r$, and for any~$t\in\R$
with~$|t|<r/|x|$.
This proves the homogeneity of~$u$ in~$B_r$
and then one can extend it on the whole of the space.
\end{proof}

In the following lemma we show that~$\Phi_u$, defined in~\eqref{Phi}, 
possesses a natural scaling. 

\begin{lemma}\label{lem:scaling}
Let~$(u,E)$ be a minimizing pair in~$B_\rho$ and let~$\Phi_u$ be as in~\eqref{Phi}.
Let also
\begin{equation}\label{GU}
G_u(r):=
r^{\sigma-n}\left(\int_{\mathcal B^+_r}z^{1-2s}
|\nabla\overline{u}|^2\,dX+
c_{n,s,\sigma}\int_{\mathcal B^+_r}z^{1-\sigma}|\nabla U|^2\,dX\right),\end{equation}
and let~$(u_r,E_r)$ 
be the rescaled pair defined in~\eqref{rescaled}. Then, for any~$t>0$,
\begin{equation}\label{p15} G_u(rt)=G_{u_r}(t)
\ {\mbox{ and }} \
\Phi_u(rt)=\Phi_{u_r}(t). \end{equation}
\end{lemma}

\begin{proof} The claim follows by observing that~$
\overline{u}_r(X)=r^{\frac{\sigma}{2}-s}\overline{u}(rX)$ and~$
U_r(X)=U(rX)$.
\end{proof}

With this, we are in the position of proving Theorem~\ref{TH:mon}.

\begin{proof}[Proof of Theorem~\ref{TH:mon}] 
We will prove that 
$$ \frac{d}{dr}\Phi_u(r)\ge 0 \quad {\mbox{ for a.e. }}r. $$  
We write 
$$ \Phi_u(r)=G_u(r)-H_u(r), $$ 
where $G_u$ is as in~\eqref{GU} and
$$ H_u(r):=\, \left(s-\frac{\sigma}{2}\right) r^{\sigma-n-1}\int_{\partial\mathcal B^+_r}z^{1-2s}\, \overline u^2\,d\mathcal H^n.$$
Thanks to the scaling properties in Lemma~\ref{lem:scaling}, 
it is sufficient to prove the theorem
when~$r=1$.

Given a small~$\epsilon>0$, we consider a competitor~$(\overline{u}^\epsilon,U^\epsilon)$ 
for~$(\overline{u},U)$ defined as follows
$$ 
\overline{u}^\epsilon(X):=\left\{ 
\begin{array}{lll} 
(1-\epsilon)^{s-\frac{\sigma}{2}}\, \overline u\left(\frac{X}{1-\epsilon}\right)   \quad & {\mbox{ if }} X\in\mathcal B_{1-\epsilon}^+,\\
|X|^{s-\frac{\sigma}{2}}\, \overline u\left(\frac{X}{|X|}\right) \quad & {\mbox{ if }} X\in\mathcal B_1^+\setminus\mathcal B_{1-\epsilon}^+,\\
\overline u(X) \quad & {\mbox{ if }} X\in (\mathcal B_1^+)^c,
\end{array} 
\right.
$$
and 
$$
U^\epsilon(X):=\left\{ 
\begin{array}{lll} 
U\left(\frac{X}{1-\epsilon}\right) \quad & {\mbox{ if }} X\in\mathcal B_{1-\epsilon}^+,\\
U\left(\frac{X}{|X|}\right) \quad & {\mbox{ if }} X\in\mathcal B_1^+\setminus\mathcal B_{1-\epsilon}^+,\\
U(X) \quad & {\mbox{ if }} X\in (\mathcal B_1^+)^c.
\end{array} 
\right.
$$
Since the pair~$(u,E)$ is a minimizer and~$\overline{u}^\eps$ and $U^\eps$ 
satisfy conditions i), ii) and iii) in the statement of Proposition~\ref{char}, 
from~\eqref{1983} we have that 
\begin{equation}\label{mini}
G_u(1)\le G_{u^\epsilon}(1),
\end{equation} 
where $u^\epsilon(x):=\overline{u}^\epsilon(x,0)$. 
Now, we compute~$G_u(1)$ and~$G_{u^\epsilon}(1)$ by splitting the integrals 
in~$\mathcal B_1^+$ into integrals in~$\mathcal B_{1-\epsilon}^+$ 
and~$\mathcal B_1^+\setminus\mathcal B_{1-\epsilon}^+$. 
Therefore, we have 
\begin{equation}\begin{split}\label{Gu}
G_u(1)=&\, \int_{\mathcal B_{1-\epsilon}^+}z^{1-2s}|\nabla\overline u|^2\, dX+ 
\epsilon\int_{\partial\mathcal B_1^+}z^{1-2s}|\nabla\overline u|^2\,d\mathcal H^n\\
&\quad + c_{n,s,\sigma}\left(\int_{\mathcal B_{1-\epsilon}^+}z^{1-\sigma}|\nabla U|^2\,dX+\epsilon\int_{\partial\mathcal B_1^+}z^{1-\sigma}|\nabla U|^2\,d\mathcal H^n\right) +o(\epsilon)\\ 
=&\, (1-\epsilon)^{n-\sigma}G_u(1-\epsilon)+\epsilon\int_{\partial\mathcal B_1^+} z^{1-2s}\left(|\overline{u}_\tau|^2 +|\overline{u}_\nu|^2\right)d\mathcal H^n\\
&\quad +\epsilon\, c_{n,s,\sigma}\int_{\partial\mathcal B_1^+}z^{1-\sigma}\left(|U_\tau|^2 +|U_\nu|^2\right)d\mathcal H^n +o(\epsilon), 
\end{split}\end{equation}
where, as usual, $u_\tau$ and~$u_\nu$ stand for the tangential and the normal gradient of~$u$ on~$\partial\mathcal B_1^+$. 

To compute~$G_{u^\epsilon}(1)$ we notice that~$\overline{u}^\epsilon$ 
and~$U^\epsilon$ coincide with the rescaling~$\overline{u}_{1/(1-\epsilon)}$ 
and~$U_{1/(1-\epsilon)}$, respectively, in~$\mathcal B_{1-\epsilon}^+$,
as given in~\eqref{rescaled},
hence
\begin{equation}\begin{split}\label{Gueps}
G_{u^\epsilon}(1)=&\, (1-\epsilon)^{n-\sigma}G_{u_{1/(1-\epsilon)}}(1-\epsilon)
+\epsilon\, c_{n,s,\sigma}\int_{\partial\mathcal B_1^+}z^{1-\sigma}|U_\tau|^2\, d\mathcal H^n\\
&\quad +\epsilon\int_{\partial\mathcal B_1^+}z^{1-2s}\left(|\overline{u}_\tau|^2+\left(s-\frac{\sigma}{2}\right)^2\overline{u}^2\right)d\mathcal H^n
+o(\epsilon).  
\end{split}\end{equation} 
Also, from Lemma~\ref{lem:scaling}
(used here with~$t:=1-\epsilon$ and~$r:=1/(1-\epsilon)$), we see that
$$ G_{u_{1/(1-\epsilon)}}(1-\epsilon)= G_u(1). $$ 
Therefore, \eqref{Gueps} 
becomes
\begin{equation}\begin{split}\label{Gueps2}
G_{u^\epsilon}(1)=&\, (1-\epsilon)^{n-\sigma}G_u(1)
+\epsilon\, c_{n,s,\sigma}\int_{\partial\mathcal B_1^+}z^{1-\sigma}|U_\tau|^2\, d\mathcal H^n\\
&\quad +\epsilon\int_{\partial\mathcal B_1^+}z^{1-2s}\left(|\overline{u}_\tau|^2+\left(s-\frac{\sigma}{2}\right)^2\overline{u}^2\right)d\mathcal H^n
+o(\epsilon).  
\end{split}\end{equation} 
Plugging~\eqref{Gu} and~\eqref{Gueps2} into~\eqref{mini} we obtain
\begin{eqnarray*}
(1-\epsilon)^{n-\sigma} G_u(1) \ge (1-\epsilon)^{n-\sigma} G_u(1-\epsilon) 
&+& \epsilon \int_{\partial\mathcal B_1^+} z^{1-2s} 
\left[ |\overline{u}_\nu|^2-\left(s-\frac{\sigma}{2}\right)^2\overline{u}^2\right]d\mathcal H^n\\
&+& \epsilon\, c_{n,s,\sigma}\int_{\partial\mathcal B_1^+}z^{1-\sigma}|U_\nu|^2\,d\mathcal H^n +o(\epsilon),
\end{eqnarray*}
which implies 
\begin{equation}\label{Gfinal}
G'_u(1)\ge \int_{\partial\mathcal B_1^+} z^{1-2s} 
\left[ |\overline{u}_\nu|^2-\left(s-\frac{\sigma}{2}\right)^2\overline{u}^2\right]d\mathcal H^n + c_{n,s,\sigma}\int_{\partial\mathcal B_1^+}z^{1-\sigma}|U_\nu|^2\,d\mathcal H^n.
\end{equation}

Now, we claim that 
\begin{equation}\label{HHH}
H_u'(1)= \left(s-\frac{\sigma}{2}\right)\int_{\partial\mathcal B_1^+}z^{1-2s}\left(2\,\overline{u}\,\overline{u}_\nu
+(\sigma-2s)\overline{u}^2\right)d\mathcal H^n. 
\end{equation} 
For this, we notice that, by using the change of variable~$X=rY$, with $z=rw$, 
we can rewrite~$H_u(r)$ as 
$$ H_u(r)=\left(s-\frac{\sigma}{2}\right)r^{\sigma-2s}\int_{\partial\mathcal B_1^+}w^{1-2s}\, 
\overline{u}^2(rY)\,d\mathcal H^n. $$
Taking the derivative with respect to $r$ and then setting $r=1$ 
we obtain \eqref{HHH}. 

From~\eqref{Gfinal} and~\eqref{HHH} we deduce that 
\begin{eqnarray*}
\Phi'_u(1)&\ge & \int_{\partial\mathcal B_1^+}z^{1-2s}\left[|\overline{u}_\nu|^2-\left(s-\frac{\sigma}{2}\right)^2\overline{u}^2\right]d\mathcal H^n +c_{n,s,\sigma} \int_{\partial\mathcal B_1^+}z^{1-\sigma}|U_\nu|^2\, d\mathcal H^n \\
&&\qquad -\left(s-\frac{\sigma}{2}\right)\int_{\partial\mathcal B_1^+}z^{1-2s}\left(2\,\overline{u}\,\overline{u}_\nu +(\sigma-2s)\overline{u}^2\right)d\mathcal H^n. 
\end{eqnarray*}
Notice that 
$$|\overline{u}_\nu|^2-\left(s-\frac{\sigma}{2}\right)^2\overline{u}^2-2\left(s-\frac{\sigma}{2}\right)\overline{u}\,\overline{u}_\nu-\left(s-\frac{\sigma}{2}\right)(\sigma-2s)\overline{u}^2 =
\left(\overline{u}_\nu-\left(s-\frac{\sigma}{2}\right)\overline{u}\right)^2, $$
and so 
\begin{equation}\label{monooo}
\Phi'_u(1)\ge \int_{\partial\mathcal B_1^+}z^{1-2s}\left(\overline{u}_\nu-\left(s-\frac{\sigma}{2}\right)\overline{u}\right)^2 \, d\mathcal H^n +c_{n,s,\sigma}\int_{\partial\mathcal B_1^+}z^{1-\sigma}|U_\nu|^2\, d\mathcal H^n. 
\end{equation}
This implies that~$\Phi_u$ is increasing in~$(0,\rho)$.  

Moreover, if~$\Phi_u$ is constant, then~\eqref{monooo} and 
Lemma~\ref{lem:eq} give that~$\overline{u}$ is homogeneous of degree~$s-\frac{\sigma}{2}$ 
and~$U$ is homogeneous of degree~$0$. 
Conversely, suppose that~$\overline{u}$ and~$U$ are homogeneous 
of degree~$s-\frac{\sigma}{2}$ and~$0$ respectively. 
Then,~$u=u_r$ for any~$r>0$, and therefore 
from Lemma~\ref{lem:scaling} we have that~$\Phi_u(rt)=\Phi_u(t)$, 
which implies that~$\Phi_u$ is constant. 
This concludes the proof of Theorem~\ref{TH:mon}.
\end{proof}

\section{Dimensional reduction and proof of Theorem~\ref{DR}}\label{DRA}

In order to establish the dimensional reduction property, as stated in
Theorem~\ref{DR},
we recall a useful gluing result from Lemma 10.2 of \cite{CRS}
(this is indeed just the translation of such result by some $a$ in the $(n+1)$th
component):

\begin{lemma}\label{10.2}
Fix $R$, $a>0$. Let $\alpha\in(-1,1)$.
Let $W$ be a bounded function in ${\mathcal{B}}_R^+\subset\R^{n+1}$.
Suppose that
$$ {\mbox{$W=0$ in a neighborhood of $\partial {\mathcal{B}}_R$ and }}
\int_{{\mathcal{B}}_R^+} z^\alpha |\nabla W|^2 dX < \infty,$$
where $X=(x,z)\in\R^{n+1}$.
Then there exists a function ${\mathcal{W}}={\mathcal{W}}(x,x_{n+1},z)$
defined on ${\mathcal{B}}_R^+ \times [a-1,a+1]$
with the following properties:
\begin{eqnarray}
&& {\mathcal{W}}=0 {\mbox{ if }} x_{n+1} < a-\frac12, \label{preCOMP0}\\
&& {\mathcal{W}}=W {\mbox{ if }} x_{n+1} > a+\frac12, \label{preCOMP1.1}\\
&& {\mathcal{W}}=0 {\mbox{ in a neighborhood of }} \partial{\mathcal{B}}_R^+\times [a-1,a+1],
\label{preCOMP1.2}\\
&& {\mathcal{W}}(x,x_{n+1},0)=
\left\{
\begin{matrix}
0 & {\mbox{ if }} x_{n+1}\leq a,\\
W(x,0) & {\mbox{ if }} x_{n+1}>a.
\end{matrix}
\right.
\label{0.3bis}
\end{eqnarray}
with
\begin{equation*}
C({\mathcal{W}}):=\int_{ {\mathcal{B}}_R^+ \times [a-1,a+1]} 
z^\alpha |\nabla {\mathcal{W}}|^2 d{\mathcal{X}}
{\mbox{ finite and independent of $a$,}}
\end{equation*}
where ${\mathcal{X}}=(x,x_{n+1},z)\in\R^{n+2}$.
\end{lemma}

{F}rom the geometric point of view, Lemma~\ref{10.2}
states that one can interpolate $0$ with a given function~$W$
by performing a sharp switch at~$\{x_{n+1}=a\}\cup \{z=0\}$, maintaining
the energy finite. As a consequence, we obtain:

\begin{corollary}\label{Co Z}
Fix $R$, $a>0$. Let $\alpha\in(-1,1)$.
Let ${\mathcal{U}}$ and ${\mathcal{V}}$ be bounded functions in ${\mathcal{B}}_R^+\subset\R^{n+1}$
with ${\mathcal{U}}={\mathcal{V}}$ in a neighborhood of $\partial{\mathcal{B}}_R$ and
$$ \int_{{\mathcal{B}}_R^+} z^\alpha \Big(|\nabla {\mathcal{U}}|^2+|\nabla {\mathcal{V}}|^2 \Big) dX < \infty,$$
where $X=(x,z)\in\R^{n+1}$.
Then there exists a function ${\mathcal{Z}}={\mathcal{Z}}(x,x_{n+1},z)$
defined on ${\mathcal{B}}_R^+ \times [-(a+1),a+1]$
with the following properties:
\begin{eqnarray}
&& {\mbox{${\mathcal{Z}}$ is even in $x_{n+1}$,}} \label{COMP0} \\
&& {\mathcal{Z}}={\mathcal{V}} {\mbox{ if }} |x_{n+1}| < a-\frac12, 
\label{COMP1} \\
&& {\mathcal{Z}}={\mathcal{U}} {\mbox{ in a neighborhood of }}
\partial{\mathcal{B}}_R^+\times [-(a+1),a+1], \label{COMP2}\\
&&  {\mathcal{Z}}(x,x_{n+1},0)=\left\{
\begin{matrix}
 {\mathcal{V}}(x,0) & {\mbox{ if }} |x_{n+1}|\in [0,a],\\
 {\mathcal{U}}(x,0) & {\mbox{ if }} |x_{n+1}|\in (a,a+1],\label{0.6bis}
\end{matrix}\right.
\end{eqnarray}
with
\begin{equation}\label{88711}
C({\mathcal{Z}}):=\int_{{\mathcal{B}}_R^+ \times [a-1,a+1]} z^\alpha |\nabla {\mathcal{Z}}|^2 d{\mathcal{X}}
{\mbox{ finite and independent of $a$,}}
\end{equation}
where ${\mathcal{X}}=(x,x_{n+1},z)\in\R^{n+2}$.
\end{corollary}

\begin{proof}
Let ${\mathcal{W}}$ be the function obtained
by applying Lemma \ref{10.2} to the function $W:={\mathcal{U}}-{\mathcal{V}}$, 
and let
$\tilde{\mathcal{W}}(x,x_{n+1},z):={\mathcal{W}}(x,|x_{n+1}|,z)$.
Then let 
$${\mathcal{Z}}(x,x_{n+1},z):=\left\{ \begin{matrix} \tilde{\mathcal{W}}(x,x_{n+1},z)+{\mathcal{V}}(x,z) &{\mbox{ if }}|x_{n+1}|\in(a-1,a+1],\\
{\mathcal{V}}(x,z) &{\mbox{ if }}|x_{n+1}|\in[0,a-1].
\end{matrix}\right. $$
We remark that \eqref{COMP0} holds true by construction, 
while \eqref{COMP1} and \eqref{COMP2} follow 
from \eqref{preCOMP0} and \eqref{preCOMP1.2} respectively.
Also, \eqref{0.6bis} is a consequence of~\eqref{0.3bis}.
\end{proof}

With this, we are ready for the proof of Theorem~\ref{DR}:

\begin{proof}[Proof of Theorem~\ref{DR}] If $(u,E)$ is a minimizer in $\R^n$, then 
$(u^\star,E^\star)$ is a minimizer in $\R^{n+1}$: this follows easily from
Proposition \ref{char} by slicing, using that for a function $\overline v(x,x_{n+1},z)$
one has that $|\nabla_{\mathcal{X}} \overline v|^2\ge
|\nabla_X \overline v|^2$ for any fixed $x_{n+1}$, where
${\mathcal{X}}:=(x,x_{n+1},z)$ and
$X:=(x,z)$.

Now we suppose that $(u^\star,E^\star)$
is minimizing in $\R^{n+1}$ and we show that $(u,E)$
is minimizing in any domain of $\R^n$. To this extent,
we use again Proposition \ref{char}. For this, we fix a competitor triplet
$V$, $\overline v$ and $\Omega:= {\mathcal{B}}_R\subset \R^{n+1}$
as prescribed by Proposition \ref{char}
(in particular, we also have a set~$F$ given in 
item~ii) there), and our goal is to show that
\eqref{1983} holds true in this case.
The idea is to construct a competitor in one dimension more
with respect to $(u^\star,E^\star)$ and
thus to use the minimality of $(u^\star,E^\star)$
for this competitor.
The details of the computation go as follows. Fix $a>0$, to be taken
arbitrarily large at the end of the argument. We take
${\mathcal{Z}}_s$ to be the function constructed
in Corollary \ref{Co Z}, applied here with $\alpha:=1-2s$,
${\mathcal{U}}:=\overline u$ and ${\mathcal{V}}:=\overline v$.
By \eqref{COMP1},
\begin{eqnarray*}
\int_{{\mathcal{B}}_R^+ \times [-(a-1),a-1]} z^{1-2s} |\nabla_{\mathcal{X}} {\mathcal{Z}}_s|^2 d{\mathcal{X}}
&=&
\int_{{\mathcal{B}}_R^+ \times [-(a-1),a-1]} z^{1-2s} |\nabla_X \overline v|^2 d{\mathcal{X}}\\
&=& 2(a-1)
\int_{{\mathcal{B}}_R^+} z^{1-2s} |\nabla_X \overline v|^2 dX, 
\end{eqnarray*}
since $\overline v$ does not depend on $x_{n+1}$.
Therefore, by \eqref{COMP0},
\begin{equation}\begin{split}\label{77-1}
&\int_{{\mathcal{B}}_R^+ \times [-(a+1),a+1]} z^{1-2s} |\nabla {\mathcal{Z}}_s|^2 d{\mathcal{X}}\\
&\qquad =
2\int_{{\mathcal{B}}_R^+ \times [a-1,a+1]} z^{1-2s} |\nabla {\mathcal{Z}}_s|^2 d{\mathcal{X}}
+
2(a-1)
\int_{{\mathcal{B}}_R^+} z^{1-2s} |\nabla \overline v|^2 dX.
\end{split}\end{equation}
Similarly, one can define ${\mathcal{Z}}_\sigma$
to be the function constructed
in Corollary \ref{Co Z}, applied here with $\alpha:=1-\sigma$,
${\mathcal{U}}:=U$ and ${\mathcal{V}}:=V$.
In analogy with \eqref{77-1}, we obtain
\begin{equation}\begin{split}\label{77-2}
&\int_{{\mathcal{B}}_R^+ \times [-(a+1),a+1]} z^{1-\sigma} |\nabla {\mathcal{Z}}_\sigma|^2 d{\mathcal{X}}
\\&\qquad =
2\int_{{\mathcal{B}}_R^+ \times [a-1,a+1]} z^{1-\sigma} |\nabla {\mathcal{Z}}_\sigma|^2 d{\mathcal{X}}
+
2(a-1)
\int_{{\mathcal{B}}_R^+} z^{1-\sigma} |\nabla V|^2 dX.
\end{split}\end{equation}
Now we point out that 
\begin{equation}\begin{split}\label{77877}
&{\mbox{${\mathcal{Z}}_\sigma$, ${\mathcal{Z}}_s$ and 
${\mathcal{B}}_R \times [-(a+1),a+1]\subset \R^{n+2}$
are an admissible triplet}}\\
& {\mbox{(with respect to $(u^\star, E^\star)$,
as prescribed by Proposition \ref{char}).}}\end{split}\end{equation}
For this, we observe that
${\mathcal{Z}}_\sigma=U$
and
${\mathcal{Z}}_s = \overline u$ on $\partial\big(
{\mathcal{B}}_R \times [-(a+1),a+1]\big)$, thanks to \eqref{COMP2}
(the first of these observations takes care of item i) in the statement of Proposition \ref{char},
while the second is involved in item iii)).

Furthermore, by \eqref{0.6bis}, we see that
${\mathcal{Z}}_\sigma\big|_{\{ z=0\}}= \chi_{\tilde F}-\chi_{\tilde F^c}$,
where
$$ \tilde F := \Big( F\cap \{x_{n+1}\le a\}\Big)\cup
\Big( E\cap \{x_{n+1}> a\}\Big),$$
and
$${\mathcal{Z}}_s\big|_{\{ z=0\}}= \left\{
\begin{matrix}
\overline v\big|_{\{ z=0\}} & {\mbox{ if }} x_{n+1}\le a,\\
\overline u\big|_{\{ z=0\}} & {\mbox{ if }} x_{n+1}> a.
\end{matrix}\right.$$
Accordingly, ${\mathcal{Z}}_s\big|_{\{ z=0\}} \ge 0$ a.e. in~$\tilde F$
and ${\mathcal{Z}}_s\big|_{\{ z=0\}} \le 0$ a.e. in~$\tilde F^c$.
This proves \eqref{77877}.

Using \eqref{77877} and the minimality of $(u^\star,E^\star)$,
we deduce from Proposition \ref{char} that
\begin{eqnarray*}
&& \int_{{\mathcal{B}}_R^+ \times [-(a+1),a+1]} z^{1-2s} |\nabla \overline u^\star|^2 d{\mathcal{X}}
+c_{n,s,\sigma}
\int_{{\mathcal{B}}_R^+ \times [-(a+1),a+1]} z^{1-\sigma} |\nabla U^\star|^2 d{\mathcal{X}}\\
&&\qquad\le
\int_{{\mathcal{B}}_R^+ \times [-(a+1),a+1]} z^{1-2s} |\nabla {\mathcal{Z}}_s|^2 d{\mathcal{X}}
+c_{n,s,\sigma}\int_{{\mathcal{B}}_R^+ \times [-(a+1),a+1]} z^{1-\sigma} |\nabla {\mathcal{Z}}_\sigma|^2 d{\mathcal{X}}
,\end{eqnarray*}
where $\overline u^\star(x,x_{n+1},z):=\overline u^\star(x,z)$ and
$U(x,x_{n+1},z):=U(x,z)$. Thus, we can compute the integrals on the left hand side
in the $(n+1)$th variable and use \eqref{77-1} and \eqref{77-2}: we obtain
\begin{eqnarray*}
&& 2(a+1)\int_{{\mathcal{B}}_R^+ } z^{1-2s} |\nabla \overline u^\star|^2 dX
+2(a+1) c_{n,s,\sigma}
\int_{{\mathcal{B}}_R^+ } z^{1-\sigma} |\nabla U^\star|^2 dX\\
&&\qquad\le
2(a-1)
\int_{{\mathcal{B}}_R^+} z^{1-2s} |\nabla \overline v|^2 dX
+
2(a-1)
c_{n,s,\sigma}\int_{{\mathcal{B}}_R^+} z^{1-\sigma} |\nabla V|^2 dX+O(1),\end{eqnarray*}
where $O(1)$ is a quantity independent on $a$ (recall \eqref{88711}).
Hence, we divide by $2a$ and we take $a$ as large as we wish: we conclude that
\begin{eqnarray*}
&& \int_{{\mathcal{B}}_R^+ } z^{1-2s} |\nabla \overline u^\star|^2 dX
+ c_{n,s,\sigma}
\int_{{\mathcal{B}}_R^+ } z^{1-\sigma} |\nabla U^\star|^2 dX\\
&&\qquad\le
\int_{{\mathcal{B}}_R^+} z^{1-2s} |\nabla \overline v|^2 dX
+
c_{n,s,\sigma}\int_{{\mathcal{B}}_R^+} z^{1-\sigma} |\nabla V|^2 dX.\end{eqnarray*}
By Proposition \ref{char}, this says that $(u,E)$ is a minimal pair in $\R^n$, as desired.
\end{proof}

\section{Some glueing lemmata}\label{GLUE}

Here we present some results that glue two admissible pairs together
by estimating the excess of energy produced by this surgery.

\begin{lemma}\label{lemma:phi}
Let~$\beta\in(0,1)$ and~$\eps\in(0,1/4).$
Then there exists a function
$$ \phi:=\phi_\eps:\overline{\R^{n+1}_+}
\rightarrow [0,1]$$ such that
\begin{eqnarray}
\label{o-01}&& \phi(x,z)=0 \ {\mbox{ for any }} (x,z)\in{\mathcal{B}}^+_{1-\eps},\\
\label{o-02}&& \phi(x,z)=1 \ {\mbox{ for any }} (x,z)\in\R^{n+1}_+\setminus{\mathcal{B}}^+_{1+\eps},\\
\label{o-03}&& \phi(x,0)=\chi_{\R^n\setminus B_1}(x) \ {\mbox{ a.e. }} x\in \R^n,\\
{\mbox{and }}
\label{o-04}
&&\int_{\R^{n+1}_+} z^\beta |\nabla\phi(x,z)|^2\,dx\,dz\le \frac{C}{\eps},\end{eqnarray}
for some $C>0$.
\end{lemma}

\begin{proof} Let~$\eps':=\eps/2$ and, for any~$X=(x,z)\in\R^{n+1}_+$, we define
$$ \phi_1(X):=\left\{
\begin{matrix}
0 & {\mbox{ if $|X|<1-\eps'$,}}\\
\big( |X|-1+\eps'\big)/(2\eps') & {\mbox{ if $|X|\in [1-\eps',\, 1+\eps')$,}}\\
1 & {\mbox{ if $|X|\ge 1+\eps'$.}}
\end{matrix}
\right. $$
Let also
$$ \phi_2(X):=\left\{
\begin{matrix}
0 & {\mbox{ if $|x|<1-z$,}}\\
\big( |x|-1+z\big)/(2z) & {\mbox{ if $|x|\in [1-z,\,1+z)$,}}\\
1 & {\mbox{ if $|x|\geq 1+z$}}
\end{matrix}
\right. $$
and
$$ \eta(X):=\left\{
\begin{matrix}
z/\eps' & {\mbox{ if $z\in (0,\eps')$,}}
\\1 & {\mbox{ if $z\ge\eps'$.}}
\end{matrix}
\right. $$
We also set
$$ \phi:= \eta\phi_1 +(1-\eta)\phi_2.$$
We remark that~$\eta\big|_{\{z=0\}}=0$, thus
$$ \phi\big|_{\{z=0\}} =\phi_2\big|_{\{z=0\}}=\chi_{\R^n\setminus B_1},$$
which proves \eqref{o-03}.

Now we prove \eqref{o-01}.
For this, we fix~$X\in\R^{n+1}_+$, with~$|X|<1-\eps=1-2\eps'$.
Then~$\phi_1(X)=0$, hence
\begin{equation}\label{tydue2if}
\phi(X)=(1-\eta(X))\phi_2(X).\end{equation}
Now, if~$|x|<1-z$, we have that~$\phi_2(X)=0$, and therefore~$\phi(X)=0$,
that proves~\eqref{o-01}
in this case. Accordingly, we may suppose that~$
|x|\ge 1-z$. So we have that
$$ z\ge \frac{1-(1-z)^2-z^2}{2}\ge \frac{1-|x|^2-z^2}{2}
=\frac{1-|X|^2}{2}>\frac{1-(1-2\varepsilon')^2}{2}>\varepsilon',$$
and so~$\eta(X)=1$. As a consequence of this
and~\eqref{tydue2if}, we obtain that~$\phi(X)=0$, and this
establishes~\eqref{o-01}.

Now we prove \eqref{o-02}.
To this goal, we fix~$X\in\R^{n+1}_+$ with~$|X|\ge 1+\eps=1+2\eps'$.
In this case, we have that
\begin{equation}\label{tydue2if-2}
\phi_1(X)=1.\end{equation}
Now, if~$z\ge\eps'$, we have that~$\eta(X)=1$ and so
$$ \phi(X)=\phi_1(X)=1.$$
Thus, we can assume that~$z<\eps'$. In this case, we have that~$|x|^2 =
|X|^2 -z^2>(1+z)^2$, which
implies that~$\phi_2(X)=1$. 
Combining this and~\eqref{tydue2if-2} we conclude that~$\phi(X)= 
\eta(X)+(1-\eta(X))=1$, which proves~\eqref{o-02}.

Now we prove \eqref{o-04}.
For this, we first observe that
$$ \nabla\phi_2(X)=\left( \frac{x}{2z\,|x|},\,
\frac{1-|x|}{2z^2}
\right)\,\chi_{(1-z,\,1+z)}(|x|)$$
and therefore
\begin{equation}\label{6df7g8fff6} |\nabla\phi_2(X)|\le \frac{C}{z}
\,\chi_{(1-z,\,1+z)}(|x|),\end{equation}
for some~$C>0$.
Moreover
$$ |\nabla\phi_1(X)|\le \frac{C}{\eps}\,\chi_{(1-\eps',\, 1+\eps')} (|X|).$$
As a consequence
\begin{equation}\label{conto-1-1-0}
\int_{\R^{n+1}_+} z^\beta |\nabla\phi_1|^2\,dX
\le \frac{C\,\Big| \big\{|X|\in (1-\eps',\, 1+\eps')
\big\}\Big|}{\eps^2} =\frac{C}{\eps},\end{equation}
up to renaming~$C>0$.
Also, $\phi=\phi_1$ if~$z>\eps'$, therefore we deduce from~\eqref{conto-1-1-0}
that
\begin{equation}\label{conto-1-1-1}
\int_{\{z>\eps'\}} z^\beta |\nabla\phi|^2\,dX
=\int_{\{z>\eps'\}} z^\beta |\nabla\phi_1|^2\,dX
\le \frac{C}{\eps}.\end{equation}
Furthermore, if~$z\le \eps'$ and~$|X|>2$, we have that~$\phi_1(X)=1$
and~$|x|^2 =|X|^2 -z^2>(1+z)^2$,
that gives~$\phi_2(X)=1$. 

As a consequence, $\phi_1-\phi_2=0$ if~$z\le\eps'$ and~$|X|>2$, therefore
\begin{equation}\label{conto-1-1-2}
\begin{split}
& \int_{\{z\le \eps'\}} z^\beta |\nabla\eta|^2\,|\phi_1-\phi_2|^2\,dX
\le \frac{C}{\eps^2}\int_{\{z\le \eps'\}\cap\{|X|\le2\}} z^\beta \,dX
\\ &\qquad\le \frac{C}{\eps^2}\int_{\{z\le \eps'\}} z^\beta \,dz
= C\eps^{\beta-1}.
\end{split}
\end{equation}
In addition, using~\eqref{6df7g8fff6}, we obtain that
\begin{equation}\label{conto-1-1-3}\begin{split}
&\int_{\{z\le \eps'\}} z^\beta |\nabla\phi_2|^2\,dX\le
C\int_{\{z\le \eps'\}\cap \{|x|\in (1-z,\,1+z)\}} z^{\beta-2} \,dX\\
&\qquad\le
C\int_{\{z\le \eps'\}} z^{\beta-1} \,dz =C\eps^\beta.
\end{split}\end{equation}
Notice also that
$$ \nabla\phi=\nabla\eta(\phi_1-\phi_2)+\eta\nabla\phi_1+(1-\eta)
\nabla\phi_2.$$
Consequently, by gathering the estimates in~\eqref{conto-1-1-0},
\eqref{conto-1-1-2} and~\eqref{conto-1-1-3}
and using Young inequality, we deduce that
\begin{eqnarray*}&&\int_{\{z\le \eps'\}} z^\beta|\nabla \phi|^2\,dX
\\&\le& C\int_{\{z\le \eps'\}} z^\beta
\Big( 
|\nabla\eta|^2\,|\phi_1-\phi_2|^2+\eta^2|\nabla\phi_1|^2+(1-\eta)^2
|\nabla\phi_2|^2
\Big)\,dX\\&\le& \frac{C}{\eps},\end{eqnarray*}
up to renaming~$C$. This and~\eqref{conto-1-1-1} imply~\eqref{o-04}.
\end{proof}

Next we give a glueing result: namely, given any admissible pair in~$B_1$,
we glue it to another admissible pair outside~$B_1$, keeping the
energy contribution under control.

\begin{lemma}\label{LE conv}
Let $\epsilon>0$. Let $(u_i,E_i)$, $i\in\{1,2\}$, be admissible pairs in $B_2$, 
and let $\overline u_i$ and $U_i$ be their extensions according to \eqref{ext} 
and \eqref{extE}. Let also~$\phi$ be the function
introduced in Lemma~\ref{lemma:phi}.

Then there exist $F\subseteq\R^n$, $\overline v:\R^{n+1}_+\rightarrow\R$ 
and $V:\R^{n+1}_+\rightarrow\R$ such that
\begin{equation}\label{A1}
F\cap B_1=E_1\cap B_1 \ {\mbox{ and }} \ F\setminus B_1 = E_2\setminus B_1, 
\end{equation}
and $\overline{v}$ and $V$ satisfy the properties listed in Proposition \ref{char}, 
namely
\begin{itemize}
\item[i)] $V=U_2$ in a neighborhood of $\partial\mathcal{B}^+_{3/2}$,
\item[ii)] the trace of $V$ on $\{z=0\}$ is $\chi_F-\chi_{F^c}$,
\item[iii)] $\overline{v}=\overline{u}_2$ in a neighborhood 
of $\partial\mathcal{B}^+_{3/2}$, and $\overline{v}|_{\{z=0\}}\ge0$ a.e. 
in $F$ and $\overline{v}|_{\{z=0\}}\le0$ a.e. in $F^c$. 
\end{itemize}
Also, 
\begin{equation}\begin{split}\label{A2}
& \int_{\mathcal{B}^+_{3/2}}z^{1-2s}\left(|\nabla\overline{v}|^2-
|\nabla\overline{u}_2|^2\right)\, dX \le 
\int_{\mathcal{B}^+_{1}}z^{1-2s}\left(|\nabla\overline{u}_1|^2-
|\nabla\overline{u}_2|^2\right)\, dX \\
& \qquad \qquad +C \epsilon^{-2}\int_{\mathcal{B}^+_{1+\epsilon}
\setminus\mathcal{B}^+_{1-\epsilon}} z^{1-2s}|\overline{u}_1-\overline{u}_2|^2 \, dX + 
C\,\int_{\mathcal{B}^+_{1+\epsilon}\setminus\mathcal{B}^+_{1-\epsilon}}
z^{1-2s}\left(|\nabla\overline{u}_1|^2+
|\nabla\overline{u}_2|^2\right)\, dX
\end{split}
\end{equation}
and 
\begin{equation}\begin{split}\label{AA2}
& \int_{\mathcal{B}^+_{3/2}}z^{1-\sigma}\left(|\nabla V|^2-
|\nabla U_2|^2\right)\, dX \le 
\int_{\mathcal{B}^+_{1}}z^{1-\sigma}\left(|\nabla U_1|^2-
|\nabla U_2|^2\right)\, dX \\
& \qquad \qquad +C \int_{\mathcal{B}^+_{1+\epsilon}\setminus\mathcal{B}^+_{1-\epsilon}} z^{1-\sigma} |\nabla\phi|^2 \,|U_1-U_2|^2 \, dX + 
C\,\int_{\mathcal{B}^+_{1+\epsilon}\setminus\mathcal{B}^+_{1-\epsilon}}
z^{1-\sigma}\left(|\nabla U_1|^2+|\nabla U_2|^2\right)\, dX
\end{split}
\end{equation}
for some $C>0$.
\end{lemma}

\begin{proof}
We set
\begin{equation}\label{AA}
F:=(E_1\cap B_1)\cup (E_2\setminus B_1).
\end{equation}
With this we have established~\eqref{A1}.

Now, we define 
$$ \overline{w}_{\pm}:=\min\{\overline{u}_1^\pm, \overline{u}_2^\pm \}.$$
Let also $\eta_1,\eta_2\in C^\infty(\R^{n+1},[0,1])$, with
\begin{eqnarray*}
&& \eta_1(X)=1 \ {\mbox{ if }} \ |X|\le 1-\epsilon, \\
&& \eta_1(X)=0 \ {\mbox{ if }} \ |X|\ge 1-\epsilon/2, \\
&& \eta_2(X)=1 \ {\mbox{ if }} \ |X|\le 1+\epsilon/2, \\
&& \eta_2(X)=0 \ {\mbox{ if }} \ |X|\ge 1+\epsilon\\
{\mbox{and }}&& |\nabla\eta_1|+|\nabla\eta_2|\le \frac{8}{\eps}.
\end{eqnarray*}
We define\footnote{We put $\pm$ as a subscript (rather than a superscript) in~$\overline{v}_\pm$
and~$\overline{w}_\pm$ not to confuse in principle
the notation with the positive/negative part of
a function.}
$$ \overline{v}_\pm:= \eta_1\,\eta_2\,\overline{u}^\pm_1 +\eta_2(1-\eta_1)\,\overline{w}_\pm + (1-\eta_2)\,\overline{u}^\pm_2 $$
and $\overline{v}:=\overline{v}_+-\overline{v}_-$. 

By construction, $\overline{v}=\overline{u}_2$ near $\partial\mathcal{B}^+_{3/2}$. 

Also, if $x\in F$, then $x\in E_1\cup E_2$, and so either $\overline{u}_1(x,0)\ge0$ 
or $\overline{u}_2(x,0)\ge0$ (up to sets of zero measure), and then either 
$\overline{u}_1^-(x,0)=0$ or $\overline{u}_2^-(x,0)=0$, 
so $\overline{w}_-(x,0)=0$. This gives that for a.e. $x\in F$
$$ \overline{v}_-(x,0)=\eta_1(x,0)\,\eta_2(x,0)\,\overline{u}_1^-(x,0)+
(1-\eta_2(x,0))\,\overline{u}_2^-(x,0), $$
and so
\begin{equation}\label{5df67g8hww65}\begin{split}
&\overline{v}(x,0) \\=\,& \eta_1(x,0)\,\eta_2(x,0)\,\overline{u}_1(x,0)+
(1-\eta_2(x,0))\,\overline{u}_2(x,0)+\eta_2(x,0)(1-\eta_1(x,0))\,\overline{w}_+(x,0)\\
\ge\, &\eta_1(x,0)\,\eta_2(x,0)\,\overline{u}_1(x,0)+
(1-\eta_2(x,0))\,\overline{u}_2(x,0).
\end{split}\end{equation}
Now, for a.e. $x\in F\cap B_1=E_1\cap B_1$ we have that $\eta_2(x,0)=1$,
thus~\eqref{5df67g8hww65}
implies that
$$ \overline{v}(x,0)\ge \eta_1(x,0)\,\overline{u}_1(x,0)\ge 0.$$
Similarly, for a.e $x\in F\setminus B_1=E_2\setminus B_1$ we have that $\eta_1(x,0)=0$, 
thus~\eqref{5df67g8hww65}
gives that
$$ \overline{v}(x,0)\ge (1-\eta_2(x,0))\,\overline{u}_2(x,0)\ge 0.$$
This shows that, for a.e. $x\in F$, $\overline{v}(x,0)\ge0$. 

Conversely, if $x\in F^c$, then $x\in E_1^c\cup E_2^c$ 
and so either $\overline{u}_1(x,0)\le0$ or $\overline{u}_2(x,0)\le0$ 
(up to sets of zero measure), that is either $\overline{u}_1^+(x,0)=0$ 
or $\overline{u}_2^+(x,0)=0$, so $\overline{w}_+(x,0)=0$. 
As a consequence, for a.e. $x\in F^c$, 
$$ \overline{v}_+(x,0)= \eta_1(x,0)\,\eta_2(x,0)\, \overline{u}_1^+(x,0)
+(1-\eta_2(x,0))\,\overline{u}_2^+(x,0), $$
and so 
\begin{eqnarray*}
&&\overline{v}(x,0) \\&=& \eta_1(x,0)\,\eta_2(x,0)\, \overline{u}_1(x,0)
+(1-\eta_2(x,0))\,\overline{u}_2(x,0)-\eta_2(x,0)(1-\eta_1(x,0))\,\overline{w}_-(x,0)\\
&\le & \eta_1(x,0)\,\eta_2(x,0)\, \overline{u}_1(x,0)
+(1-\eta_2(x,0))\,\overline{u}_2(x,0).
\end{eqnarray*}
In particular, for a.e. $x\in F^c\cap B_1=E_1^c\cap B_1$, we have that 
$\eta_2(x,0)=1$, so $\overline{v}(x,0)\le\eta_1(x,0)\,\overline{u}_1(x,0)\le 0$, 
and for a.e. $x\in F\setminus B_1=E_2^c\setminus B_1$, we have that $\eta_1(x,0)=0$, 
so $\overline{v}(x,0)\le (1-\eta_2(x,0))\overline{u}_2(x,0)\le 0$. 
This shows that $\overline{v}(x,0)\le0$ for a.e. $x\in F^c$, 
thus completing the proof of iii). 

Now we prove \eqref{A2}. For this, we notice that
\begin{eqnarray*}
\nabla\overline{v}_\pm &=& \eta_1\,\eta_2\,\nabla\overline{u}_1^\pm 
+ \eta_2(1-\eta_1)\nabla\overline{w}_\pm +(1-\eta_2)\nabla\overline{u}_2^\pm \\
&&\qquad + \nabla\left(\eta_1\,\eta_2\right)\overline{u}_1^\pm 
+ \nabla \left(\eta_2(1-\eta_1)\right)\overline{w}_\pm 
+ \nabla (1-\eta_2)\overline{u}_2^\pm,
\end{eqnarray*}
so 
\begin{eqnarray*}
\nabla\overline{v} &=& \eta_1\,\eta_2\,\nabla\overline{u}_1 
+ \eta_2(1-\eta_1)\left(\nabla\overline{w}_+ -\nabla\overline{w}_-\right) 
+(1-\eta_2)\nabla\overline{u}_2 \\
&&\qquad + \nabla\left(\eta_1\,\eta_2\right)\overline{u}_1 
+ \nabla \left(\eta_2(1-\eta_1)\right)\left(\overline{w}_+-\overline{w}_-\right) 
+ \nabla (1-\eta_2)\overline{u}_2.
\end{eqnarray*}
Now we notice that 
\begin{eqnarray*}
&& \nabla\left(\eta_1\,\eta_2\right)\overline{u}_1^+ 
+ \nabla\left(\eta_2(1-\eta_1)\right)\overline{w}_+ 
+\nabla(1-\eta_2)\overline{u}_2^+\\
&=& \left(\eta_1\,\nabla\eta_2+\nabla\eta_1\,\eta_2\right)\overline{u}_1^+ 
+ \left((1-\eta_1)\nabla\eta_2-\nabla\eta_1\,\eta_2\right)\overline{w}_+
-\nabla\eta_2\,\overline{u}_2^+\\
&=& \left(\eta_1\,\overline{u}_1^+ +(1-\eta_1)\overline{w}_+-\overline{u}_2^+\right) 
\nabla\eta_2 + \left(\overline{u}_1^+ -\overline{w}_+\right)\eta_2\,\nabla\eta_1\\
&=& O\left(\epsilon^{-1}|\overline{u}_1^+-\overline{u}_2^+|\right)\chi_{\mathcal{B}_{
1+\epsilon}^+\setminus\mathcal{B}_{1-\epsilon}^+}\\
&=& O\left(\epsilon^{-1}|\overline{u}_1-\overline{u}_2|\right)\chi_{\mathcal{B}_{
1+\epsilon}^+\setminus\mathcal{B}_{1-\epsilon}^+},
\end{eqnarray*}
and similarly for the negative parts. Hence,
\begin{eqnarray*}
\nabla\overline{v} &=& \chi_{\mathcal{B}_1^+}\nabla\overline{u}_1 
+ \left(\eta_1\,\eta_2-\chi_{\mathcal{B}_1^+}\right)\nabla\overline{u}_1 
+(1-\eta_2)\nabla\overline{u}_2\\ 
&& \qquad +\eta_2(1-\eta_1)O\left(|\nabla\overline{u}_1| + |\nabla\overline{u}_2|\right) 
+ O\left(\epsilon^{-1}|\overline{u}_1-\overline{u}_2|\right)\chi_{\mathcal{B}_{
1+\epsilon}^+\setminus\mathcal{B}_{1-\epsilon}^+}.
\end{eqnarray*}
That is 
\begin{eqnarray*}
|\nabla\overline{v}|^2 &\le & \chi_{\mathcal{B}_1^+}|\nabla\overline{u}_1|^2 
+ (1-\eta_2)^2|\nabla\overline{u}_2|^2\\
&&\qquad + C\chi_{\mathcal{B}_{1+\epsilon}^+\setminus\mathcal{B}_{1-\epsilon}^+}
\left(|\nabla\overline{u}_1|^2 +|\nabla\overline{u}_2|^2\right)+C\chi_{\mathcal{B}_{
1+\epsilon}^+\setminus\mathcal{B}_{1-\epsilon}^+} \epsilon^{-2}|\overline{u}_1-\overline{u}_2|^2, 
\end{eqnarray*}
for some constant $C>0$. 
Since $\overline{v}=\overline{u}_2$ outside $\mathcal{B}_{1+\epsilon}^+$, 
we conclude that 
\begin{eqnarray*}
&& \int_{\mathcal{B}_{3/2}^+}z^{1-2s}
\left( |\nabla\overline{v}|^2-|\nabla\overline{u}_2|^2\right)\,dX \\
&&\quad =\int_{\mathcal{B}_{1+\epsilon}^+}z^{1-2s}
\left( |\nabla\overline{v}|^2-|\nabla\overline{u}_2|^2\right)\,dX \\
&&\quad \le \int_{\mathcal{B}_{1}^+}z^{1-2s}
\left( |\nabla\overline{u}_1|^2-|\nabla\overline{u}_2|^2\right)\,dX 
+ C\int_{\mathcal{B}_{1+\epsilon}^+\setminus\mathcal{B}_{1-\epsilon}^+}z^{1-2s}
\left( |\nabla\overline{u}_1|^2+|\nabla\overline{u}_2|^2\right)\,dX \\
&& \qquad \qquad +C\epsilon^{-2}\int_{\mathcal{B}_{
1+\epsilon}^+\setminus\mathcal{B}_{1-\epsilon}^+}z^{1-2s}
|\overline{u}_1-\overline{u}_2|^2\,dX,
\end{eqnarray*}
that concludes the proof of \eqref{A2}. 

Now, let~$\phi$ be as in Lemma~\ref{lemma:phi}, and set $\tilde{\chi}_E:=\chi_E-\chi_{E^c}$. 
We define~$V:=(1-\phi)U_1+\phi U_2$.
We observe that, for a.e.~$x\in\R^n$, $\phi(x,0)=\chi_{\R^n\setminus B_1}$
and therefore
\begin{equation}\label{De1}
V\big|_{\{z=0\}}=\chi_{B_1}\tilde{\chi}_{E_1}+
\chi_{\R^n\setminus B_1}\tilde{\chi}_{E_2}=\tilde{\chi}_F,
\end{equation}
where~$F$ is defined in~\eqref{AA}. This establishes ii). 

Also, $\phi=1$ outside~${\mathcal{B}}_{1+\eps}^+$ hence
\begin{equation}\label{De2}
{\mbox{$V=U_2$ outside ${\mathcal{B}}_{1+\eps}^+$,}}
\end{equation}
thus proving i).

Now we show~\eqref{AA2}. We observe that
$$ \nabla V=\nabla U_1+(U_2-U_1)\nabla \phi+\phi\nabla(U_2-U_1).$$
Therefore, by Young inequality we have 
$$|\nabla V|^2\le
C\left(|\nabla\phi|^2|U_2-U_1|^2+|\nabla U_1|^2+|\nabla U_2|^2\right), $$
for suitable~$C>0$.
Hence, integrating over~${\mathcal{B}}^+_{1+\eps}\setminus {\mathcal{B}}^+_{1-\eps}$ 
we get
\begin{equation}\label{De3}\begin{split}
&\int_{{\mathcal{B}}^+_{1+\eps}\setminus {\mathcal{B}}^+_{1-\eps}}
z^{1-\sigma} |\nabla V|^2 \,dX\\
&\qquad \le\, C \int_{{\mathcal{B}}^+_{1+\eps}\setminus {\mathcal{B}}^+_{1-\eps}}
z^{1-\sigma} \big( 
|\nabla\phi|^2|U_2-U_1|^2+|\nabla U_1|^2+|\nabla U_2|^2
\big)\,dX,
\end{split}\end{equation}
for some~$C>0$. Furthermore, $V=U_1$ in ${\mathcal{B}}^+_{1-\eps}$. Thus, using~\eqref{De2}
and~\eqref{De3} we obtain that
\begin{eqnarray*}
&& \int_{{\mathcal{B}}^+_{3/2}} z^{1-\sigma} |\nabla V|^2 \,dX \\
&=& \int_{{\mathcal{B}}^+_{1-\eps}} z^{1-\sigma} |\nabla U_1|^2 \,dX
+\int_{{\mathcal{B}}^+_{3/2}\setminus {\mathcal{B}}^+_{1+\eps}}
z^{1-\sigma} |\nabla U_2|^2 \,dX
+\int_{{\mathcal{B}}^+_{1+\eps}\setminus {\mathcal{B}}^+_{1-\eps}}
z^{1-\sigma} |\nabla V|^2 \,dX \\
&\le &
\int_{{\mathcal{B}}^+_{1}} z^{1-\sigma} |\nabla U_1|^2 \,dX
+\int_{{\mathcal{B}}^+_{3/2}}
z^{1-\sigma} |\nabla U_2|^2 \,dX
-\int_{{\mathcal{B}}^+_{1+\eps}}
z^{1-\sigma} |\nabla U_2|^2 \,dX
\\ &&\qquad+C \int_{{\mathcal{B}}^+_{1+\eps}\setminus {\mathcal{B}}^+_{1-\eps}}
z^{1-\sigma} \big(
|\nabla\phi|^2|U_2-U_1|^2+|\nabla U_1|^2+|\nabla U_2|^2
\big)\,dX.
\end{eqnarray*}
This implies~\eqref{AA2} and concludes the proof of Lemma \ref{LE conv}.
\end{proof}

\begin{lemma}\label{6df7u8gihojgfff55}
Let $E_1$, $E_2\subseteq\R^n$ and~$F:= (E_1\cap B_1)\cup (E_2\setminus B_1)$.
Then 
\begin{equation}\label{A111}
F\cap B_1=E_1\cap B_1 \ {\mbox{ and }} \ F\setminus B_1 = E_2\setminus B_1, 
\end{equation}
and 
\begin{equation}\begin{split}\label{A3}
& \Per_\sigma(F,B_{3/2})-\Per_\sigma(E_2,B_{3/2})\\
&\qquad \le \Per_\sigma(E_1,B_{1})-\Per_\sigma(E_2,B_{1}) 
+\mathcal L(B_1, (E_1\Delta E_2)\setminus B_1).
\end{split}
\end{equation}
\end{lemma}

\begin{proof} 
It is clear that~$F$ satisfies~\eqref{A111}.
Now we prove \eqref{A3}. For this, we use \eqref{A111} to see that
\begin{equation}\label{jskd88}
\Per_\sigma(F,B_{3/2})-\Per_\sigma(E_2,B_{3/2})=
\Per_\sigma(F,B_{1})-\Per_\sigma(E_2,B_{1}).
\end{equation}
Furthermore, \eqref{A111} also gives that
\begin{eqnarray*}
&& \Per_\sigma(F,B_{1})-\Per_\sigma(E_1,B_{1}) \\
&& \qquad = 
\mathcal{L}(F\cap B_1, F^c\cap B_1)
+\mathcal{L}(F\cap B_1, F^c\cap B_1^c)
+\mathcal{L}(F^c\cap B_1, F\cap B_1^c) \\
&& \qquad \qquad-\mathcal{L}(E_1\cap B_1, E_1^c\cap B_1)
-\mathcal{L}(E_1\cap B_1, E_1^c\cap B_1^c)
-\mathcal{L}(E_1^c\cap B_1, E_1\cap B_1^c) \\
&& \qquad =
\mathcal{L}(E_1\cap B_1, E_1^c\cap B_1)
+\mathcal{L}(E_1\cap B_1, E_2^c\cap B_1^c)
+\mathcal{L}(E_1^c\cap B_1, E_2\cap B_1^c) \\
&& \qquad \qquad-\mathcal{L}(E_1\cap B_1, E_1^c\cap B_1)
-\mathcal{L}(E_1\cap B_1, E_1^c\cap B_1^c)
-\mathcal{L}(E_1^c\cap B_1, E_1\cap B_1^c) \\
&& \qquad = \mathcal{L}(E_1\cap B_1, E_2^c\cap B_1^c)
-\mathcal{L}(E_1\cap B_1, E_1^c\cap B_1^c) \\
&& \qquad \qquad+ \mathcal{L}(E_1^c\cap B_1, E_2\cap B_1^c)
-\mathcal{L}(E_1^c\cap B_1, E_1\cap B_1^c) \\
&& \qquad \le \mathcal{L}(E_1\cap B_1, (E_1\setminus E_2)\cap B_1^c)
+\mathcal{L}(E_1^c\cap B_1, (E_2\setminus E_1)\cap B_1^c)
\\ && \qquad \le \mathcal{L}(B_1, (E_1\setminus E_2)\cap B_1^c)
+\mathcal{L}(B_1, (E_2\setminus E_1)\cap B_1^c)
\\ && \qquad =\mathcal{L}(B_1, (E_1\Delta E_2)\cap B_1^c)
.\end{eqnarray*}
By combining this and \eqref{jskd88}, we conclude that
\begin{eqnarray*}
&& \Per_\sigma(F,B_{3/2})-\Per_\sigma(E_2,B_{3/2}) \\
&& \qquad=\Per_\sigma(F,B_{1})-\Per_\sigma(E_2,B_{1})
+\Per_\sigma(E_1,B_{1})-\Per_\sigma(E_1,B_{1})\\
&& \qquad\le \Per_\sigma(E_1,B_{1})-\Per_\sigma(E_2,B_{1})
+\mathcal{L}(B_1, (E_1\Delta E_2)\cap B_1^c),
\end{eqnarray*}
which establishes \eqref{A3}.
\end{proof}

\section{Uniform energy bounds for minimizing pairs and proof of Theorem~\ref{ENERGY}}\label{ENERGY:sec}

Here we prove that if~$(u,E)$ is a minimizing pair in some
ball then its energy in a smaller ball is bounded uniformly,
only in dependence of a weighted $L^2$ norm of~$u$.
For this, we start with some technical observations:

\begin{lemma}\label{6ydhwiiw}
Let~$\eta\in C^\infty_0(B_1)$ and~$u:\R^n\rightarrow\R$
be a measurable function. Then
\begin{equation}\label{8ssqff56ff9}
\iint_{\R^{2n}\setminus(B_1^c)^2} |u(y)|^2\,\frac{|\eta(x)-\eta(y)|^2}{
|x-y|^{n+2s}}\,dx\,dy \le C \int_{\R^n}
\frac{|u(y)|^2}{1+|y|^{n+2s}}\,dy.
\end{equation}
Here~$C>0$ only depends on~$\|\eta\|_{C^1(\R^n)}$, $n$ and~$s$.
\end{lemma}

\begin{proof} We suppose that the right-hand side of~\eqref{8ssqff56ff9}
is finite, otherwise we are done. Then we observe that, for any~$y\in\R^n$,
\begin{equation}\label{8ssqff56ff9-1}
\begin{split}
& \int_{\R^n}\frac{|\eta(x)-\eta(y)|^2}{
|x-y|^{n+2s}}\,dx \le \int_{\R^n}\frac{\min\{
4\|\eta\|_{L^\infty(\R^n)}^2, \
\|\nabla\eta\|_{L^\infty(\R^n)}^2 |x-y|^2\}}{
|x-y|^{n+2s}}\,dx \\
&\qquad\le C\int_{\R^n}\frac{\min\{1,|z|^2\}}{|z|^{n+2s}}\,dz\le C,
\end{split}\end{equation}
for some~$C>0$ (that may be different from step to step). Similarly,
we have that
\begin{equation}\label{8ssqff56ff9-2}
\begin{split}
& \sup_{y\in B_2\setminus B_1} \int_{B_1}\frac{|\eta(x)-\eta(y)|^2}{
|x-y|^{n+2s}}\,dx \le
\sup_{y\in B_2\setminus B_1} \int_{B_1}\frac{
\|\nabla\eta\|_{L^\infty(\R^n)}^2 |x-y|^2}{|x-y|^{n+2s}}\,dx
\\ &\qquad\le C \int_{B_3} 
\frac{|z|^2}{|z|^{n+2s}}\,dz\le C.
\end{split}\end{equation}
Furthermore, if~$y\in\R^n\setminus B_2$ and~$x\in B_1$,
we have that~$|x-y|\ge|y|-|x|\ge |y|/2$, therefore
\begin{equation}\label{8ssqff56ff9-3}
\begin{split}
&\int_{B_1}\frac{|\eta(x)-\eta(y)|^2}{
|x-y|^{n+2s}}\,dx \le
4\|\eta\|_{L^\infty(\R^n)}^2
\cdot 2^{n+2s}\int_{B_1}\frac{dx}{
|y|^{n+2s}} \\ &\qquad\le \frac{C}{|y|^{n+2s}} \qquad{\mbox{ for any }}y\in\R^n\setminus B_2.
\end{split}\end{equation}
Accordingly, using \eqref{8ssqff56ff9-1}, \eqref{8ssqff56ff9-2}
and~\eqref{8ssqff56ff9-3}, we see that
\begin{eqnarray*}
&& \iint_{\R^{2n}\setminus(B_1^c)^2} |u(y)|^2\,\frac{|\eta(x)-\eta(y)|^2}{
|x-y|^{n+2s}}\,dx\,dy \\
&=&
\iint_{\R^{n}\times B_1} |u(y)|^2\,\frac{|\eta(x)-\eta(y)|^2}{
|x-y|^{n+2s}}\,dx\,dy +
\iint_{B_1\times (B_2\setminus B_1)} |u(y)|^2\,\frac{|\eta(x)-\eta(y)|^2}{
|x-y|^{n+2s}}\,dx\,dy \\ &&\qquad+
\iint_{B_1\times (\R^n\setminus B_2)} |u(y)|^2\,\frac{|\eta(x)-\eta(y)|^2}{
|x-y|^{n+2s}}\,dx\,dy \\
&\le& C\left(
\int_{B_1} |u(y)|^2\,dy
+\int_{B_2\setminus B_1} |u(y)|^2\,dy
+\int_{\R^n\setminus B_2} \frac{|u(y)|^2}{|y|^{n+2s}}\,dy
\right),
\end{eqnarray*}
that gives~\eqref{8ssqff56ff9}.
\end{proof}

\begin{corollary}\label{cor-energy-bound}
Let~$(u,E)$ be a minimizing pair in~$B_2$. Then
$$ \iint_{\R^{2n}\setminus(B_{1}^c)^2} \frac{|u(x)-u(y)|^2}{
|x-y|^{n+2s}}\,dx\,dy \le C  \int_{\R^n}
\frac{|u(y)|^2}{1+|y|^{n+2s}}\,dy,$$
for some~$C>0$ only depending on~$n$ and~$s$.
\end{corollary}

\begin{proof} Let~$\eta\in C^\infty_0(B_2)$ with~$\eta=1$
in~$B_1$. Let~$\eps\in\R$
and~$u_\eps:=(1+\eps\eta^2)u$.
We observe that the sign of~$u_\eps$ is the same as the one of~$u$,
as long as $\eps$ is sufficiently small, and so~$(u_\eps,E)$ is
an admissible competitor. Therefore~${\mathcal{F}}(u_\eps,E)-
{\mathcal{F}}(u,E)\ge0$. Dividing by~$\eps$ and taking the limit as~$\eps\to0$,
we obtain that
\begin{equation}\label{EUler}
\iint_{\R^{2n}\setminus(B_{2}^c)^2}
\frac{\big( u(x)-u(y)\big)\,\big( \eta^2(x)u(x)-\eta^2(y) u(y)\big)}{
|x-y|^{n+2s}}\,dx\,dy =0.\end{equation}
Moreover
\begin{equation}\begin{split}\label{xcvbuyq6thjcfff}
& \big( u(x)-u(y)\big)\,\big( \eta^2(x)u(x)-\eta^2(y) u(y)\big)\\
=\,& \big( u(x)-u(y)\big)\,\big( \eta^2(x)u(x)-\eta^2(x)u(y)+
\eta^2(x)u(y)-\eta^2(y) u(y)\big)\\
=\,& \eta^2(x)\big| u(x)-u(y)\big|^2 +u(y)\big( u(x)-u(y)\big)
\big( \eta(x)-\eta(y)\big) \big( \eta(x)+\eta(y)\big) \\
\ge\,& \eta^2(x)\big| u(x)-u(y)\big|^2 -\frac18 \big|\eta(x)+\eta(y)\big|^2
\big| u(x)-u(y)\big|^2 -8u^2(y) \big|\eta(x)-\eta(y)\big|^2.
\end{split}\end{equation}
Also, if we use the symmetry of the kernel,
we see that
\begin{eqnarray*}
&& \iint_{\R^{2n}\setminus(B_{2}^c)^2} \frac{
\big|\eta(x)+\eta(y)\big|^2
\big| u(x)-u(y)\big|^2}{|x-y|^{n+2s}}\,dx\,dy\\
&\le& 
2 \iint_{\R^{2n}\setminus(B_{2}^c)^2} \frac{
\big(\eta^2(x)+\eta^2(y)\big)\,
\big| u(x)-u(y)\big|^2}{|x-y|^{n+2s}}\,dx\,dy \\
&=& 4\iint_{\R^{2n}\setminus(B_{2}^c)^2} \frac{
\eta^2(x)      
\big| u(x)-u(y)\big|^2}{|x-y|^{n+2s}}\,dx\,dy.
\end{eqnarray*}
Consequently, if we integrate~\eqref{xcvbuyq6thjcfff}
and we use the latter estimate, we conclude that
\begin{eqnarray*}
&& \iint_{\R^{2n}\setminus(B_{2}^c)^2} \frac{
\big( u(x)-u(y)\big)\,\big( \eta^2(x)u(x)-\eta^2(y) u(y)\big)
}{|x-y|^{n+2s}}\,dx\,dy \\
&\ge&
\iint_{\R^{2n}\setminus(B_{2}^c)^2} \frac{
\eta^2(x)\big| u(x)-u(y)\big|^2 -\frac18 \big|\eta(x)+\eta(y)\big|^2
\big| u(x)-u(y)\big|^2 -8u^2(y) \big|\eta(x)-\eta(y)\big|^2
}{|x-y|^{n+2s}}\,dx\,dy
\\ &\ge&
\iint_{\R^{2n}\setminus(B_{2}^c)^2} \frac{
\frac12\eta^2(x)\big| u(x)-u(y)\big|^2 
-8u^2(y) \big|\eta(x)-\eta(y)\big|^2
}{|x-y|^{n+2s}}\,dx\,dy.
\end{eqnarray*}
By inserting this into~\eqref{EUler} and using that~$\eta=1$ in~$B_1$
we obtain
$$ \iint_{B_1\times\R^n} \frac{\big| u(x)-u(y)\big|^2
}{|x-y|^{n+2s}}\,dx\,dy\le
16 \iint_{\R^{2n}\setminus(B_{2}^c)^2} \frac{u^2(y) \big|\eta(x)-\eta(y)\big|^2
}{|x-y|^{n+2s}}\,dx\,dy.$$
By interchanging the variable we obtain a similar estimates
with~$\R^n\times B_1$ as domain in the left-hand side, and therefore,
by summing up
$$ \iint_{\R^{2n}\setminus(B_1^c)^2} \frac{\big| u(x)-u(y)\big|^2
}{|x-y|^{n+2s}}\,dx\,dy\le
32 \iint_{\R^{2n}\setminus(B_{2}^c)^2} \frac{u^2(y) \big|\eta(x)-\eta(y)\big|^2
}{|x-y|^{n+2s}}\,dx\,dy.$$
This and Lemma~\ref{6ydhwiiw}
imply the desired result.
\end{proof}

Now we are ready for the completion of the proof of Theorem~\ref{ENERGY}:

\begin{proof}[Proof of Theorem~\ref{ENERGY}] We use Lemma~\ref{6df7u8gihojgfff55} with~$E_1:=\R^n$ and
$E_2:=E$, and we obtain that there exists~$F$ such that~$F\setminus B_1=
E\setminus B_1$ and
\begin{equation}\label{A3-qui}
\Per_\sigma(F,B_{3/2})-\Per_\sigma(E,B_{3/2})
\le -\Per_\sigma(E,B_{1})
+\mathcal L(B_1, B_1^c).\end{equation}
In addition, we take~$\eta\in C^\infty_0 (B_{3/2},[0,1])$
with~$\eta=1$ in~$B_1$,
and we define~$v:=(1-\eta) u$. We observe that~$v=u$ outside~$B_{3/2}$.
Also, the positive set of~$u$ and~$v$ are the same and~$v=0$ in~$B_1$.
This implies that~$v\ge0$ in~$F$ and~$v\le0$ in~$F^c$, thus~$(v,F)$
is an admissible competitor in~$B_{3/2}$, which gives that
\begin{equation}\label{A3567d-qui}
\begin{split} &
\iint_{\R^{2n}\setminus(B_{3/2}^c)^2} \frac{\big| v(x)-v(y)\big|^2
}{|x-y|^{n+2s}}\,dx\,dy 
+\Per_\sigma(F,B_{3/2})\\ &\qquad-
\iint_{\R^{2n}\setminus(B_{3/2}^c)^2} \frac{\big| u(x)-u(y)\big|^2
}{|x-y|^{n+2s}}\,dx\,dy
-\Per_\sigma(E,B_{3/2})\ge0.\end{split}
\end{equation}
Now we observe that
\begin{eqnarray*}
&& \big| v(x)-v(y)\big|^2
= \big| (1-\eta(x))u(x)-(1-\eta(x))u(y)+(1-\eta(x))u(y)
-(1-\eta(y))u(y)\big|^2 \\
&&\qquad\le 2 \left( \big(1-\eta(x)\big)^2 \big|u(x)-u(y)\big|^2 +
u^2(y) \big|\eta(x)-\eta(y)\big|^2\right).
\end{eqnarray*}
Integrating this inequality and using Lemma~\ref{6ydhwiiw}
and Corollary~\ref{cor-energy-bound} we obtain that
$$ \iint_{\R^{2n}\setminus(B_{3/2}^c)^2} \frac{\big| v(x)-v(y)\big|^2
}{|x-y|^{n+2s}}\,dx\,dy\le 
C \int_{\R^n}
\frac{|u(y)|^2}{1+|y|^{n+2s}}\,dy,$$
for some~$C>0$. This, \eqref{A3-qui}
and~\eqref{A3567d-qui} imply that
\begin{eqnarray*}&&
0 \le
\iint_{\R^{2n}\setminus(B_{3/2}^c)^2} \frac{\big| v(x)-v(y)\big|^2
-\big| u(x)-u(y)\big|^2
}{|x-y|^{n+2s}}\,dx\,dy -\Per_\sigma(E,B_{1})+C
\\ &&\quad\le C \int_{\R^n}
\frac{|u(y)|^2}{1+|y|^{n+2s}}\,dy
-
\iint_{\R^{2n}\setminus(B_{3/2}^c)^2} \frac{\big| u(x)-u(y)\big|^2
}{|x-y|^{n+2s}}\,dx\,dy-\Per_\sigma(E,B_{1})+C,
\end{eqnarray*}
up to renaming~$C$, and this implies the thesis of Theorem~\ref{ENERGY}.
\end{proof}
 
\section{Convergence results and proof of Theorem~\ref{prop:conv_ext}}\label{prop:conv_ext:sec}

In the sequel, given~$\alpha\in(0,1)$
and~$r>0$, we denote by~$L^2_\alpha({\mathcal{B}}_r^+)$
the weighted Lebesgue space with respect to the weight~$z^{1-2\alpha}$,
i.e. the Lebesgue space with norm
$$ \|v\|_{L^2_\alpha({\mathcal{B}}_r^+)}:=\sqrt{
\int_{ {\mathcal{B}}_r^+ } z^{1-2\alpha} |v(X)|^2\,dX
}.$$
Now we study the convergence of the energy for a sequence
of minimizing pairs.

For this, we first obtain a useful ``integration by parts'' formula.

\begin{lemma}\label{parts} Let~$R>0$.
Let~$u:\R^n\rightarrow\R$ be such that
\begin{equation}\label{JK:POfg98}
{\mbox{$|u(x)|\le C\,|x|^\alpha$,
with~$\alpha<2s$ and~$C>0$. }}\end{equation}
Suppose that
\begin{equation}\label{4}
{\mbox{$(-\Delta)^s u=0$ in~$B_R\cap \{u\ne0\}$,}}\end{equation} and let~$\overline u$ be
as in~\eqref{ext}. Assume also that
\begin{equation}\label{JK:POfg98-BOA}
{\mbox{$\overline u$ is continuous
in~$\overline{\mathcal{B}^+_R}$}}\end{equation} and
\begin{equation}\label{44}
|\nabla \overline u|\in L^2_s({\mathcal{B}}_R^+).\end{equation}
Then
\begin{equation}\label{45}
\int_{\R^{n+1}_+} z^{1-2s} \nabla \overline u\cdot\nabla (\overline u \phi)\,dX\,=\,0\end{equation}
for any~$\phi\in C^\infty_0({\mathcal{B}}_R^+)$.
\end{lemma}

\begin{proof} We observe that the equation in~\eqref{4}
is well defined, thanks to~\eqref{JK:POfg98}.
Also, in virtue of~\eqref{JK:POfg98-BOA},
we have that~$u$ is continuous.
By Sard's Lemma, we can take a sequence of~$\eps\searrow0$
such that~$S_1:=\{\overline u=\pm\eps\}$ is a smooth set in~$\R^{n+1}_+$.
So we write~${\mathcal{B}}_R^+\cap(\partial \{|\overline u|>\eps\})
=S_1\cup S_2$, with~$S_2\subseteq\R^n\times\{0\}$
and~$|\overline u(X)|\ge\eps$ for any~$X\in S_2$. Accordingly, from~\eqref{4},
the quantity~$z^{1-2s}\partial_z\overline u$ vanishes along~$S_2$ and therefore, by the
Divergence
Theorem,
\begin{eqnarray*}
&&\int_{\{ |\overline u|>\eps \}} {\rm div}\,(z^{1-2s} \overline u\phi \nabla \overline u)\,dX=
-\int_{S_1} z^{1-2s} \overline u\,\phi \,
\partial_z\overline u \,d{\mathcal{H}}^n = \mp\eps \int_{S_1} z^{1-2s} \phi \partial_z\overline u \,d{\mathcal{H}}^n \\
&&\qquad=\mp\eps \int_{\{ |\overline u|>\eps \}} {\rm div}\,(z^{1-2s} \phi \nabla \overline u)\,dX
= \mp\eps \int_{\{ |\overline u|>\eps \}} z^{1-2s}\nabla \phi\cdot \nabla \overline u\,dX.
\end{eqnarray*}
{F}rom this and~\eqref{44} we obtain that
\begin{equation*}
\lim_{\eps\searrow0} \int_{\{ |\overline u|>\eps \}} z^{1-2s} \nabla(\overline u\phi)\cdot \nabla \overline u\,dX\,=\,
\lim_{\eps\searrow0} \int_{\{ |\overline u|>\eps \}} {\rm div}\,(z^{1-2s} \overline u\phi \nabla \overline u)\,dX\,=\,0.
\qedhere\end{equation*}
\end{proof}

The importance of the ``integration by parts'' formula in~\eqref{45}
is exploited in the next observation:

\begin{lemma}\label{parts2}
Let~$u$ and~$\overline u$ be as
in Lemma~\ref{parts}. Then
\begin{equation}\label{Der4}
\int_{ {\mathcal{B}}_R^+ }
\overline u^2\, {\rm div}\,\Big(z^{1-2s} \nabla\phi\Big)\,dX=
2\int_{ {\mathcal{B}}_R^+ } z^{1-2s}\,\phi\,|\nabla \overline u|^2\,dX.
\end{equation}
for any~$\phi\in C^\infty_0({\mathcal{B}}_R)$ that is even in~$z$.
\end{lemma}

\begin{proof} Since~$\phi$ is even in~$z$, we have that~$\partial_z\phi(x,0)=0$ and so
for any~$z>0$ we have that
$$ \partial_z\phi(x,z)=\partial_z^2\phi(x,\tilde z)\,z,$$
for some~$\tilde z\in [0,z]$ and so
$$ \lim_{z\rightarrow0} z^{1-2s} \partial_z\phi(x,z)=
\lim_{z\rightarrow0} z^{2-2s}\partial_z^2\phi(x,\tilde z)=0.$$
{F}rom this
and the Divergence
Theorem, we obtain that
\begin{equation}\label{Der3}
\int_{ {\mathcal{B}}_R^+ } {\rm div}\,\Big(z^{1-2s}
\overline u^2(X)\,\nabla\phi(X) \Big)\,dX\,=\,0.
\end{equation}
Furthermore a direct computation shows that
$$ \overline u^2\, {\rm div}\,\Big(z^{1-2s} \nabla\phi\Big)
- 2z^{1-2s}\,\phi\,|\nabla \overline u|^2
={\rm div}\,\Big(z^{1-2s}
\overline u^2\,\nabla\phi \Big)
-2z^{1-2s} \nabla \overline u\cdot\nabla(\phi\overline u).$$
Consequently, if we integrate this identity
and make use of~\eqref{45}
and~\eqref{Der3}, we obtain~\eqref{Der4}.
\end{proof}

\begin{lemma}\label{yys6}
Let~$(u_m,E_m)$ be a minimizing pair in~$B_{2}$,
and let~$\overline u_m$ be the extension of~$u_m$ as in~\eqref{ext}.
Suppose that~$u_m$ converges to some~$u$ in~$L^\infty(B_2)$ and
$\overline u_m$ converges to some~$\overline u$
in~$L^\infty({\mathcal{B}}_{2}^+)$, with~$\overline u$ continuous
in~$\overline{\R^{n+1}_+}$, being~$\overline u$ the extension of~$u$
as in~\eqref{ext}. Then
$$ \lim_{m\rightarrow+\infty}\int_{ {\mathcal{B}}_1^+ } z^{1-2s} |
\nabla \overline u_m(X)|^2\,dX \,=\,
\int_{ {\mathcal{B}}_1^+ } z^{1-2s} |\nabla \overline u(X)|^2\,dX
.$$\end{lemma}

\begin{proof} First we observe that, for
any~$\phi\in C^\infty_0({\mathcal{B}}_{2}^+)$,
we have that
\begin{equation}\label{NN1}
\int_{ \R^{n+1}_+} z^{1-2s}
\nabla \overline u_m(X) \cdot \nabla (\phi\overline u_m) (X)\,dX=0.
\end{equation}
We point out that~\eqref{NN1}
does not follow directly from Lemma~\ref{parts},
since we do not assume that~$\overline u_m$
satisfies the necessary assumptions
(on the other hand, we will use the minimality condition).
More precisely, to prove \eqref{NN1}, we 
denote by~$U_m$
the extension of~$E_m$ according to~\eqref{extE},
and we set~$\tilde u_m:=
(1+\eps \phi)\overline u_m$, with~$|\eps|<1$ to be taken sufficiently small. We have that the positive set
of~$\tilde u_m$ coincide with the one of~$\overline u_m$,
and~$(\tilde u_m, U_m)$ is a competing pair with~$(\overline u_m,U_m)$
in Proposition~\ref{char}.

As a consequence, the minimality property of~$(\overline u_m,U_m)$
gives that
\begin{eqnarray*} 0&\le& \int_{ {\mathcal{B}}_2^+ } z^{1-2s} |
\nabla \tilde u_m(X)|^2\,dX -
\int_{ {\mathcal{B}}_2^+ } z^{1-2s} |
\nabla \overline u_m(X)|^2\,dX\\& =&
2\eps
\int_{ {\mathcal{B}}_2^+ } z^{1-2s} 
\nabla \overline u_m(X) \cdot \nabla (\phi\overline u_m)(X)\,dX+o(\eps),\end{eqnarray*}
which implies~\eqref{NN1}.

Now we check that~$\overline u$ satisfies~\eqref{4}
and~\eqref{44} (this will allow us to exploit Lemma~\ref{parts2}
in the sequel).
For this, we take~$p\in B_R$, with~$u(p)\ne0$.
So there exists~$r>0$ such that $u\ne0$ in $B_r(p)$.
By the uniform convergence, for~$m$ sufficiently large we have that~$u_m\ne0$
in~$B_r(p)$. Then, by minimality and Lemma \ref{SH}, we know that~$(-\Delta)^s u_m=0$ 
in~$B_r(p)$. So, by uniform convergence, we obtain that~$(-\Delta)^s u=0$
in the weak (and so in the strong) sense in~$B_r(p)$.
This shows that~$u$ satisfies~\eqref{4}.

Moreover, given any~$\psi\in C^\infty_0({\mathcal{B}}_{2}^+)$,
if we apply~\eqref{NN1} with~$\phi:=\psi^2$ we obtain that
$$ 0=\int_{ \R^{n+1}_+} 2 z^{1-2s}
\psi \,\overline{u}_m\,\nabla \overline u_m(X) \cdot \nabla \psi \,dX
+ \int_{ \R^{n+1}_+} z^{1-2s}
\psi^2 \,|\nabla \overline u_m|^2\,dX.$$
Thus, using Young inequality, we see that
$$ \int_{ \R^{n+1}_+} z^{1-2s}
\psi^2 \,|\nabla \overline u_m|^2\,dX \le C\,
\int_{ \R^{n+1}_+} z^{1-2s}
\overline u_m^2 \,|\nabla \psi|^2\,dX,$$
for some~$C>0$. In particular, fixing~$\psi$
with~$\psi=1$ in~${\mathcal{B}}_{2-(1/10)}^+$,
we obtain that
\begin{equation}\label{serena678id} 
\int_{ {\mathcal{B}}_{2-(1/10)}^+ } z^{1-2s}
|\nabla \overline u_m|^2\,dX \le C\,
\int_{ {\mathcal{B}}_{2}^+ } z^{1-2s}
\overline u_m^2 \,dX\le 1+C\,
\int_{ {\mathcal{B}}_{2}^+ } z^{1-2s}
\overline  u^2 \,dX,\end{equation}
for large~$m$, up to renaming~$C$. 
As a consequence, we may suppose that
\begin{equation}\label{dfgijjserena}
{\mbox{$z^{(1-2s)/2} \nabla \overline u_m$
converges to some~$\Phi$ weakly in~$L^2({\mathcal{B}}_{2-(1/10)}^+)$.}}\end{equation}
We claim that
\begin{equation}\label{6789hxhhxh}
\Phi=z^{(1-2s)/2} \nabla \overline u
\end{equation}
in the weak sense. Indeed, fixed any ball~$B\subset {\mathcal{B}}_{2-(1/10)}^+$
such that~$\overline B
\subset\R^{n+1}_+$, for any~$\Psi\in C^\infty_0(B,\R^n)$ we have that
$$ \int_B \,{\rm div}\, (\overline u_m\Psi)\,dX =0,$$
due to the Divergence Theorem, therefore
$$ \int_B \nabla \overline u_m\cdot\Psi\,dX
=\int_B \,{\rm div}\, ( \overline u_m\Psi)\,dX
-\int_B \overline u_m\,{\rm div}\,\Psi\,dX = -\int_B \overline u_m\,{\rm div}\,\Psi\,dX.$$
Also, by~\eqref{dfgijjserena}, we have that
$$ \lim_{m\to+\infty}\int_B \nabla \overline u_m\cdot\Psi\,dX
=\lim_{m\to+\infty}\int_{B} z^{(1-2s)/2}
\nabla \overline u_m\cdot(z^{(2s-1)/2}\Psi)\,dX =
\int_{B} \Phi\cdot(z^{(2s-1)/2}\Psi)\,dX.$$
On the other hand, by the uniform convergence of~$\overline u_m$,
we have that
$$ \lim_{m\to+\infty}
\int_B \overline u_m\,{\rm div}\,\Psi\,dX=
\int_B \overline u\,{\rm div}\,\Psi\,dX.$$
These observations imply that
$$ \int_{B} \Phi\cdot(z^{(2s-1)/2}\Psi)\,dX=
-\int_B \overline u\,{\rm div}\,\Psi\,dX,$$
that is~$\nabla\overline u=z^{(2s-1)/2}\Phi$ in~$B$, in the weak sense,
which concludes the proof of~\eqref{6789hxhhxh}.

{F}rom~\eqref{dfgijjserena} and~\eqref{6789hxhhxh}
we conclude that~$z^{(1-2s)/2} \nabla \overline u_m$
converges to~$z^{(1-2s)/2} \nabla \overline u$
weakly in~$L^2({\mathcal{B}}_{2-(1/10)}^+)$.
As a consequence, recalling~\eqref{serena678id}, we obtain that
\begin{eqnarray*}
&& \int_{ {\mathcal{B}}_{2-(1/10)}^+ } z^{1-2s}
|\nabla \overline u|^2\,dX 
= \lim_{m\rightarrow +\infty} \int_{ {\mathcal{B}}_{2-(1/10)}^+ } z^{1-2s}\Big(
|\nabla \overline u_m|^2-
|\nabla \overline u_m-\nabla\overline u|^2\Big)\,dX \\
&&\quad\le\lim_{m\rightarrow +\infty}\int_{ {\mathcal{B}}_{2-(1/10)}^+ } z^{1-2s}
|\nabla \overline u_m|^2\,dX 
\le
1+C\,
\int_{ {\mathcal{B}}_{2}^+ } z^{1-2s}
\overline  u^2 \,dX.
\end{eqnarray*}
This proves
that $\overline u$ satisfies~\eqref{44}
(up to renaming the radius of the ball).

Therefore we are in the position to apply Lemma~\ref{parts2}, which
gives that
$$ 
\int_{ {\mathcal{B}}_2^+ }
\overline u^2\, {\rm div}\,\Big(z^{1-2s} \nabla\phi\Big)\,dX=
2\int_{ {\mathcal{B}}_2^+ } z^{1-2s}\,\phi\,|\nabla \overline u|^2\,dX,
$$
for any~$\phi\in C^\infty_0({\mathcal{B}}_2)$ that is even in~$z$.
On the other hand, \eqref{NN1} implies that
$$ \int_{ {\mathcal{B}}_2^+ }
\overline u_m^2\, {\rm div}\,\Big(z^{1-2s} \nabla\phi\Big)\,dX=
2\int_{ {\mathcal{B}}_2^+ } z^{1-2s}\,\phi\,|\nabla \overline u_m|^2\,dX, $$
for any~$\phi\in C^\infty_0({\mathcal{B}}_2)$ that is even in~$z$. 
As a consequence, if we take~$\eps>0$ and~$\phi$ with image in~$[0,1]$,
such that~$\phi=1$ in~${\mathcal{B}}_{1}$
and~$\phi=0$ outside~${\mathcal{B}}_{1+\eps}$, we obtain that
\begin{eqnarray*}
&& \lim_{m\rightarrow+\infty}
 2\int_{ {\mathcal{B}}_{1}^+ } z^{1-2s}\,|\nabla \overline u_m|^2\,dX
\le \lim_{m\rightarrow+\infty}
2\int_{ {\mathcal{B}}_2^+ } z^{1-2s}\,\phi\,|\nabla \overline u_m|^2\,dX
\\ &&\qquad=
\lim_{m\rightarrow+\infty} \int_{ {\mathcal{B}}_2^+ }
\overline u_m^2\, {\rm div}\,\Big(z^{1-2s} \nabla\phi\Big)\,dX
=\int_{ {\mathcal{B}}_2^+ }
\overline u^2\, {\rm div}\,\Big(z^{1-2s} \nabla\phi\Big)\,dX
\\ &&\qquad=2\int_{ {\mathcal{B}}_{1+\eps}^+ } z^{1-2s}
\,\phi\,|\nabla \overline u|^2\,dX
.\end{eqnarray*}
Since~$\eps$ can be taken as small as we like, we obtain
\begin{equation}\label{L.01.0}
\lim_{m\rightarrow+\infty}
2\int_{ {\mathcal{B}}_{1}^+ } z^{1-2s}\,|\nabla \overline u_m|^2\,dX
\le
2\int_{ {\mathcal{B}}_{1}^+ } z^{1-2s}
\,|\nabla \overline u|^2\,dX.
\end{equation}
On the other hand,
if we take~$\eps>0$ and~$\phi$ with image in~$[0,1]$,
such that~$\phi=1$ in~${\mathcal{B}}_{1-\eps}$
and~$\phi=0$ outside~${\mathcal{B}}_{1}$,
the argument above gives
\begin{eqnarray*}
&& \lim_{m\rightarrow+\infty}
2\int_{ {\mathcal{B}}_{1}^+ } z^{1-2s}\,|\nabla \overline u_m|^2\,dX
\ge \lim_{m\rightarrow+\infty}
2\int_{ {\mathcal{B}}_2^+ } z^{1-2s}\,\phi\,|\nabla \overline u_m|^2\,dX
\\ &&\qquad=
\lim_{m\rightarrow+\infty} \int_{ {\mathcal{B}}_2^+ }
\overline u_m^2\, {\rm div}\,\Big(z^{1-2s} \nabla\phi\Big)\,dX
=
\int_{ {\mathcal{B}}_2^+ }
\overline u^2\, {\rm div}\,\Big(z^{1-2s} \nabla\phi\Big)\,dX
\\ &&\qquad=
2\int_{ {\mathcal{B}}_2^+ } z^{1-2s}\,\phi\,|\nabla \overline u|^2\,dX
\ge2\int_{ {\mathcal{B}}_{1-\eps}^+ } z^{1-2s}
\,|\nabla \overline u|^2\,dX,
\end{eqnarray*}
and so, taking~$\eps$ as small as we like, 
$$ \lim_{m\rightarrow+\infty}
2\int_{ {\mathcal{B}}_{1}^+ } z^{1-2s}\,|\nabla \overline u_m|^2\,dX
\ge
2\int_{ {\mathcal{B}}_{1}^+ } z^{1-2s}
\,|\nabla \overline u|^2\,dX.$$
This and \eqref{L.01.0} complete the proof of Lemma \ref{yys6}.
\end{proof}

Now we can complete the proof of Theorem \ref{prop:conv_ext}:

\begin{proof}[Proof of Theorem~\ref{prop:conv_ext}] 
The first relation in~\eqref{TESI}
is a direct consequence of Lemma~\ref{yys6}.
As for the second, it follows as in Proposition~9.1 of~\cite{CRS}
(using Lemma~\ref{6df7u8gihojgfff55}
to control the fractional perimeter in Theorem~3.3 of~\cite{CRS}).
This completes the proof of~\eqref{TESI}.

Now, in order to show that~$(\overline{u},U)$ is a minimizing pair in~$\mathcal{B}^+_{1/2}$, 
we take 
\begin{equation}\label{compe}
{\mbox{a competitor $(\overline{v},V)$ for~$(\overline{u},U)$
in ${\mathcal{B}}_{1/2}^+$,}} \end{equation}
according to Proposition~\ref{char}, 
and we claim that 
\begin{equation}\begin{split}\label{se20}
&\int_{\mathcal{B}^+_{1/2}}z^{1-2s}|\nabla\overline{u}|^2\,dX +c_{n,s,\sigma}\,
\int_{\mathcal{B}^+_{1/2}}z^{1-\sigma}|\nabla U|^2\,dX\\
&\qquad \le \int_{\mathcal{B}^+_{1/2}}z^{1-2s}|\nabla\overline{v}|^2\,dX +c_{n,s,\sigma}\,\int_{\mathcal{B}^+_{1/2}}z^{1-\sigma}|\nabla V|^2\,dX.
\end{split}\end{equation}
For this, we use Lemma~\ref{LE conv} (with~$(\overline{u}_1,U_1):=(\overline{v},V)$ 
and~$(\overline{u}_2,U_2):=(\overline{u}_m,U_m)$) 
to find a pair~$(\overline{v}_m,V_m)$ 
such that~$\overline{v}_m=\overline{u}_m$ and~$V_m=U_m$ in a neighborhood of~$\partial\mathcal{B}^+_{3/2}$. Hence, $(\overline{v}_m,V_m)$ is a competitor 
for~$(\overline{u}_m,U_m)$ in~$\mathcal{B}^+_{3/2}$ according to Proposition~\ref{char}, and so 
\begin{equation}\begin{split}\label{se90}
&\int_{\mathcal{B}^+_{3/2}}z^{1-2s}|\nabla\overline{u}_m|^2\,dX +c_{n,s,\sigma}\,
\int_{\mathcal{B}^+_{3/2}}z^{1-\sigma}|\nabla U_m|^2\,dX\\
&\qquad \le \int_{\mathcal{B}^+_{3/2}}z^{1-2s}|\nabla\overline{v}_m|^2\,dX +c_{n,s,\sigma}\,\int_{\mathcal{B}^+_{3/2}}z^{1-\sigma}|\nabla V_m|^2\,dX,
\end{split}\end{equation}
since~$(\overline{u}_m,U_m)$ is a minimizing pair. 

Moreover, thanks to~\eqref{A2} and~\eqref{AA2}, we have that 
\begin{equation}\begin{split}\label{se91} 
&\int_{\mathcal{B}^+_{3/2}}z^{1-2s}|\nabla\overline{v}_m|^2\,dX 
+c_{n,s,\sigma}\, \int_{\mathcal{B}^+_{3/2}}z^{1-\sigma}|\nabla 
V_m|^2\,dX\\ &\qquad \le 
\int_{\mathcal{B}^+_{1}}z^{1-2s}|\nabla\overline{v}|^2\,dX 
+c_{n,s,\sigma}\, \int_{\mathcal{B}^+_{1}}z^{1-\sigma}|\nabla V|^2\,dX 
\\ &\qquad\qquad + 
\int_{\mathcal{B}^+_{3/2}}z^{1-2s}|\nabla\overline{u}_m|^2\,dX 
+c_{n,s,\sigma}\, \int_{\mathcal{B}^+_{3/2}}z^{1-\sigma}|\nabla 
U_m|^2\,dX\\ &\qquad\qquad 
-\int_{\mathcal{B}^+_{1}}z^{1-2s}|\nabla\overline{u}_m|^2\,dX 
-c_{n,s,\sigma}\, \int_{\mathcal{B}^+_{1}}z^{1-\sigma}|\nabla U_m|^2\,dX 
+ c_m(\epsilon), \end{split}\end{equation} 
where 
\begin{eqnarray*} 
c_m(\epsilon)&:=& 
C\,\epsilon^{-2}\int_{\mathcal{B}^+_{1+\epsilon}\setminus\mathcal{B}^+_{1-\epsilon}} 
z^{1-2s}|\overline{v}-\overline{u}_m|^2\,dX 
+C\,\int_{\mathcal{B}^+_{1+\epsilon}\setminus\mathcal{B}^+_{1-\epsilon}}z^{1-2s}\left(|\nabla\overline{v}|^2+|\nabla\overline{u}_m|^2\right)\,dX 
\\ &&\quad 
+C\,\int_{\mathcal{B}^+_{1+\epsilon}\setminus\mathcal{B}^+_{1-\epsilon}} 
z^{1-\sigma}|\nabla\phi|^2|V-U_m|^2\,dX 
+C\,\int_{\mathcal{B}^+_{1+\epsilon}\setminus\mathcal{B}^+_{1-\epsilon}}z^{1-\sigma}\left(|\nabla 
V|^2+|\nabla U_m|^2\right)\,dX, \end{eqnarray*} 
with~$\phi:=\phi_\eps$ as in Lemma~\ref{lemma:phi}.
Putting 
together~\eqref{se90} and~\eqref{se91}, we obtain that 
\begin{equation}\begin{split}\label{se92} 
&\int_{\mathcal{B}^+_{1}}z^{1-2s}|\nabla\overline{u}_m|^2\,dX 
+c_{n,s,\sigma}\, \int_{\mathcal{B}^+_{1}}z^{1-\sigma}|\nabla 
U_m|^2\,dX\\ &\qquad \le 
\int_{\mathcal{B}^+_{1}}z^{1-2s}|\nabla\overline{v}|^2\,dX 
+c_{n,s,\sigma}\, \int_{\mathcal{B}^+_{1}}z^{1-\sigma}|\nabla V|^2\,dX 
+ c_m(\epsilon).\end{split}\end{equation} Now we take the limit 
as~$m\rightarrow +\infty$ in~\eqref{se92}. Thanks to~\eqref{TESI}
(which has been already proved), we have that  
\begin{equation}\label{se50} 
{\mbox{the left-hand side converges to }} 
\int_{\mathcal{B}^+_{1}}z^{1-2s}|\nabla\overline{u}|^2\,dX 
+c_{n,s,\sigma}\, \int_{\mathcal{B}^+_{1}}z^{1-\sigma}|\nabla U|^2\,dX. 
\end{equation} 
Now we compute the limit 
of~$c_m(\epsilon)$ as $m\to+\infty$, for a fixed~$\eps>0$
(and then send~$\eps\to0$ at the end). For this, we first observe 
that~$\overline{v}=\overline{u}$ and~$V=U$ 
outside~$\mathcal{B}^+_{1/2}$, thanks to~\eqref{compe},
and so~$c_m(\epsilon)$ can be written as 
\begin{equation}\begin{split}\label{se60} c_m(\epsilon)&=\, 
C\,\epsilon^{-2}\int_{\mathcal{B}^+_{1+\epsilon}\setminus\mathcal{B}^+_{1-\epsilon}} 
z^{1-2s}|\overline{u}-\overline{u}_m|^2\,dX 
+C\,\int_{\mathcal{B}^+_{1+\epsilon}\setminus\mathcal{B}^+_{1-\epsilon}}z^{1-2s}\left(|\nabla\overline{u}|^2+|\nabla\overline{u}_m|^2\right)\,dX 
\\ &\quad 
+C\,\int_{\mathcal{B}^+_{1+\epsilon}\setminus\mathcal{B}^+_{1-\epsilon}} 
z^{1-\sigma}|\nabla\phi|^2|U-U_m|^2\,dX 
+C\,\int_{\mathcal{B}^+_{1+\epsilon}\setminus\mathcal{B}^+_{1-\epsilon}}z^{1-\sigma}\left(|\nabla 
U|^2+|\nabla U_m|^2\right)\,dX. \end{split}\end{equation} Now we claim 
that \begin{equation}\label{se70} 
\int_{\mathcal{B}^+_{1+\epsilon}\setminus\mathcal{B}^+_{1-\epsilon}} 
z^{1-2s}|\overline{u}-\overline{u}_m|^2\,dX\rightarrow 0 \ {\mbox{ and 
}} \ 
\int_{\mathcal{B}^+_{1+\epsilon}\setminus\mathcal{B}^+_{1-\epsilon}} 
z^{1-\sigma}|\nabla\phi|^2|U-U_m|^2\,dX \rightarrow 0, \end{equation} 
as~$m\rightarrow +\infty$ (for a fixed~$\eps>0$). Indeed, the first 
limit follows from~\eqref{se00}. As for the second limit, we observe 
that $|U_m|\le1$, since~$U_m$ is obtained by convolution between a 
characteristic function and the Poisson kernel which has integral~$1$. 
Hence also~$|U|\le1$ in~$\mathcal{B}^+_2$. This means that, for a 
fixed~$\eps>0$, $$ z^{1-\sigma}|\nabla\phi|^2|U-U_m|^2 \le 
4z^{1-\sigma}|\nabla\phi|^2,$$ and this function lies 
in~$L^1(\R^{n+1}_+)$, thanks to~\eqref{o-04}, applied here 
with~$\beta:=1-\sigma\in(0,1)$.
Moreover, for a fixed~$\eps>0$,
we have that~$z^{1-\sigma}|\nabla\phi|^2|U-U_m|^2\to0$
as~$m\to+\infty$. Then the second limit in~\eqref{se70} follows from
the Dominated Convergence Theorem.
This completes the proof of~\eqref{se70}.

Now, we claim that 
\begin{equation}\label{se80}
\lim_{m\rightarrow +\infty}\int_{\mathcal{B}^+_{1+\epsilon}\setminus\mathcal{B}^+_{1-\epsilon}}z^{1-2s}\left(|\nabla\overline{u}|^2+|\nabla\overline{u}_m|^2\right)\,dX\le C\, \int_{\mathcal{B}^+_{1+2\epsilon}\setminus\mathcal{B}^+_{1-2\epsilon}}z^{1-2s}|\nabla\overline{u}|^2\,dX,
\end{equation}
for a suitable~$C>0$. For this, we observe that~$\mathcal{B}^+_{1+\epsilon}\setminus\mathcal{B}^+_{1-\epsilon}$ can be covered 
by
a finite overlapping family of~$N_\epsilon$ balls of radius~$\epsilon$, say~$\mathcal{B}_\epsilon(X_j)$ with~$j=1,\ldots N_\epsilon$, and so 
\begin{eqnarray*}\int_{\mathcal{B}^+_{1+\epsilon}\setminus\mathcal{B}^+_{1-\epsilon}}
z^{1-2s}
|\nabla\overline{u}_m|^2\,dX&\le& 
\int_{\cup_{j=1}^{N_\epsilon}\mathcal{B}_\epsilon(X_j)}z^{1-2s}
|\nabla\overline{u}_m|^2\,dX 
\\&\le& C\,\sum_{j=1}^{N_\epsilon}\int_{\mathcal{B}_\epsilon(X_j)}z^{1-2s}
|\nabla\overline{u}_m|^2\,dX.\end{eqnarray*}
By using \eqref{TESI} once again, this implies that 
\begin{eqnarray*}\lim_{m\rightarrow +\infty}
\int_{\mathcal{B}^+_{1+\epsilon}\setminus\mathcal{B}^+_{1-\epsilon}}z^{1-2s}
|\nabla\overline{u}_m|^2\,dX \le C\,\sum_{j=1}^{N_\epsilon}
\int_{\mathcal{B}_\epsilon(X_j)}z^{1-2s}
|\nabla\overline{u}|^2\,dX
\\ \le C\,\int_{\mathcal{B}^+_{1+2\epsilon}\setminus\mathcal{B}^+_{1-2\epsilon}}z^{1-2s}|\nabla\overline{u}|^2\,dX,
\end{eqnarray*}
which shows~\eqref{se80} up to renaming constants. 

Analogously, one can prove that 
\begin{equation}\label{se81}
\lim_{m\rightarrow +\infty}\int_{\mathcal{B}^+_{1+\epsilon}\setminus\mathcal{B}^+_{1-\epsilon}}z^{1-\sigma}
\left(|\nabla U|^2+|\nabla U_m|^2\right)\,dX\le C\, \int_{\mathcal{B}^+_{1+2\epsilon}\setminus\mathcal{B}^+_{1-2\epsilon}}z^{1-\sigma}|\nabla U|^2\,dX,
\end{equation}
for some~$C>0$. 
Using~\eqref{se70}, \eqref{se80} and~\eqref{se81} into~\eqref{se60}, we get 
$$ \lim_{m\rightarrow +\infty}c_m(\epsilon)\le C\, \int_{\mathcal{B}^+_{1+2\epsilon}\setminus\mathcal{B}^+_{1-2\epsilon}}z^{1-2s}
|\nabla\overline{u}|^2\,dX + C\, \int_{\mathcal{B}^+_{1+2\epsilon}\setminus\mathcal{B}^+_{1-2\epsilon}}z^{1-\sigma}
|\nabla U|^2\,dX,$$
up to relabelling constants. 
Hence, sending~$\epsilon\rightarrow 0$, we obtain that 
$$ \lim_{\eps\to0} \ \lim_{m\rightarrow +\infty}c_m(\epsilon)=0 .$$ 
Using this and~\eqref{se50} into~\eqref{se92}, we obtain that 
\begin{eqnarray*}
&&\int_{\mathcal{B}^+_{1}}z^{1-2s}|\nabla\overline{u}|^2\,dX +c_{n,s,\sigma}\,
\int_{\mathcal{B}^+_{1}}z^{1-\sigma}|\nabla U|^2\,dX\\
&&\qquad \le \int_{\mathcal{B}^+_{1}}z^{1-2s}|\nabla\overline{v}|^2\,dX +c_{n,s,\sigma}\,
\int_{\mathcal{B}^+_{1}}z^{1-\sigma}|\nabla V|^2\,dX,
\end{eqnarray*}
which implies that~$(\overline{u},U)$ is a minimizing pair in~$
\mathcal{B}^+_{1/2}$, according to Proposition~\ref{char}. 
This shows~\eqref{se20} and concludes the proof of Theorem~\ref{prop:conv_ext}.
\end{proof}
 
\section{Limit of the blow-up sequences and proof of Theorem~\ref{CONI 0}}\label{CONI 0:sec}

Here we show that the blow-up limit of a minimizing pair
is a minimizing cone and prove Theorem~\ref{CONI 0}.

\begin{proof}[Proof of Theorem~\ref{CONI 0}] First of all, we notice that, for
any~$x$, $\tilde x\in\R^n$,
\begin{equation}\label{sd77ssf8www}
|u_r(x)-u_r(\tilde x)|= r^{\frac\sigma2-s}
|u(rx)-u(r\tilde x)|\le 
\|u\|_{C^{s-\frac\sigma2}(\R^n)} |x-\tilde x|^{s-\frac\sigma2}.\end{equation}
This shows that~$u_r\in C^{s-\frac\sigma2}(\R^n)$,
with norm bounded uniformly in~$r$. So, up to a subsequence,
we may assume that
\begin{equation}\label{d7hwer8ffffjw}
{\mbox{$u_r$ converges locally uniformly
to some~$u_0\in C^{s-\frac\sigma2}(\R^n)$.}}\end{equation}
We observe that $u\ge0$ in~$E$ and~$u\le0$ in~$E^c$: thus,
since~$0\in\partial E$, we have that~$u(0)=0$. As a consequence~$u_r(0)=0$
and therefore, by~\eqref{sd77ssf8www},
\begin{equation}\label{sd77ssf8www-BIS}
|u_r(x)|\le \|u\|_{C^{s-\frac\sigma2}(\R^n)}
|x|^{s-\frac\sigma2},\end{equation}
and so
\begin{equation}\label{tdyfugsddsdsqq1}
|u_0(x)|
\le \|u\|_{C^{s-\frac\sigma2}(\R^n)}|x|^{s-\frac\sigma2}.\end{equation}
Since~$(u_r,E_r)$ is a minimizing pair in~$B_{1/r}$, we can fix any~$R>0$,
take~$r\in(0,1/(4R))$ and use Theorem~\ref{ENERGY}: we obtain that
\begin{eqnarray*}&& \iint_{\R^{2n}\setminus(B_R^c)^2} \frac{\big| u_r(x)-u_r(y)\big|^2
}{|x-y|^{n+2s}}\,dx\,dy +
\Per_\sigma(E_r,B_R)\\&&\qquad \le C\,\left(1+
\int_{\R^n}
\frac{|u_r(y)|^2}{1+|y|^{n+2s}}\,dy\right) \le C,\end{eqnarray*}
for some~$C>0$, possibly different from step to step,
where~\eqref{sd77ssf8www-BIS} was used in the last passage.
In particular, we have that~$\Per_\sigma(E_r,B_R)$
is bounded uniformly in~$r$. By compactness,
this shows that, up to a subsequence,~$E_{r}$
converges in~$L^1_{\rm loc}(\R^n)$ to some~$E_0$.

Now, let~$\overline u_r$ and~$U_r$ be the extension
functions of~$u_r$ and~$E_r$, as in~\eqref{ext} and~\eqref{extE}. 
Similarly, let~$\overline u_0$ and~$U_0$ be the extension
functions of~$u_0$ and~$E_0$.

By~\eqref{sd77ssf8www-BIS}
and~\eqref{tdyfugsddsdsqq1}, if we fix~$\rho>0$ and we take~$x\in B_\rho$ and~$z\in(0,\rho)$,
we have that
$$ P_s(y,z) \ |u_r(x-y)-u_0(x-y)|\le
\frac{2c_{n,s}\|u\|_{C^{s-\frac{\sigma}{2}}(\R^n)}\rho^{2s} (\rho+|y|)^{s-\frac\sigma2}}{
|y|^{n+2s}}.$$
This implies that
$$ \int_{\R^n\setminus B_1} P_s(y,z) \ |u_r(x-y)-u_0(x-y)|\,dy <+\infty$$
and therefore for any fixed~$\eps>0$ there exists~$R:=R_{\rho,\eps}>0$
such that
$$ \int_{\R^n\setminus B_R} P_s(y,z) \ |u_r(x-y)-u_0(x-y)|\,dy \le\eps.$$
Consequently,
for any~$\rho>0$ and any~$x\in B_\rho$ and~$z\in(0,\rho)$, we have that
\begin{eqnarray*}
|\overline u_r(x,z)-\overline u_0(x,z)|&\le&
\int_{B_R} P_s(y,z) \ \|u_r-u_0\|_{L^\infty(B_{R+\rho})}\,dy
+\eps \\
&\le& \|u_r-u_0\|_{L^\infty(B_{R+\rho})}+\eps.
\end{eqnarray*}
That is
$$ \|\overline u_r-\overline u_0\|_{L^\infty(B_\rho\times(0,\rho))}
\le \|u_r-u_0\|_{L^\infty(B_{R+\rho})}+\eps$$
and therefore, by \eqref{d7hwer8ffffjw}, 
$$ \lim_{r\to0}\|\overline u_r-\overline u_0\|_{L^\infty(B_\rho\times(0,\rho))}
\le \eps.$$
Since~$\eps>0$ may be taken arbitrarily small, we infer that
$$ \lim_{r\to0}\|\overline u_r-\overline u_0\|_{L^\infty(B_\rho\times(0,\rho))}=0,$$
hence~$\overline u_r$ converges locally uniformly to~$\overline u_0$.

Moreover, as in Proposition~9.1 in~\cite{CRS}, we have that~$U_r$
converges, up to subsequence, to some~$U_0$ locally in~$L^2_{\sigma/2}$.
These observations give that~\eqref{se00}
is satisfied in this case.
Now we claim that~$\overline u_0$ is continuous on~$\overline{\R^{n+1}_+}$.
For this, we take a sequence~$(x_k,z_k)\in \R^{n+1}_+$,
with~$(x_k,z_k)\to(x,z)\in\overline{\R^{n+1}_+}$ as~$k\to+\infty$.
We have
\begin{eqnarray*} && \overline u_0(x_k,z_k)=
\int_{\R^n} P_s(y,z_k) \, u_0(x_k-y)\,dy
\\ &&\qquad= \int_{\R^n} z_k^{-n} P_s(z_k^{-1} y,1) \, u_0(x_k-y)\,dy
=\int_{\R^n} P_s(\tilde y,1) \, u_0(x_k-z_k\tilde y)\,d\tilde y.\end{eqnarray*}
Now we observe that
$$ \lim_{k\to+\infty} u_0(x_k-z_k\tilde y)=u_0(x-z\tilde y),$$
due to~\eqref{d7hwer8ffffjw}. Also, by~\eqref{tdyfugsddsdsqq1},
$$ P_s(\tilde y,1) \, |u_0(x_k-z_k\tilde y)|\le
\frac{c_{n,s}\, \|u\|_{C^{s-\frac\sigma2}(\R^n)}
(1+|x|+z|\tilde y|)^{s-\frac\sigma2}.
}{(1+|\tilde y|^2)^{\frac{n+2s}{2}}} ,$$
for~$k$ large, which is integrable in~$\tilde y\in\R^n$. Accordingly, by the
Dominated Convergence Theorem,
$$ \lim_{k\to+\infty} \overline u_0(x_k,z_k)=
\int_{\R^n} P_s(\tilde y,1) \, u_0(x-z\tilde y)\,d\tilde y=
\overline u_0(x,z),$$
that proves the continuity of~$\overline u_0$ in~$\overline{\R^{n+1}_+}$.

Therefore, we can use Theorem~\ref{prop:conv_ext}
and obtain that~$(\overline u_0,U_0)$ is a minimizing pair. Thus, by
Proposition~\ref{char}, we have that~$(u_0,E_0)$ is a minimizing pair.

It remains to show that~$(u_0,E_0)$ is homogeneous
(hence it is a minimizing cone). For this, we recall \eqref{Phi} and we use~\eqref{TESI}
to see that
$$ \lim_{r\to0} \Phi_{u_r} (t)=
\Phi_{u_0}(t).$$
This and~\eqref{p15} give that
$$ \lim_{r\to0} \Phi_{u} (rt)=
\Phi_{u_0}(t).$$
That is
$$ \Phi_{u_0}(t)=\lim_{\tau\to0} \Phi_{u} (\tau),$$
and this limit exists since~$\Phi_u$ is monotone (recall
Theorem~\ref{TH:mon}).
In particular, $\Phi_{u_0}$ is constant and so, by Theorem~\ref{TH:mon},
we have that~$(u_0, E_0)$ is homogeneous. This completes the proof of
Theorem~\ref{CONI 0}.
\end{proof}

\section{A maximum principle in unbounded domains for the fractional Laplacian and proof of Theorems~\ref{TRI}
and~\ref{MAX PLE}}\label{sec6}

The purpose of this section is to prove Theorems \ref{MAX PLE}
and \ref{TRI}.

\begin{proof}[Proof of Theorem \ref{MAX PLE}] 
First we observe that 
\begin{equation}\label{HA 1}
{\mbox{$(-\Delta)^s v^+\le0$ in the whole of~$\R^n$}}\end{equation}
in the viscosity sense. To check this, let~$\phi$ be a competing
function touching~$v^+$ from above at~$p$.

If~$v^+(p)>0$, then~$p\in D$, since~$v^+=0$ outside~$D$.
Notice also that~$\phi\ge v^+\ge v$
and~$\phi(p)=v^+(p)=v(p)$, thus~$\phi$ touches~$v$ from above at~$p$,
and therefore, by~\eqref{eq-0}, we obtain that
$(-\Delta)^s\phi(p)\le0$. 

On the other hand, if~$v^+(p)=0$, then we have that~$\phi\ge v^+\ge0$
and~$\phi(p)=v^+(p)=0$, which gives directly that
$$ \int_{\R^n}\frac{\phi(p+y)+\phi(p-y)-2\phi(p)
}{|y|^{n+2s}}\,dy \ge 0.$$
This proves~\eqref{HA 1}.

Now we show that
\begin{equation}\label{HA 2}
{\mbox{$v^+$ vanishes identically.}}
\end{equation}
Suppose not, then we can define
$$ A:=\sup_{\R^n} v^+ \in (0,+\infty).$$
So we fix any~$q=(q',q_n)\in\R^n$ such that~$v^+(q)>0$.
Notice that~$q_n>0$ since~$v^+=0$ in~$\{x_n\le0\}$.
So we can set~$r:=2q_n>0$ and~$\tilde q:=(q',\,-r/4)$, and we remark that
$$ B_{r/4}(\tilde q)\subseteq B_r(q)\cap \{x_n\le0\}.$$
Accordingly,
$$ \big| B_r(q)\cap \{v^+\le0\} \big|\ge
\big| B_r(q)\cap \{x_n\le0\} \big|\ge \big|B_{r/4}(\tilde q)\big|
\ge \delta r^n\,$$
for some universal~$\delta>0$. So we are in the position of
applying a Harnack-type inequality (see e.g. Corollary~4.5
in~\cite{SIHA}) and we conclude that~$v^+\le (1-\gamma)A$
in~$B_{r/2}(q)$, and so, in particular~$v^+(q)\le(1-\gamma)A$,
for some~$\gamma\in(0,1)$. As a consequence, since~$q$ is arbitrary,
$$ A=\sup_{q\in \{v^+>0\}} v^+(q)\le (1-\gamma)A,$$
which is a contradiction.
This proves~\eqref{HA 2} which in turn implies Theorem~\ref{MAX PLE}.
\end{proof}

As a consequence of Theorem \ref{MAX PLE}
we have the following classification result:

\begin{corollary}\label{T cone}
Let $A>0$. Let $E$ be a cone in $\R^n$ such that
\begin{equation}\label{8.8}
E \subseteq \{x_n>0\}.\end{equation}
Let $u\in C^2(E)$ and
continuous on $\overline E$, with $[u]_{C^\gamma(E)}<+\infty$ 
for some~$\gamma\in(0,s]$.
Assume that
$$ \left\{ \begin{matrix}
(-\Delta)^s u\le0 & {\mbox{ in $E$,}}\\
u\ge0 & {\mbox{ in $E$,}}\\
u\le0 & {\mbox{ in $E^c$.}}
\end{matrix}\right. $$
Then
\begin{equation}\label{9.9} u(x)\le C_A (x_n)_+^s\end{equation}
for any $x\in E\cap \{x_n\le A\}$ with
\begin{equation}\label{9.9.9}
C_A := A^{\gamma-s}\,[u]_{C^\gamma(E)}.\end{equation}
Also, if $u$ is homogeneous of degree $\alpha<s$ then $u$ vanishes identically in $E$.
\end{corollary}

\begin{proof} First
we focus on the proof of \eqref{9.9}. For this, we first observe that
\begin{equation}\label{17.7bis}
{\mbox{for every
$x\in E$,
$u(x)\le [u]_{C^\gamma(E)}x_n^\gamma$.}}\end{equation} 
Indeed,
by \eqref{8.8},
for any $x=(x',x_n)\in E$ there exists $\tau\in [0,x_n)$ such that
$y:=(x',\tau)\in\partial E$. Therefore $u(y)=0$ and so
$$ u(x)=u(x)-u(y)\le 
[u]_{C^\gamma(E)}|x-y|^\gamma = [u]_{C^\gamma(E)} |x_n-\tau|^\gamma \le [u]_{C^\gamma(E)} x_n^\gamma,$$
that establishes \eqref{17.7bis}. In particular, we have that
\begin{equation}\label{17.7}
{\mbox{for every
$x\in E\cap\{x_n\le A\}$,
$u(x)\le 
A^\gamma [u]_{C^\gamma(E)}
$.}}\end{equation} 

So we
define $v(x):=u(x)- C_A (x_n)_+^s$, with $C_A$ as in \eqref{9.9.9}.
By \cite{Dyda}, we know that $(-\Delta)^s (x_n)_+^s=0$ in $\{x_n>0\}$, therefore
$(-\Delta)^s v\le0$ in $D:=E\cap\{x_n\le A\}$.

Moreover, if $x\in E^c$ we have that $v(x)\le - C_A (x_n)_+^s\le 0$,
and if $x\in E\cap \{x_n>A\}$ we have that
\begin{eqnarray*}
v(x)&\le & [u]_{C^\gamma(E)}x_n^\gamma -C_A x_n^s=
x_n^\gamma\left([u]_{C^\gamma(E)}-C_A x_n^{s-\gamma}\right)\\
&\le & x_n^\gamma\left([u]_{C^\gamma(E)}-C_A A^{s-\gamma}\right) \le 0,
\end{eqnarray*}
thanks to \eqref{17.7bis} and~\eqref{9.9.9} (recall also that~$\gamma\le s$).
As a consequence $v(x)\le 0$ for any $x\in E^c \cup (E\cap \{x_n>A\})=
\big(E\cap (E\cap \{x_n>A\})^c\big)^c = (E\cap\{x_n\le A\})^c=D^c$.
Also $v\in L^\infty(D)$ thanks to \eqref{17.7}.
So we can apply
Theorem \ref{MAX PLE} and obtain that $v\le0$ in $D$, which is \eqref{9.9}.

Now we establish
the second claim in the statement of Corollary
\ref{T cone}. For this we suppose
in addition that
$u$ is homogeneous of degree $\alpha<s$: then, fix any $x\in E$
and any $A>x_n$. By \eqref{9.9} we have
$$ u(x)= t^{-\alpha} u(tx)\le C_A t^{s-\alpha} (x_n)_+^s$$
for any $t\in(0,1)$, hence, by taking $t\rightarrow0$ the second claim
of Corollary \ref{T cone} follows.
\end{proof}

\begin{proof}[Proof of Theorem \ref{TRI}]
We make some preliminary observations. 
First, we notice that if~$u$ vanishes identically then 
the thesis trivially follows. Therefore, we can suppose that~$u\neq 0$,
and so 
\begin{equation}\label{esiste omega}
{\mbox{there exists~$\omega\in S^{n-1}$ such that~$u(\omega)\neq 0$.}} 
\end{equation}

Now, we claim that~$s-\frac{\sigma}{2}\ge 0$. 
For this, we observe that~$u\in C^\gamma(\R^n)$, in particular it 
belongs to~$C^\gamma(B_2)$. Therefore, from Weierstra{\ss}'s theorem, 
we have that~$u$ is bounded in~$B_2$. On the other hand,~$u$ 
is homogeneous of degree~$s-\frac{\sigma}{2}$, and so 
\begin{equation}\label{ghghghghgh}
u(rx)=r^{s-\frac{\sigma}{2}}u(x) 
\end{equation}
for any~$x\in B_2$ and~$r\in(0,1]$. Since~$x,rx\in B_2$, we have that 
both~$u(x)$ and~$u(rx)$ are bounded. Therefore, sending~$r\searrow 0$ in~\eqref{ghghghghgh}, we obtain that~$s-\frac{\sigma}{2}\ge0$. 

Now, if~$s-\frac{\sigma}{2}=0$, then~$u=c$ for some constant~$c\in\R$. 
Then, the claim of the theorem easily follows: 
indeed, for instance, if the positivity set~$E$ is contained in a halfspace then~$u=c\le 0$.

Hence, from now on we assume that
\begin{equation}\label{s sigma}
s-\frac{\sigma}{2}>0. \end{equation} 
We prove that
\begin{equation}\label{uzerozero}
u(0)=0.
\end{equation} 
Indeed, since~$u$ is homogeneous of degree~$s-\frac{\sigma}{2}$, 
we have that 
$$ u(0)=r^{s-\frac{\sigma}{2}}u(0) $$ 
for any~$r>0$, which implies~\eqref{uzerozero}. 

Now, we recall that~$u\in C^\gamma(\R^n)$ and we prove that
\begin{equation}\label{gamma}
\gamma=s-\frac{\sigma}{2}. 
\end{equation} 
For this, we take~$\omega$ as in~\eqref{esiste omega} and we obtain that, for any~$r>0$,   
$$ |u(r\omega)-u(0)|\le [u]_{C^\gamma(\R^n)}r^\gamma. $$
On the other hand, 
$$ |u(r\omega)-u(0)|=|u(r\omega)|=r^{s-\frac{\sigma}{2}}|u(\omega)|,$$
thanks to~\eqref{uzerozero}. 
Therefore, 
\begin{equation}\label{gfbnjfrla}
r^{s-\frac{\sigma}{2}}|u(\omega)|\le [u]_{C^\gamma(\R^n)}r^\gamma. 
\end{equation} 
Since~$|u(\omega)|\neq0$ (recall~\eqref{esiste omega}), this implies that 
\begin{equation}\label{jkhg}
r^{s-\frac{\sigma}{2}-\gamma}\le \frac{[u]_{C^\gamma(\R^n)}}{|u(\omega)|}=:C_1, 
\end{equation}
for a suitable positive constant~$C_1$. 
Moreover,~$u$ is not identically a constant, thanks to~\eqref{s sigma}, 
and so~$[u]_{C^\gamma(\R^n)}\neq0$. Hence,~\eqref{gfbnjfrla} implies that 
\begin{equation}\label{saeqe}
r^{\gamma-s+\frac{\sigma}{2}}\ge\frac{|u(\omega)|}{[u]_{C^\gamma(\R^n)}}=:C_2, 
\end{equation}
for some constant~$C_2>0$. 
Now, if~$\gamma<s-\frac{\sigma}{2}$ then, we send~$r$ to~$+\infty$ in~\eqref{jkhg} 
and we obtain a contradiction. If~$\gamma>s-\frac{\sigma}{2}$, we send~$r\searrow0$ in~\eqref{saeqe} and we reach again a contradiction. 
This proves~\eqref{gamma}. 
    
Now, we prove the first claim in Theorem \ref{TRI}
(the proof of the second claim is similar, and then the
last claim clearly follows).
For this, we suppose, up to a rigid motion, that
$E\subseteq\{x_n>0\}$ and we show that $u\le0$. So we
assume, by contradiction, that
$$ E^+ := \{ u>0\} \ne \varnothing.$$
By construction, $u\le 0$ in $E^c\supseteq\{x_n\le0\}$,
therefore $E^+ \subseteq \{x_n>0\}$. Also, by Lemma \ref{SH},
$(-\Delta)^s u=0$ in $E^+$, and $u\le0$ outside $E^+$. 
Moreover,~$[u]_{C^{s-\frac{\sigma}{2}}(\R^n)}<+\infty$, thanks to~\eqref{gamma}.
Therefore, by the second claim in Corollary \ref{T cone},
we obtain that $u$ vanishes identically in $E^+$, hence $u^+$
is identically zero, and so $u\le0$.
\end{proof}

\section{Functions, sets and proof of Remark~\ref{R}}\label{RR}

We observe that the scaling properties in~\eqref{rescaled}
suggest that when~$s=\sigma/2$, homogeneous functions
of degree zero play a crucial role for the problem.

This may lead to the conjecture that, at least in this case,
a minimizing pair~$(u,E)$ reduces to the set~$E$ itself,
i.e.~$u=\chi_E-\chi_{E^c}$, provided that the boundary data allow such configuration 
(notice that when~$s=\sigma/2$ then~$s\in(0,1/2)$ and so the 
Gagliardo seminorm of the characteristic function
of a smooth set is finite, 
thus the energy is also well defined).

The content of Remark~\ref{R} is that this is not true.

\begin{proof}[Proof of Remark~\ref{R}]
Suppose by contradiction that~$u=\chi_E-\chi_{E^c}$, with~$E\ne\varnothing$
and~$E^c\ne\varnothing$, that is, in the measure theoretic sense,
\begin{equation}\label{6t7yccc7c77c7cff}
{\mbox{$|E|>0$ and $|E^c|>0$.}}
\end{equation}
Notice that either~$|E\cap B_1|>0$ or~$|E^c\cap B_1|>0$.
So, for concreteness, we may suppose that~$|E\cap B_1|>0$.
As a consequence, there exists~$r\in (0,1)$ such that
\begin{equation}\label{6789sssjjsgsggagga}
|E\cap B_r|>0.
\end{equation}
Let now~$R\in(r,1)$ and~$\tau\in C^\infty_0(B_R,[0,1])$, with
\begin{equation}\label{isIJJSf}
{\mbox{$\tau=1$
in~$B_r$.}}\end{equation}
For any~$t\in[0,1)$, let~$u_t(x):=(1-t\tau(x))u(x)$.

We observe that~$u_0=u$. In addition,~$u_t=u$ outside~$B_1$.
Also~$1-t\tau(x)\ge 1-t>0$, hence the sign of~$u$ is the same
as the one of~$u_t$. As a consequence,
the pair~$(u_t,E)$ is admissible,
hence~${\mathcal{F}}(u,E)\le{\mathcal{F}}(u_t,E)$ by minimality.
Accordingly
\begin{equation} \label{6789000}
\begin{split}
0\,&\le {\mathcal{F}}(u_t,E)-{\mathcal{F}}(u,E)\\
&=
\iint_{\R^{2n}\setminus
(B_1^c)^2}\frac{|u_t(x)-u_t(y)|^2-|u(x)-u(y)|^2}{|x-y|^{n+2s}}\, dx\,dy
\\&=
\iint_{\R^{2n}}\frac{|u_t(x)-u_t(y)|^2-|u(x)-u(y)|^2}{|x-y|^{n+2s}}\, dx\,dy   
.\end{split}\end{equation}
Notice that
\begin{eqnarray*}
&& |u_t(x)-u_t(y)|^2\\
&=& (1-t\tau(x))^2 u^2(x)+(1-t\tau(y))^2 u^2(y)
-2(1-t\tau(x))\,(1-t\tau(y))\,u(x)\,u(y) \\
&=& |u(x)-u(y)|^2 \\
&& +2t\Big[ -\tau(x) u^2(x)-\tau(y) u^2(y)
+(\tau(x)+\tau(y))\,u(x)\,u(y)\Big] \\
&& +t^2\Big|
\tau(x)\,u(x)-\tau(x)\,u(x)\Big|^2.
\end{eqnarray*}
By inserting this into~\eqref{6789000} and dividing by~$2t$
we thus obtain
$$ 0\le
\iint_{\R^{2n}}\frac{
-\tau(x) u^2(x)-\tau(y) u^2(y)
+(\tau(x)+\tau(y))\,u(x)\,u(y)}{|x-y|^{n+2s}}\, dx\,dy+\Xi t,$$
for some~$\Xi\in\R$ depending on~$u$ and~$\tau$ but independent
of~$t$. Hence we may send~$t\searrow0$
and we conclude that
\begin{equation}\label{fghjscs77sss1}
0 \le \iint_{\R^{2n}}\frac{
-\tau(x) u^2(x)-\tau(y) u^2(y)
+(\tau(x)+\tau(y))\,u(x)\,u(y)}{|x-y|^{n+2s}}\, dx\,dy 
.\end{equation}
Now, if either~$(x,y)\in E\times E$ or~$(x,y)\in E^c\times E^c$
we have that~$u^2(x)=u^2(y)=u(x)u(y)=1$ hence the integrand
in~\eqref{fghjscs77sss1} vanishes.
Hence, since the role of~$x$ and~$y$ is symmetric,
we obtain from~\eqref{fghjscs77sss1} that
\begin{eqnarray*}
0 &\le& 2\iint_{E\times E^c}\frac{
-\tau(x) u^2(x)-\tau(y) u^2(y)
+(\tau(x)+\tau(y))\,u(x)\,u(y)}{|x-y|^{n+2s}}\, dx\,dy
\\
&=&2\iint_{E\times E^c}\frac{
-\tau(x) -\tau(y) 
-(\tau(x)+\tau(y))}{|x-y|^{n+2s}}\, dx\,dy
\\ &=&-4\iint_{E\times E^c}\frac{
\tau(x)+\tau(y)}{|x-y|^{n+2s}}\, dx\,dy
.\end{eqnarray*}
Since the integrand above is nonnegative, recalling~\eqref{6t7yccc7c77c7cff}
we infer that~$\tau(x)+\tau(y)=0$ for a.e.~$(x,y)\in E\times E^c$.

As a consequence, $\tau(x)=0$ for a.e.~$x\in E$,
and so, in particular, for a.e.~$x\in E\cap B_r$.
This set has indeed positive measure, thanks to~\eqref{6789sssjjsgsggagga},
hence we get that there exists~$p\in E\cap B_r$ such that~$\tau(p)=0$.
But this is in contradiction with~\eqref{isIJJSf}
and thus it proves Remark~\ref{R}.
\end{proof}

\section{Removable singularities and proof of Remark~\ref{Remo}}\label{ydd88syhh}

In this section we give the simple proof of Remark~\ref{Remo}.
As a matter of
fact, we stress that Remark~\ref{Remo} only aims at pointing out
the possible development of plateau in a simple, concrete
example, using as little technology as possible
(more general results may be obtained
by capacity considerations, and with the use
of the fundamental solution of the fractional Laplacian
when~$n\ge2s$).

\begin{proof}[Proof of Remark~\ref{Remo}]
Assume that
\begin{equation}\label{-91}
\{u=0\}\cap (-1,1)\subseteq \{p_1,\dots,p_N\}.\end{equation}
We show that~$(-\Delta)^{1/2} u=0$ in~$(-1,1)$.

For this, we take~$\bar u$ to be the harmonic extension of~$u$
in~$\R^2_+$. Namely~$\bar u$ has finite~$H^1(\R^2_+)$-seminorm
and satisfies
\begin{equation}\label{7dddj2j2a}
\left\{\begin{matrix} \Delta \bar u (x,y)=0 & {\mbox{for any $x\in\R$
and any $y>0$,}}\\
\bar u(x,0)=u(x) & {\mbox{for any $x\in\R$.}}
\end{matrix}\right.\end{equation}
By Lemma~\ref{SH} and~\eqref{-91} we have that~$(-\Delta)^{1/2}u(x)=0$
for any~$x\in\R\setminus \{p_1,\dots,p_N\}$, hence
\begin{equation}\label{-92}
\partial_y \bar u(x,0)=0 {\mbox{ for any }}
x\in\R\setminus \{p_1,\dots,p_N\}.
\end{equation}
Now we take the even symmetric extension of~$\bar u$, that is, we define
\begin{equation*}
u^*(x,y):=
\left\{\begin{matrix} \bar u(x,y)=0 & {\mbox{for any $x\in\R$
and any $y\ge0$,}}\\
\bar u(x,-y)=0 & {\mbox{for any $x\in\R$
and any $y<0$.}} \end{matrix}\right.\end{equation*}
We observe that
\begin{equation*}
\Delta u^*=0 {\mbox{ for any~$(x,y)\in\R^2\setminus\{
(p_1,0),\dots,(p_N,0)
\}$.}}
\end{equation*}
Therefore, by the
removal of singularities result for harmonic functions, we conclude that~$
\Delta u^*=0$ in the whole of~$\R^2$ and therefore~$\partial_y u^*$
is continuous also in the vicinity of~$(p_1,0),\dots,(p_N,0)$.
This implies that~$\partial_y \bar u(x,0)=0$ 
for any~$x\in(-1,1)$,
which means~$(-\Delta)^{1/2} u=0$ in~$(-1,1)$.
\end{proof}

\appendix

\section{Regularity of cones in the plane
and proof of Theorem~\ref{prop:reg}}\label{temp}

This section is devoted to the regularity of the two-dimensional cones. Namely,
in order to prove Theorem~\ref{prop:reg}, we follow the methods introduced 
in~\cite{SV1, SV2} to prove the regularity of~$\sigma$-minimal surfaces 
and used in~\cite{CSV} to obtain the regularity of the minimizers 
of the functional~\eqref{caso1}. 

We first introduce some notations. 
We define, for any~$r>0$, 
\begin{equation}\label{math}
\mathcal E_r(\overline{v},V):=\int_{\mathcal B_r^+}z^{1-2s}\,|\nabla\overline{v}|^2\, dX 
+ c_{n,s,\sigma}\int_{\mathcal B_r^+}z^{1-\sigma}\, |\nabla V|^2\,dX. 
\end{equation}
We consider a cutoff function~$\varphi\in C^{\infty}(\R)$ such that  
$$ \varphi=1 {\mbox{ in }}[-1/2,1/2] {\mbox{ and }} \varphi=0 {\mbox{ outside }} (-3/4,3/4), $$ 
and, for any~$R>0$, we introduce the following diffeomorphism in~$\R^{n+1}_+$, defined for every~$X\in\R^{n+1}_+$ as 
\begin{equation}\label{diffeo}
X\mapsto Y:=X+\varphi\left(\frac{|X|}{R}\right)\, e_1. 
\end{equation}
Then, we define 
$$ \overline{u}^+_R(Y):=\overline{u}(X) \quad  {\mbox{ and }}\quad  U^+_R(Y):=U(X). $$ 
We may also define~$\overline{u}^-_R$ and~$U^-_R$ by simply changing~$e_1$ 
into~$-e_1$ in~\eqref{diffeo}. 

The argument that we perform is similar to the one of Proposition~6.2 
in~\cite{CSV}. The main difference here is that the two terms involved in 
the functional~\eqref{math} are defined in the extension, and therefore 
we have to consider domain variations in~$\R^{n+1}_+$ both 
for~$\overline{u}$ and for~$U$.

First we prove an estimate for the second variation of 
the energy~$\mathcal E_R$.

\begin{lemma}\label{lemma:est}
Let~$(\overline{u},U)$ be a minimizer of~$\mathcal E_R$. Suppose 
that~$\overline{u}$ and~$U$ are homogeneous of degree~$s-\frac{\sigma}{2}$ 
and~$0$, respectively. Then, there exists a constant~$C>0$ independent of~$R$ 
such that 
$$ \mathcal E_R(\overline{u}^+_R,U^+_R)+ \mathcal E_R(\overline{u}^-_R,U^-_R)- 
2\mathcal E_R(\overline{u},U)\le C\, R^{n-2-\sigma}. $$
\end{lemma}

\begin{proof} 
By direct computations (see formula~(11) in~\cite{SV1}), one can prove that 
\begin{eqnarray*}
z^{1-2s}\left(|\nabla\overline{u}^+_R|^2+|\nabla\overline{u}^-_R|^2\right)\,dY &=& 2\,z^{1-2s}\left(1+O(1/R^2)\,\chi_{\mathcal B^+_R\setminus\mathcal B^+_{R/2}}\right) |\nabla\overline{u}|^2\, dX, \\
z^{1-\sigma}\left(|\nabla U^+_R|^2+|\nabla U^-_R|^2\right)\,dY &=& 2\,z^{1-\sigma}\left(1+O(1/R^2)\,\chi_{\mathcal B^+_R\setminus\mathcal B^+_{R/2}}\right) |\nabla U|^2\, dX.
\end{eqnarray*}
Therefore 
\begin{eqnarray*}
&&\int_{\mathcal B^+_R}z^{1-2s}\left(|\nabla\overline{u}^+_R|^2+|\nabla\overline{u}^-_R|^2\right)\,dY -2\int_{\mathcal B^+_R}z^{1-2s}\, |\nabla\overline{u}|^2\, dX \\
&&\qquad \le C\, R^{-2}\int_{\mathcal B^+_R\setminus\mathcal B^+_{R/2}} z^{1-2s}\,|\nabla\overline{u}|^2\, dX.
\end{eqnarray*}
Now, since~$\overline{u}$ is homogeneous of degree~$s-\frac{\sigma}{2}$, 
we have that~$z^{1-2s}\,|\nabla\overline{u}|^2$ is homogeneous of degree~$-1-\sigma$, and so 
\begin{equation}\label{111}
\int_{\mathcal B^+_R}z^{1-2s}\left(|\nabla\overline{u}^+_R|^2+|\nabla\overline{u}^-_R|^2\right)\,dY -2\int_{\mathcal B^+_R}z^{1-2s}\, |\nabla\overline{u}|^2\, dX \le C\, R^{-2}\cdot R^{n-\sigma}. 
\end{equation} 
Similarly, $U$ is homogeneous of degree~$0$, and therefore~$z^{1-\sigma}\,|\nabla U|^2$ is homogeneous of degree~$-1-\sigma$. Hence
\begin{eqnarray*}
&& c_{n,s,\sigma}\int_{\mathcal B^+_R}z^{1-\sigma}\left(|\nabla U^+_R|^2+|\nabla U^-_R|^2\right)\,dY -2c_{n,s,\sigma}\int_{\mathcal B^+_R}z^{1-\sigma}\,|\nabla U|^2\, dX\\
&&\qquad \le C\, R^{-2}\int_{\mathcal B^+_R\setminus\mathcal B^+_{R/2}} z^{1-\sigma}\,|\nabla U|^2\, dX \le C R^{-2}\cdot R^{n-\sigma}. 
\end{eqnarray*} 
By summing up this and~\eqref{111}, we obtain the thesis (recall~\eqref{math}).
\end{proof}

\begin{corollary}
Let~$(\overline{u},U)$ be a minimizer of~$\mathcal E_R$. Suppose 
that~$\overline{u}$ and~$U$ are homogeneous of degree~$s-\frac{\sigma}{2}$ 
and~$0$, respectively. Then, there exists a constant~$C>0$ independent of~$R$ 
such that 
$$ \mathcal E_R(\overline{u}^+_R,U^+_R)\le 
\mathcal E_R(\overline{u},U)+ C\, R^{n-2-\sigma}. $$

In particular, if~$n=2$, we have 
\begin{equation}\label{aggiunto}
\mathcal E_R(\overline{u}^+_R,U^+_R)\le 
\mathcal E_R(\overline{u},U)+ \frac{C}{R^{\sigma}}. 
\end{equation}
\end{corollary}

\begin{proof}
Thanks to Proposition~\ref{char}, the minimality of~$(\overline{u},U)$ gives 
$$ \mathcal E_R(\overline{u},U) \le \mathcal E_R(\overline{u}^-_R,U^-_R). $$
From this and Lemma~\ref{lemma:est} we get the desired claim. 
\end{proof}

\begin{proof}[Proof of Theorem~\ref{prop:reg}] 
We follow the line of the proof of Proposition~6.2 in~\cite{CSV}. 
For the sake of completeness we repeat the proof here.

We suppose that~$n=2$ and we argue by contradiction, 
assuming that~$E$ is not a halfplane. 
Hence, we can find a point~$p\in B_M$, for some~$M>0$, say on the~$e_2$-axis, 
such that~$p$ lies in the interior of~$E$ but~$p+e_1$ and~$p-e_1$ lie in~$E^c$. 
Therefore, recalling the notation introduced at the beginning of this section, 
we have that, for~$R>4M$, 
\begin{equation}\begin{split}\label{not}
& \overline{u}^+_R(X)=\overline{u}(X-e_1)\quad  {\mbox{ for any }}X\in\mathcal B^+_{2M}, \\
& U^+_R(X)=U(X-e_1) \quad {\mbox{ for any }}X\in\mathcal B^+_{2M}, \\
&\overline{u}^+_R(X)=\overline{u}(X) \quad {\mbox{ for any }}X\in\R^3_+\setminus\mathcal B^+_{R}, \\
& U^+_R(X)=U(X) \quad {\mbox{ for any }}X\in\R^3_+\setminus\mathcal B^+_{R}.
\end{split}\end{equation} 
Now, we define
\begin{eqnarray*}
&& \overline{v}_R(X):=\min\{\overline{u}(X),\overline{u}^+_R(X)\}, \quad \overline{w}_R(X):=\max\{\overline{u}(X),\overline{u}^+_R(X)\},\\
&& V_R(X):=\min\{U(X),U^+_R(X)\} \quad {\mbox{ and }} \quad W_R(X):=\max\{U(X),U^+_R(X)\}. 
\end{eqnarray*}
Moreover, we set~$P:=(p,0)\in\R^3$. 
We claim that 
\begin{eqnarray}
&& U^+_R<W_R=U {\mbox{ in a neighborhood of~$P$}}\label{prima}\\
{\mbox{and }}&& U<W_R=U^+_R {\mbox{ in a neighborhood of~$P+e_1$.}}\label{seconda}
\end{eqnarray}
Indeed, by~\eqref{not}
\begin{eqnarray*}
&& U^+_R(P)=U(P-e_1)=\left(\chi_E-\chi_{E^c}\right)(p-e_1)=-1, \\
&& U(P)=\left(\chi_E-\chi_{E^c}\right)(p)=1, \\
&& U^+_R(P+e_1)=U(P)=1 \\
{\mbox{and }}&& U(P+e_1)=\left(\chi_E-\chi_{E^c}\right)(p+e_1)=-1. 
\end{eqnarray*}
Then, the claim follows from the continuity of the functions~$U$ and~$U^+_R$ 
at~$P$ and~$P+e_1$. 

By Proposition~\ref{char}, the minimality of~$(\overline{u},U)$ gives 
$$ \mathcal E_R(\overline{u},U)\le \mathcal E_R(\overline{v}_R,V_R). $$ 
Moreover, we have that 
$$  \mathcal E_R(\overline{v}_R,V_R)+\mathcal E_R(\overline{w}_R,W_R)= \mathcal E_R(\overline{u},U)+\mathcal E_R(\overline{u}^+_R,U^+_R). $$
Therefore
\begin{equation}\label{222}
\mathcal E_R(\overline{w}_R,W_R)\le \mathcal E_R(\overline{u}^+_R,U^+_R).
\end{equation}

Now, we prove that~$(\overline{w}_R,W_R)$ is not a minimizer for~$\mathcal E_{2M}$ 
with respect to compact perturbations 
in~$\mathcal B^+_{2M}\times\mathcal B^+_{2M}$. 
Indeed, if~$(\overline{w}_R,W_R)$ was a minimizer, 
then~$W_R$ would be a minimizer for the~$\sigma$-perimeter, 
thanks to Proposition~\ref{char} 
(one can fix~$\overline{w}_R$ and perturb only~$W_R$, and notice 
that compact perturbations inside~$\mathcal B^+_{2M}\times\mathcal B^+_{2M}$ 
do not touch the trace). 
On the other hand, from the definition of~$W_R$ we have that~$U\le W_R$. 
Moreover,~$U$ and~$W_R$ satisfy the same equation in~$\mathcal B^+_{2M}\times\mathcal B^+_{2M}$. 
Hence, \eqref{prima} and the strong maximum principle imply that~$U=W_R$ 
in~$\mathcal B^+_{2M}$, which is a contradiction to~\eqref{seconda}.  
Therefore, there exists~$\delta>0$ and a competitor~$(\overline{u}_*,U_*)$ 
that coincides with~$(\overline{w}_R,W_R)$ outside~$\mathcal B^+_{2M}\times\mathcal B^+_{2M}$ (actually we take~$\overline{u}_*=\overline{w}_R$) 
and such that 
$$ \mathcal E_{2M}(\overline{u}_*,U_*)+\delta\le\mathcal E_{2M}(\overline{w}_R,W_R). $$ 
Notice that~$\delta$ does not depend on~$R$, since~$(\overline{w}_R,W_R)$ 
does not depend on~$R$ in~$\mathcal B^+_{2M}\times\mathcal B^+_{2M}$, thanks 
to~\eqref{not}. 
Since~$(\overline{u}_*,U_*)$ agrees with~$(\overline{w}_R,W_R)$ 
outside~$\mathcal B^+_{2M}\times\mathcal B^+_{2M}$, we conclude that 
$$ \mathcal E_R(\overline{u}_*,U_*)+\delta\le\mathcal E_R(\overline{w}_R,W_R). $$ 
From this, \eqref{aggiunto} and~\eqref{222} we obtain 
$$ \mathcal E_R(\overline{u}_*,U_*)+\delta \le \mathcal E_R(\overline{w}_R,W_R) \le \mathcal E_R(\overline{u}^+_R,U^+_R) \le \mathcal E_R(\overline{u},U)+C\, R^{-\sigma}. $$ 
Therefore, if~$R$ is large enough, we have that 
$$ \mathcal E_R(\overline{u}_*,U_*)<\mathcal E_R(\overline{u},U), $$ 
and this is a contradiction to the minimality of~$(\overline{u},U)$. 
\end{proof}

\end{document}